\documentclass[11pt]{amsart}

\usepackage{url}
\usepackage{color}
\definecolor{darkgreen}{rgb}{0,0.5,0}
\usepackage[
        colorlinks, citecolor=darkgreen,
        backref,
        pdfauthor={F. Bianchi, E. Kaya, J.S. M\"uller},
        pdftitle={Extended Coleman--Gross Heights and $p$-adic
        N\'eron Functions on Jacobians of Genus $2$ Curves}
]{hyperref}

\usepackage[utf8]{inputenc}
\usepackage[all]{xy} 		
\usepackage{tabularx}
\usepackage{enumitem}
\usepackage{colonequals,amsmath,amssymb,amsthm,mathtools}
\usepackage{times}
\usepackage[ruled,vlined]{algorithm2e}
\usepackage{float}
\usepackage{tikz,,tikz-cd,graphicx}
\graphicspath{ {./images/} }
\usepackage{hyperref}
\usepackage{comment}

\usepackage{thmtools}
\usepackage{thm-restate}
\theoremstyle{plain}
\newtheorem{thm}{Theorem}[section]
\newtheorem{cor}[thm]{Corollary}
\newtheorem{lem}[thm]{Lemma}
\newtheorem{prop}[thm]{Proposition}
\newtheorem{conj}[thm]{Conjecture}

\theoremstyle{definition}
\newtheorem{defn}[thm]{Definition}

\newtheorem{ass}[thm]{Assumption}

\theoremstyle{remark}
\newtheorem{rem}[thm]{Remark}

\def\presuper#1#2%
  {\mathop{}%
   \mathopen{\vphantom{#2}}^{#1}%
   \kern-12\scriptspace%
   #2}

\newcommand{\V}{{\operatorname{Vol}}}
\newcommand\Vint{\presuper\V\int}
\newcommand{\lV}{{\operatorname{Vol\ }}}
\newcommand\lVint{\presuper\lV\int}

\newcommand{\delbr}{\bar{\partial}}
\newcommand{\frka}{\mathfrak{a}}
\newcommand{\dR}{{\operatorname{dR}}}
\newcommand{\supp}{{\operatorname{supp}}}
\newcommand{\dd}{d}
\newcommand{\an}{{\operatorname{an}}}
\newcommand{\CG}{{\operatorname{CG}}}
\newcommand{\MT}{{\operatorname{MT}}}

\renewcommand{\mod}{{\operatorname{mod}}}
\newcommand{\hdr}{H_{\textup{dR}}}
\newcommand{\ocol}{\Omega_V}
\newcommand{\ocola}{\Omega_{V,1}}

\newcommand{\acol}{\O_V}

\newcommand{\Res}{{\operatorname{Res}}}

\newcommand{\Div}{{\operatorname{Div}}}
\newcommand{\ord}{{\operatorname{ord}}}
\renewcommand{\div}{{\operatorname{div}}}
\newcommand{\Pic}{{\operatorname{Pic}}}

\newcommand{\tr}{{\operatorname{tr}}}
\newcommand{\id}{{\operatorname{id}}}
\newcommand{\im}{{\operatorname{im}}}

\newcommand{\bi}{\mathrm{bi}}
\newcommand{\p}{\mathfrak{p}}
\newcommand{\q}{\mathfrak{q}}

\newcommand*\Z{\mathbb{Z}}
\newcommand*\Q{\mathbb{Q}}
\newcommand*\R{\mathbb{R}}
\newcommand*\C{\mathbb{C}}
\newcommand*\G{\mathbb{G}}
\newcommand*\F{\mathbb{F}}
\renewcommand*\P{\mathbb{P}}

\newcommand*\Qp{\Q_p}
\newcommand*\Cp{\C_p}

\renewcommand{\O}{\mathcal{O}}

\newcommand*\cD{\mathcal{D}}
\newcommand*\cE{\mathcal{E}}
\newcommand*\cC{\mathcal{C}}
\newcommand*\cP{\mathcal{P}}
\newcommand*\cL{\mathcal{L}}
\newcommand*\cM{\mathcal{M}}

\numberwithin{equation}{section}

\definecolor{lightyellow}{RGB}{255, 255, 197}

\title[Coleman--Gross and N\'eron functions in genus $2$]{
Coleman--Gross Heights and $p$-adic N\'eron
Functions on Jacobians of Genus $2$ Curves}

\subjclass[2020]{Primary: 11G50, 14G40. Secondary: 14H42, 14K25, 11S40, 11S80.}

\author[Bianchi]{Francesca Bianchi}
\address{\hspace{-.2in} F. Bianchi}
\email{francesca.bianchi.maths@gmail.com}

\author[Kaya]{Enis Kaya}
\address{\hspace{-.2in} E. Kaya, Department of Mathematics, Bilkent University, 06800 Ankara, Turkey}
\email{enis.kaya@bilkent.edu.tr}

\author[M\"uller]{J. Steffen M\"uller}
\address{\hspace{-.2in} J. S. M\"uller,  Bernoulli Institute,
Rijksuniversiteit Groningen,  Nijenborgh 9,  9747 AG Groningen, The Netherlands}
\email{steffen.muller@rug.nl}

\date{\today}

\begin{document}

\maketitle

\begin{abstract}
We develop a theory of $p$-adic N\'eron functions on abelian varieties,
  depending on various auxiliary choices, and show that the global $p$-adic
  height functions constructed by Mazur and Tate  can be decomposed into a
  sum of $p$-adic N\'eron functions if the same auxiliary  choices are
  made. We also consider a decomposition of the $p$-adic height constructed
  by Coleman and Gross for good reduction, and extended to arbitrary
  reduction by Colmez and Besser, into a sum of certain local height functions for Jacobians of odd degree genus~$2$ curves. We show that this local height function is equal to the $p$-adic N\'eron function with the same auxiliary choices, regardless of the reduction type of the curve. This extends work of Balakrishnan and Besser for elliptic curves. When the curve has semistable reduction and the reduction of the Jacobian is ordinary, we also describe the $p$-adic N\'eron function that arises from the canonical Mazur--Tate splitting explicitly in terms of a generalisation of the $p$-adic sigma function constructed by Blakestad. 
\end{abstract}

\setcounter{tocdepth}{1}
\tableofcontents

\section{Introduction}\label{S:intro}

Various authors have described constructions of
$p$-adic analogues of the real-valued global N\'eron--Tate height pairing on an
abelian variety $A$ defined over a number field $F$; see, for example,
\cite{Sch82}, \cite{MT83}, \cite{CG89}, \cite{Nek93}, \cite{Zar90}.
Similar to the
N\'eron--Tate height pairing, these $p$-adic height pairings are of
great arithmetic importance; for instance, a $p$-adic version of the
N\'eron--Tate regulator appears in the $p$-adic analogue of the
Birch and Swinnerton-Dyer conjecture for elliptic curves of
Mazur--Tate--Teitelbaum~\cite{MTT86}, partially generalised by
Balakrishnan, Stein and the third author in~\cite{BMS16}. Recently,
explicit methods for $p$-adic heights were used to great success in the quadratic Chabauty method for the computation of rational points on certain
curves, developed by Balakrishnan--Dogra in~\cite{BD18} and applied, for
instance, in~\cite{BDMTV19, BDMTV23}. 

It is a natural question how the different constructions of $p$-adic heights are related. Our main result is a direct local comparison
between two constructions in the case of a Jacobian $J$ of a genus $2$ curve
$C\colon y^2=b(x)$, where $b(x)\in \O_F[x]$ has degree 5 and no repeated roots and $\O_F$ is the
maximal order of a number field $F$. 
More precisely, we compare local height functions appearing in suitable
decompositions of the global height functions of 
Coleman--Gross~\cite{CG89} and Mazur--Tate~\cite{MT83}. This is the genus~$2$
analogue of a
result of Balakrishnan--Besser~\cite{balakrishnanbesser2015} for elliptic
curves.

We first discuss
the construction of Coleman--Gross (see Section~\ref{S:CG}).  Their work in~\cite{CG89} assumes
good reduction, but it was later extended to arbitrary reduction by
Colmez~\cite{colmez1998integration} and by Besser~\cite{BesserPairing}, so
that our results do not assume good reduction. 
Coleman and Gross construct local height pairings between divisors of
degree~0 on $C$ with disjoint support, one for each finite place $v$ of
$F$ (see~\S\ref{subsec:disjoint}). For places above $p$, the local pairing
is defined in terms of Vologodsky integration \cite{vologodsky2003hodge}, which we review in
Section~\ref{S:Vol}. The sum of these pairings 
depends only on the linear equivalence
classes of the divisors; one obtains a bilinear $\Q_p$-valued 
pairing on $J(F)$ and hence a quadratic form $h^{\CG}\colon J(F)\to \Q_p$.
Following Balakrishnan and Besser~\cite{balakrishnanbesser2015}, one may extend the local pairings to
divisors with common support (see~\S\ref{subsec:local-arb}). Using this, we
define in Section~\ref{S:com}, for each finite $v$, a local height function
$$\lambda_v^{\CG}\colon J(F_v)\setminus \supp(\Theta)\to \Q_p$$ 
such that for $x\in J(F)\setminus \supp(\Theta)$, we have 
$$
h^{\CG}(x) = \sum_v \lambda^{\CG}_v(x)\,.
$$
Here, $F_v$ is the completion of $F$ at $v$ and $\Theta$ is the theta divisor on $J$ with respect to the base point
$\infty\in C(F)$.
To be more precise, the construction of Coleman--Gross requires the choice of a continuous id\`ele class character
\begin{equation}\label{E:chi}
\chi = (\chi_v)_{v}\colon \mathbb{A}_F^\times/F^\times \to \Qp
\end{equation}
and, for each $v$ such that $\chi_v$ is ramified (meaning that
$\chi_v(\O^\times_{v})\ne 0$), a splitting of the Hodge filtration on $C_v\colonequals C\otimes{F_v}$, which we encode as a subspace $W_v$ of $H_{\dR}^1(C_v/F_v)$ that is complementary to the subspace of holomorphic forms.

Mazur and Tate use biextensions to construct quite general height pairings on abelian varieties. As a special case, we obtain a $p$-adic height pairing on $J(F)$, and hence a quadratic form
$h^{\MT}\colon J(F)\to \Q_p$, which depends on various choices (see
Section~\ref{sec:MT}, in particular~\S\ref{subsec:MTglob}), including a character
$\chi$ as in~\eqref{E:chi}.
These height pairings are also sums of local pairings,
so-called $\chi_v$-splittings, between certain divisors on $J$ with disjoint support.
Extending work of the first author~\cite{Bia23} and using $p$-adic
metrics on line bundles constructed in~\cite{Bes05, BMS}, we construct, for each finite prime $v$ and each divisor $D\in \Div(J)$,
a $p$-adic analogue $\lambda_{D,v}$ of the classical (real-valued) N\'eron
function (see~\S\ref{subsec:splittings}). These give rise to $\chi_v$-splittings (see
Lemmas~\ref{L:MTNeron} and~\ref{L:MTNeronp}).
For primes $v$ such that $\chi_v$ is unramified, there is a canonical
choice: it is equal to the real-valued N\'eron function, up to a multiplicative constant. For the ramified case, we show that a choice of a complementary subspace $W_v$ leads to a choice of a $p$-adic N\'eron function such that we have 
$$
h^{\MT}(x) = \sum_v\lambda_{X,v}(x)\,;\quad X=2\Theta\,,\quad x \notin \supp
(\Theta)\,.
$$
This decomposition may be viewed as a $p$-adic analogue of the usual
decomposition of the N\'eron--Tate height function as a sum of N\'eron
functions, see for instance~\cite[Chapter~11]{Lan83}
or~\cite[Chapter~9]{BG06}.
It follows from~\cite{BMS} that the global
height functions $h^{\MT}$ and $h^{\CG}$ with respect to the same choices
$\chi$ and $(W_v)_v$ are equal for any Jacobian (see
Remark~\ref{R:globcomp}).
In the case we are considering, this also follows
from work of the first author (see \cite[Corollary~5.35]{Bia23}).

Our main result is the following direct comparison between the local 
height functions. 

\begin{thm}\label{T:main}
  Let $J/F$ be the Jacobian of an odd degree genus~$2$ curve $C/F$. 
  Fix a continuous id\`ele class character $\chi\colon
  \mathbb{A}_F^\times/F^\times \to \Qp$ and, for each $v$ such that
  $\chi_v$ is ramified, a complementary subspace $W_v$ that is
  isotropic with respect to the cup product pairing. Then the local
  heights with respect to these choices satisfy
  \begin{equation}\label{E:main}
    \lambda^{\CG}_v = \lambda_{X,v}\,.
  \end{equation}
\end{thm}
The elliptic curve analogue of Theorem~\ref{T:main} already
appears in the literature, see~\S\ref{S:ell}.

We prove Theorem~\ref{T:main} when $\chi_v$ is unramified in
\S\ref{sec:comp_away} as follows. Using work of Heinz~\cite{hei04}, we can relate the
real-valued N\'eron function (and hence $\lambda_{X,v}$) to the admissible pairing of
Zhang~\cite{zhang93:admissible_pairing}. The latter can be related to
$\lambda^{\CG}_v$, which is defined in terms of intersection
theory. See~Theorem~\ref{thm:comp_away} and its proof.
Our result expresses both the $p$-adic and the
real-valued N\'eron function purely on (a regular model of) the curve,
rather than the Jacobian itself. We are not aware of any previous result of
this kind, and we believe it could be of independent interest.

For primes $v$ of ramification of $\chi$,~\cite[Theorem~4.44]{BMS} shows
that the $p$-adic N\'eron function $\lambda_{X,v}$ is
essentially given by a symmetric Green function as introduced by
Colmez~\cite{colmez1998integration}. The latter was used by Colmez to
give a new construction of the local Coleman--Gross height pairing at $v$
for divisors with disjoint support. We prove Theorem~\ref{T:main} in this
case by studying the limit case where both divisors are equal. This is based on work
of the first author~\cite{Bia23}, in particular the comparison
result~\cite[Corollary~5.32]{Bia23}. 
See~\S\ref{S:comram}, in particular Theorem~\ref{thm:comp_above} and its proof.

\begin{rem}\label{R:}
Another comparison result  
  between local decompositions of $h^{\MT}$ and $h^{\CG}$ (into local
  pairings between divisors with disjoint support) can be deduced
from the proof of~\cite[Theorem~6.10]{BMS}, see~\cite[Remark~6.11,
  Corollary~6.15]{BMS}. For Jacobians of odd degree genus~$2$ curves, see also~\cite[Corollary~5.32]{Bia23}, which is based on Colmez's work~\cite{colmez1998integration}. However, these comparisons require various choices, such as a pair of divisors of degree~0 on $C$ with disjoint support, both representing the same point, whereas no such choices are required in the present work.  Moreover, local height functions (rather than pairings between divisors with disjoint support) provide a natural setting for the quadratic Chabauty method. See also the discussion in~\cite[\S1]{Bia23}.
\end{rem}

When $A$ is an abelian variety of semistable ordinary reduction, Mazur--Tate actually construct
a canonical $\chi_v$-splitting at all finite primes $v$, yielding a
canonical Mazur--Tate $p$-adic height. This height is of
special importance; for instance, for an elliptic curve with good ordinary
reduction, this is the height that appears in the conjecture of
Mazur--Tate--Teitelbaum \cite{MTT86}. When $\chi_v$ is ramified,
it is not immediately clear that the canonical $\chi_v$-splitting arises
from a complementary subspace $W_v$; thus it is not clear that the
canonical Mazur--Tate $p$-adic height
can be decomposed into a sum of $p$-adic N\'eron functions. For
elliptic curves this is known, see~\S\ref{S:ell}.
We conjecture that this is always the case, see Conjecture~\ref{C:TNF1}. 

We prove Conjecture~\ref{C:TNF1} in Section~\ref{S:CanThetaZeta}
for ordinary reduction Jacobians of odd degree genus~$2$ curves such that the given model of the curve has semistable
reduction (see Definition~\ref{D:CSemOrd}), as follows. 
By work of Papanikolas~\cite{papanikolas}, the canonical
$\chi_v$-splitting for primes $v$ of ramification of $\chi$ 
is induced by the $v$-adic theta function $\theta_X$ of
Norman \cite{norman1985}. We show in Theorem~\ref{T:thetaneron} that $\theta_X$ can be expressed in terms of a $v$-adic sigma function. Combined with the above-mentioned work of the first author,
this gives the link to the $p$-adic N\'eron function with respect to $X$
and a certain
complementary subspace, defined using Blakestad's zeta function. 
When the reduction is good and ordinary, then this is in fact the unit root
subspace with respect to Frobenius, as shown by the first
author~\cite[Proposition~3.8]{Bia23}. We conjecture that this extends to
all Jacobians with semistable ordinary reduction, see
Conjecture~\ref{C:TNF2}.
In order to apply Blakestad's work, we find a criterion for ordinarity when
the curve has semistable reduction (see Proposition~\ref{H1ord}). The proof
uses the Cartier--Manin matrix, extending work of Manin for
good reduction (see Lemma~\ref{Hvsord}).

Our description of the above-mentioned subspace and the resulting $p$-adic
N\'eron function is explicit. In~\cite{BKM23}, we develop algorithms for their computation. There we also describe how to compute the Coleman--Gross height in bad
reduction, see also the second author's PhD thesis~\cite{KayaPhD}. The main issue is the computation of Vologodsky integrals of differentials of the third kind, based on work of the second
author~\cite{Kaya}, partially in collaboration with Katz~\cite{KatzKaya}.
We also give in~\cite{BKM23} the first example in the literature of
a quadratic Chabauty computation in bad reduction.

\subsection{Elliptic curves}\label{S:ell}
Though our notion of $p$-adic N\'eron functions with respect to divisors on
abelian varieties seems to be new, the special case of the divisor
$X=2(\infty)$ on an elliptic curve given by a Weierstrass equation and the
relation to analogously defined local Coleman--Gross heights has been studied before.
In this case, the equality~\eqref{E:main} follows for primes $v$ such
that  $\chi_v$ is unramified from a
real-valued comparison result~\cite[Theorem~1]{BM12}. In the ramified
setting, Balakrishnan and Besser showed the elliptic analogue
of~\eqref{E:main} in the special case where the curve has good ordinary
reduction and $W_v$ is the unit root subspace
(see~\cite[Corollary~4.3]{balakrishnanbesser2015}). For this, they express
$\lambda_v^{\CG}$ as a double integral
(see~\cite[Theorem~4.1]{balakrishnanbesser2015}), which they then express in terms of
the $p$-adic sigma function due to Mazur and Tate
(see~\cite[Corollary~4.2]{balakrishnanbesser2015}), which induces the
Mazur--Tate height.
It is not hard to see that their proof extends to arbitrary reduction and subspaces. 
In fact, their results also
prove Conjectures~\ref{C:TNF1} and~\ref{C:TNF2} for good ordinary
reduction. In addition to the analogue of~\eqref{E:main}, they use a
description of the unit root subspace on elliptic curves due to
Katz~\cite[(A.2.4.1)]{katz1973p}. This remains valid for semistable
ordinary reduction, so their proof of Conjectures~\ref{C:TNF1}
and~\ref{C:TNF2} extends to this case as well.

\subsection{Previous comparison results}
There are numerous results comparing  various constructions of
$p$-adic heights in the literature. We give a brief summary.
For Jacobians with good
ordinary reduction, Coleman showed in~\cite{Col91} that the global canonical Mazur--Tate
height agrees with the global Coleman--Gross height with respect to the
unit root subspace. 
Iovita--Werner already showed in~\cite{IW03} that on any abelian variety 
with semistable
ordinary reduction, the canonical Mazur--Tate height pairing is the
same as the Zarhin~\cite{Zar90} height pairing with respect to the unit root
splitting. Similarly to the Coleman--Gross construction, the Zarhin height depends on a splitting of the Hodge
filtration, but is defined for arbitrary abelian varieties. Despite their
similar constructions, the heights of Zarhin and
Coleman--Gross are not known to be equivalent for Jacobians with arbitrary reduction. In
good reduction, this is known, since in this case both have been proved to be equivalent 
to a special case of the much more general construction of Nekov\'a\v{r}~\cite{Nek93} due
by Besser~\cite{Bes04} and Nekov\'a\v{r}~\cite[\S8.1]{Nek93}.
The equivalence between Coleman--Gross and Nekov\'a\v{r} can actually be shown 
without any assumption on the reduction (see~\cite{BesserPairing}), but the
equivalence between Zarhin and Nekov\'a\v{r} is currently only known for
good ordinary reduction.

We finally discuss the relation of the above-mentioned heights to the one
of Schneider~\cite{Sch82}, which is particularly interesting, since this is
the height that appears in the conjecture of Mazur--Tate--Teitelbaum for
elliptic curves with split multiplicative reduction. It is equivalent to the canonical Mazur--Tate height
in good ordinary reduction as shown by Mazur--Tate~\cite{MT83}, but this is
not true in general; the difference was described by Werner~\cite{Wer98}.

\subsection{Notation}\label{subsec:notn}
Throughout this paper, $p$ is a prime number and $F$ is a number field. 
Let $\O_{F}$ denote the ring of integers of $F$. For a finite place $v$ of
$F$, we write $F_v$ for the completion at $v$ and $\O_v$ for its valuation
ring. For each $v$, we
fix a uniformizer $\pi_v$ of $\O_v$ and let $\log_v$ be the unique homomorphism from $\O_v^{\times}$ to $F_v$ extending the convergent series expansion
\[\log_v(1+z) = \sum_{n=1}^{\infty} \frac{(-1)^{n+1}z^n}{n},\ \ \ z\in \pi_v\O_v.\]

We fix a continuous id\`ele class character
\[\chi = (\chi_v)_{v}\colon \mathbb{A}_F^\times/F^\times \to \Qp\,.\]
If $\chi_v$ is
unramified (that is $\chi_v(\O_v^\times)=0$, which implies $v\nmid p$), then
$\chi_v$ is determined by $\chi_v(\pi_v)$. If on the other hand, $\chi_v$ is
ramified, we decompose $\chi_v$ on $\O_v^\times$ as
\begin{equation}\label{chitrlog}
\chi_v = \tr_v\circ \log_v\,,
\end{equation}
where $\tr_v$ is a $\Qp$-linear map from $F_v$ to $\Qp$. Moreover, we extend $\log_v$ to
\begin{equation}
\label{BranchOfP-adicLogComingFromEllv}
  \log_v\colon F_v^\times \to F_v\,,
\end{equation}
so that the diagram
\[\begin{tikzcd}
F_v^\times\arrow[rr, "\chi_v"] \arrow[dr, "\log_v"'] & & \Qp \\
& F_v \arrow[ur, "\tr_v"']
\end{tikzcd}\]
is commutative. For a number field $L/F$, we define the continuous id\`ele class
character
\begin{equation}\label{chiL}
  \chi_L = (\chi_{L,w})_{w}\colonequals \chi\circ N_{L/F} \colon
  \mathbb{A}_L^\times/L^\times \to \Qp\,,
\end{equation}
where $N_{L/F}$ is the id\`ele norm.
For a variety $V/F$, we write $V_v\colonequals V\otimes F_v$.
Finally, we let $q_v$ denote the unique prime number such that
$\ord_v(q_v)>0$, and we set
\begin{equation}\label{cv}
  c_v \colonequals \log\mathrm{Nm}_{F_v/\Q_{q_v}}(v)\in \R\,,
\end{equation}
where $\mathrm{Nm}$ is the norm.

\subsection*{Acknowledgements} We thank Amnon Besser,
Clifford Blakestad, Stevan Gajovi\'c, Johannes Nicaise and Michael Stoll
for helpful discussions regarding this work. We are grateful to Tianci Kang
for comments on an earlier version of the paper and for finding
several mistakes in Section~\ref{sec:comp_away} and helping to fix
them.
The first and third author were supported by NWO Grant VI.Vidi.192.106. The second author was supported by NWO grant 613.009.124 and by FWO grant GYN-D9843-G0B1721N during various stages of this project. 

\section{Vologodsky functions and integrals}\label{S:Vol}
Several constructions in this work are given in terms of Vologodsky
functions, so we start by briefly summarising the relevant parts of the theory here. 
Let $v$ be a finite place of $F$ such that $v\mid p$ and let $X/F_v$ be a
smooth geometrically connected variety. 
For such varieties over $\C$, Chen~\cite{Che77} has developed a theory of iterated path
integrals of closed forms.
When $X$ has good reduction, then the problem of finding a $p$-adic
analogue was solved for certain affine subsets of $\P^1$ by
Coleman~\cite{coleman1982dilogarithms}, who then extended his theory to single
integrals in arbitrary dimension in~\cite{coleman85:torsion} and to iterated
integrals on curves with de Shalit in~\cite{coleman_shalit88:p_adic_regulator}. Coleman's idea
is based on a notion of ``analytic continuation along Frobenius'' using
rigid analysis and Monsky--Washnitzer cohomology. 
This technique was reinterpreted in terms of the Tannakian theory of unipotent isocrystals and extended to
projective varieties of any dimension with good reduction by Besser~\cite{Bes02} who showed  
the existence of a unique Frobenius invariant path in this theory.
By axiomatizing the situation, he defined a notion of
Coleman(-analytic) functions; locally, a Coleman function is an iterated
Coleman integral, and these are  solutions to unipotent differential
equations, compatible with respect to the Frobenius invariant path. They
are~\textit{locally analytic} in the sense that 
they can be expanded locally into a convergent power series.

In this work, we do not assume that $X$ has good reduction, and we will use Vologodsky integration (see~\cite[Theorem~B]{vologodsky2003hodge}
and~\cite[Theorem~2.1]{Bes05}). This is a generalisation of both Coleman
integration 
and certain approaches to single integrals that work in arbitrary reduction due to
Colmez~\cite{colmez1998integration} and Zarhin~\cite{Zar96}. More
precisely, when the variety under consideration has good reduction at $p$,
Vologodsky integration is the same as Coleman integration, and single
Vologodsky integration agrees with Colmez and Zarhin integration.
To define a Vologodsky integral with values in $F_v$, we need to choose a
branch of the $p$-adic logarithm. But we have already deduced a branch from
our id\`ele class character $\chi$, namely
\eqref{BranchOfP-adicLogComingFromEllv}; we will use that one. 

\subsection{Vologodsky functions}

Extending his work on Coleman functions\footnote{Except that 
while Coleman's theory extends to varieties over $\C_p$, Vologodsky's
theory takes place over a finite extension of $\Q_p$, since
it uses the Zariski topology rather than the rigid topology. This will not
be an issue for us.},
Besser used Vologodsky's work to introduce a notion of locally
analytic~\textit{Vologodsky functions}
in~\cite[\S2]{Bes05} (though he calls them Coleman functions in loc.
cit.) and associated differential forms.
Following~\cite[Theorem~3.2]{BMS}, we first summarise what we need about these functions.
Let $\Omega^i(X)$ be the space of global holomorphic $i$-forms on $X$. 
\begin{thm}
\emph{(Vologodsky~\cite{vologodsky2003hodge}, Besser~\cite{Bes05})}
\label{T:volsum}
For every $i\ge 0$, there is an $F_v$-vector space of
  \em{Vologodsky $i$-forms} $\ocol^i(X)$ with values in $\Omega^i(X)$ and differentials $\dd\colon  \ocol^i(X) \to \ocol^{i+1}(X)$,  such that 
\begin{enumerate}
\item \label{t1} 
    the sequence  
    \[
    0\to K \to \acol(X)\xrightarrow{\dd} \ocol^1(X)
  \xrightarrow{\dd} \ocol^2(X)
\]
    is exact, where $\acol(X)=\ocol^0(X)$ is the space
    of~\em{Vologodsky functions on $X$}. 
  \item $\Omega^i(X)$ embeds into $\ocol^i(X)$.
  \item $\ocol^i(X)$ embeds into the space of locally analytic $\overline{F_v}$-valued $i$-forms on
    $X$, as defined in~\cite{Bes05}, and $\dd$ becomes the usual 
    differential on the image. We will usually view Vologodsky
    functions in this way.
  \item \label{t3} For every coordinate $z$ on $\mathbb{A}^1$, we have
    $\log_v(z)\in \acol(\mathbb{A}^1-\{0\})$ and
  $d\log_v(z) = \frac{dz}{z}$.
  \item \label{t4} Vologodsky forms and functions are functorial with
    respect to pullbacks and products (these are defined
    in~\cite[p.321]{Bes05} and are compatible with differentials).
 \end{enumerate}
\end{thm}
See~\cite[Theorem~3.2]{BMS} for further properties and pointers to the
proofs in~\cite{Bes05}.
\subsection{Vologodsky integrals}\label{}

We are ready to define Vologodsky integrals.

\begin{defn}\label{D:}
  A~\textit{(single) Vologodsky integral} $\int\omega \in \acol(X)$ of a closed Vologodsky~1-form $\omega\in
  \ocol^1(X)$ is a preimage of $\omega$ under $\dd$. An \textit{iterated Vologodsky 
  integral} is a Vologodsky function that can be defined iteratively by
  $$
  \int \omega_1\circ\cdots\circ\omega_n \colonequals
   \int \left( \omega_1 \int \omega_2 \circ
  \cdots \circ \omega_n \right)\,.
  $$
\end{defn}
As usual, to obtain definite integrals, one needs to fix constants of
integration.
Iterated Vologodsky integrals are important because 
every Vologodsky function can be locally described by an iterated 
integral.

To give a rough idea how this connects to unipotent isocrystals, observe
that the condition $\dd \int\omega=\omega$ can be phrased as a unipotent
differential equation. More generally, 
the iterated integral
$ \int \omega_1\circ\cdots\circ\omega_n$ is the
$y_n$-component of the solution of
unipotent system of $p$-adic differential equations
\begin{equation}\label{ColUni}
d\vec{y} = \Omega \vec{y}, \ \ \ \text{ where } 
  \Omega\colonequals \begin{pmatrix} 0&0&\cdots & 0&0\\
  \omega_1&0&\cdots & 0&0\\
  0&\omega_2&\cdots & 0&0\\
  \vdots & \vdots & \ddots & \vdots & \vdots\\
  0&0& \cdots &\omega_n&0 \end{pmatrix},
\end{equation}
with $y_0=1$. This system can be rephrased in terms of a unipotent
isocrystal with connection 
\begin{equation}\label{basicconnection}
  \nabla(\vec{y}) = \dd\vec{y}- \Omega\cdot \vec{y}\,.
\end{equation}
See~\cite{Bes05} for more details in the Coleman setting; the statements
and proofs carry over to our present situation.

For single integrals
Vologodsky showed in {\cite{vologodsky2003hodge}} that his
construction recovers the Colmez integral when both constructions apply.
The latter has the following properties.

\begin{thm}
\emph{(Colmez,~\cite[Th\'{e}or\`{e}me~1]{colmez1998integration})}
\label{T:Colmez} 
Let $K$ be a complete subfield of $\Cp$ and let $X/K$ be a smooth, geometrically connected algebraic variety. For each pair of points $P,Q\in X(K)$ and each closed $1$-form $\omega$ on $X$, there is a unique integral $\int_P^Q \omega \in K$ with the following properties:

\begin{enumerate}
    \item For a fixed point $P_0\in X(K)$, the function $f\colon X(K)\to K$ which sends $P$ to $\int_{P_0}^P \omega$ is locally analytic and satisfies $df = \omega$.
    \item It is additive: $\int_P^R \omega = \int_P^Q \omega + \int_Q^R \omega$.
    \item It is linear: $\int_P^Q (c_1\omega_1+c_2\omega_2) = c_1\int_P^Q \omega_1 + c_2\int_P^Q \omega_2$.
    \item It is functorial: If $f\colon X\to Y$ is a $K$-morphism of varieties, $P,Q\in X(K)$ and $\omega$ is a closed $1$-form on $Y$, then $\int_P^Q f^*\omega = \int_{f(P)}^{f(Q)} \omega$.
    \item If $f$ is a rational function on $X$, then
    \[\int_P^Q df = f(Q)-f(P),\ \ \ \int_P^Q \frac{df}{f} = \log_v\left(\frac{f(Q)}{f(P)}\right)\]
    provided that all terms are defined.
\end{enumerate} 
\end{thm}

\begin{thm}
\emph{(Vologodsky,~\cite[Proposition~43]{vologodsky2003hodge})}
\label{T:Vol} 
Suppose that $K$ is a finite extension of $\Qp$. Then the single Vologodsky
  integral is the integral from Theorem~\ref{T:Colmez}.
\end{thm}

\subsection{The $\delbr$-operator}\label{subsec:delbr}
We end this summary by reviewing an operator on a certain 
$\acol(X)$-submodule $\ocola^1(X)\subset \ocol^1(X)$, since we will require
it in our discussion of canonical log functions in~\S\ref{subsec:logfns}.
The elements of this submodule arise from extensions of two trivial connections, meaning they are
locally given as $\sum_i \omega_i \int \eta_i$. Besser shows
in~\cite[\S6]{Bes02} and~\cite[\S2]{Bes05} that one may
define  an operator 
\begin{equation}\label{delbar}
\delbr\colon \ocola^1(X) \to \Omega^1(X) \otimes\hdr^1(X/F_v)
\end{equation}
that plays a similar role to $\delbr$ in the archimedean theory.
When $X$ is affine, $\delbr$ is surjective and can be described
explicitly by $\delbr (\omega_i \int
  \eta_i) = \omega_i \otimes [\eta_i]$.

\section{Coleman--Gross heights}\label{S:CG}

Suppose that $C$ is a smooth projective geometrically connected curve defined
over $F$. In this section, we review the definition of the
(extended) Coleman--Gross\footnote{When $C$ has good reduction at all
places dividing $p$, the construction is due to Coleman--Gross \cite{CG89},
hence the name ``Coleman--Gross''; Colmez~\cite{colmez1998integration} and
Besser \cite{BesserPairing} later gave an extended
definition of the Coleman--Gross pairing without any assumptions on the
reduction type, hence the adjective ``extended''. For simplicity, we will
drop the word ``extended''.} height pairing on $C$. This pairing, which in
this section we denote for simplicity by $h$ (rather than $h^{\CG}$), is a function from $\Div^0(C) \times \Div^0(C)$ to $\Qp$, where $\Div^0(C)$ denotes the group of degree $0$ divisors on $C$ defined over $F$. We first follow \cite[\S2]{BesserPairing} for the case of disjoint support. We then generalize the construction to divisors with common support, following \cite{gro86} and \cite{balakrishnanbesser2015}.

The pairing $h$ depends on $\chi$.
Moreover, for each $v\mid p$ such that $\chi_v$ is ramified, let
$H_{\dR}^{1,0}(C_v/F_v)$ denote the image of the space of holomorphic
forms inside $H_{\dR}^1(C_v/F_v)$ under the Hodge filtration and
let $W_v$ be a subspace of $H_{\dR}^1(C_v/F_v)$ that is complementary to this image.

\begin{rem}\label{R:}
When the Jacobian variety $J/F$ of $C$ has semistable ordinary reduction at $v$ in the sense
of Definition~\ref{D:SemOrd}, there is a canonical choice of such a complementary space: the unit root subspace for the action of Frobenius. See, for example, the paragraph after \cite[Lemma~2.2]{iovita2000formal} for some generalities on unit root subspaces. 
\end{rem}

We now describe the global height pairing $h =h^{\CG}  = h_{\chi,(W_v)_v}$ 
\footnote{We write
$h^{\CG}$ (respectively $h_v^{\CG}$) for the global (respectively local)
Coleman--Gross height in other sections of this paper. For ease of
notation, we drop the superscript in the present section.}.
It is, by definition, a sum
of local height pairings $h_v$ over all finite places of $F$. In more
precise terms, let $D_1$ and $D_2$ be two elements of $\Div^0(C)$ with
disjoint support, and for a finite place $v$ of $F$, let $h_v(D_1,D_2)
=h^{\CG}_v(D_1,D_2)$ denote the local height pairing at $v$, which will be
defined in~\eqref{IntersThryForm} and~\eqref{IntegralForm}. Then
\begin{equation}\label{globloc}
  h(D_1,D_2)= \sum_v h_v(D_1\otimes F_v,D_2\otimes F_v).
\end{equation}
We will remove the condition that $D_1$ and $D_2$ have disjoint
support below; see Proposition~\ref{P:decomp}. 

The following result is due to Coleman--Gross~\cite{CG89} for good reduction and was
extended by Besser~\cite{BesserPairing} to arbitrary reduction, replacing Coleman integration by
Vologodsky integration.
\begin{prop}\label{P:globalCG}
The Coleman--Gross height pairing $h=\sum_v h_v$   with respect to a choice
$\chi=(\chi_v)_v$ and $(W_v)_v$ induces a bilinear pairing
$$h= h^{\CG}\colon J(F)\times J(F)\to \Q_p\,.$$
  The pairing is symmetric if all $W_v$ are isotropic with respect to the cup product.
\end{prop}

\begin{proof}
This follows from Proposition~\ref{P:hvprops} below.
\end{proof}

Now, fix a finite place $v$ of $F$, and let $D_1$ and $D_2$ be divisors of degree~0 on $C_v$. We will define the local component $h_v(D_1,D_2)$. 

\subsection{Local Coleman--Gross height pairings for disjoint divisors}
\label{subsec:disjoint}
We first consider the case where $D_1$ and $D_2$ have disjoint support. 

\subsubsection{Local Coleman--Gross height pairings for 
disjoint divisors and $\chi_v$ unramified}
\label{subsec:away-disjoint}
When $\chi_v$ is unramified, the local term is described using arithmetic intersection theory. 
For divisors $\cD$ and $\cE$ on an arithmetic surface $\mathcal{X}$ over $\O_v$ without common horizontal component, we define the intersection multiplicity
 $(\cD\cdot \cE)\in \Z $ as in~\cite[\S III.2, \S III.3]{Lan88}.
We can  extend this to $\Q$-divisors without common horizontal component by
extending scalars.

Now choose a proper regular model $\cC_v$ of $C_v$ over $\O_v$, and, for
$i=1,2$, an extension $\cD_i$  of $D_i$ to a $\Q$-divisor on $\cC_v$ that
has trivial intersection with all vertical divisors. Following
Coleman--Gross (see~\cite[Proposition~1.2]{CG89}), we define
\begin{equation}\label{IntersThryForm}
h_v(D_1,D_2) \colonequals \chi_v(\pi_v)\cdot(\cD_1\cdot\cD_2)\in \Q_p\,.
\end{equation}
This is independent of the choice of regular model or of the extensions
$\cD_i$.

\begin{rem}\label{R:CGNeron}
  By~\eqref{IntersThryForm}, we have $h_v(D_1,D_2) = \chi_v(\pi_v)\cdot
  h_v^{\Q}(D_1,D_2)$, where $h_v^{\Q}(D_1,D_2)\in \Q$. It follows
  from~\cite[Equation~(3.7)]{gro86} or~\cite[Theorem~5.2]{Lan88} that 
  $h_v^{\Q}(D_1,D_2)\cdot c_v\in \R$ is the
  \textit{local N\'eron pairing} or \textit{local N\'eron symbol} between
  $D_1$ and $D_2$, where $c_v$ is defined in~\eqref{cv}. Similar to the local Coleman--Gross height pairing, it
  is a local summand of the global N\'eron--Tate height pairing between the
  classes on $J$ of $D_1$ and $D_2$, and hence our results for $h_v$ in the
  unramified case also hold for the local N\'eron pairing and vice versa.
  Similar relations between real-valued and $p$-adic analogues will hold for the valuations
  in~\S\ref{subsec:NFreal} and the N\'eron functions
  in~\S\ref{subsec:NFunram}.
\end{rem}

\subsubsection{Local Coleman--Gross height pairings   
for disjoint divisors and $\chi_v$ ramified}
\label{subsec:above-disjoint}

The local term at a place $v\mid p$ such that $\chi_v$ is ramified is given
in terms of a Vologodsky integral.

We say that a meromorphic $1$-form $\omega$ on $C_v$ over $F_v$ is of the
third kind if it is regular, except possibly for simple poles with integer
residues. Such differentials form a group, which we denote by $T(F_v)$.
Logarithmic differentials, i.e., those of the form $df/f$ for $f\in F_v(C_v)^{\times}$, are of the third kind. The residue divisor homomorphism
\[\Res\colon T(F_v)\to \Div^0(C_v),\ \ \ \nu\mapsto \sum_P \Res_P \nu \cdot
(P)\,,\]
where the sum is taken over closed points of $C_v$, gives the following exact sequence:
\[0\rightarrow H_{\dR}^{1,0}(C_v/F_v) \rightarrow T(F_v) \rightarrow \Div^0(C_v) \rightarrow 0.\]
Therefore, there exists a form of the third kind on $C_v$ defined over
$F_v$ whose residue divisor is $D_1$, and this form is unique up to a holomorphic differential. By \cite[Proposition~2.5]{CG89}, there is a canonical homomorphism
\[\Psi\colon T(F_v)\to H_{\dR}^1(C_v/F_v)\]
which is the identity on holomorphic differentials, and our complementary subspace $W_v$ gives a unique choice $\omega_{D_1}$ satisfying
\[\Res(\omega_{D_1})=D_1,\ \ \ \Psi(\omega_{D_1})\in W_v\,.\]
The uniqueness follows from the definition of $W_v$ and the fact that $\Psi$ is the identity on holomorphic differentials. Since
$\omega_{D_1}$ is $F_v$-rational if $D_1$ is
(see~\cite[Proposition~3.2]{CG89}), we may define the local term as
\begin{equation}\label{IntegralForm}
h_v(D_1,D_2) \colonequals \tr_v\left(\Vint_{D_2}\omega_{D_1}\right)\in \Q_p\,.
\end{equation}

\subsubsection{Properties of local Coleman--Gross height pairings for disjoint support}
\label{subsec:hvprops}
Recall that we fixed a finite place $v$ of $F$. For a divisor $D=\sum_Pn_P(P)\in \Div(C_v)$ and a principal divisor $\div(f)\in \Div^0(C_v)$ relatively prime to $D$, we define $f(D) \colonequals \prod_Pf(P)^{n_P}$.

\begin{prop}\label{P:hvprops}
The local height pairing $h_v$ satisfies the following properties. 
  \begin{enumerate}
    \item\label{hvbiadd} $h_v$ is biadditive on divisors with disjoint support.
    \item\label{hvprinc} $h_v(\div(f),D_2) = \chi_v(f(D_2))$ for all $D_2\in \Div^0(C_v)$ and $f\in
      F_v(C_v)^\times$ such that $\div(f)$ and $D_2$ have disjoint support. 
    \item\label{hvsymm} If $\chi_v$ is unramified, then $h_v$ is symmetric.
      If  $\chi_v$ is ramified, then $h_v$ is symmetric if
      and only if the subspace $W_v$ is isotropic with respect to the cup
      product pairing.
    \item\label{hvcont} Fix $D_2$ and a point $P_0\in C_v(F_v)\setminus \supp(D_2)$. Then the function
    \[C_v(F_v) \setminus \supp(D_2) \to \Qp,\ \ \ P\mapsto  h_v((P)-(P_0), D_2)_v\]
    is continuous. 
    \item\label{hvS} If $\chi_v$ is unramified, then $h_v$
      takes values in the $\Q$-algebra $S\colonequals \Q\cdot
      \chi_v(\pi_v)\subset \Q_p$ and
      is locally constant. 
    \item\label{hvuniq}  If $\chi_v$ is unramified, then $h_v$
      is the unique pairing
      satisfying~\eqref{hvbiadd}--\eqref{hvS}.  
    \item\label{hvext} 
      Let $L/F$ be a finite extension, let $w$ be an extension of $v$ to $L$, and let
      $h_w$ be the local height pairing on $\Div^0(C_{w})$ with respect
      to $\chi_{L,w}$ (and to $W_v\otimes L_w$ if $\chi$ is ramified at
      $v$). 
      Then we have 
      $$
        h_{w}(D_1,D_2) = [L_w:F_v]h_v(D_1,D_2)
      $$
      for relatively prime $D_1,D_2\in \Div^0(C_v)$.

  \end{enumerate}
\end{prop}
\begin{proof}
When $\chi_v$ is unramified, then $h_v(D_1,D_2)$ is, up to a multiplicative constant, the
  local (real-valued) N\'eron symbol defined in~\cite[\S III.5]{Lan88}
  and~\cite[\S3]{gro86}, so~\eqref{hvbiadd}--\eqref{hvcont} follow
  from~\cite[Theorem~III.5.1]{Lan88} and~\eqref{hvext} follows from the
  discussion in~\cite[\S2]{gro86}.
  The definition~\eqref{IntersThryForm} and the proof of~\cite[Theorem~III.5.2]{Lan88} show~\eqref{hvS}.
  For~\eqref{hvuniq}, suppose that there are two pairings
  satisfying~\eqref{hvbiadd}--\eqref{hvS}. Their difference then defines a
  locally constant pairing $J(F_v)\times J(F_v)\to S$ that is linear in one
  argument if the other is fixed. Since $S$ is a one-dimensional
  $\Q$-algebra, the image of the pairing is finite by
  construction, so it must be trivial.

  For 
  $\chi_v$ ramified,~\eqref{hvbiadd}--\eqref{hvcont}
  are~\cite[Proposition~5.1]{CG89} (the proof is given there for good reduction,
  but it works more generally). Property~\eqref{hvext} follows from the
  definition and properties of the Vologodsky (or Colmez) integral; see
  \cite[Lemme~II.2.16\thinspace{}(iv)]{colmez1998integration}. 
  \end{proof}

\subsection{Local Coleman--Gross height pairings  for arbitrary divisors}
\label{subsec:local-arb}
In~\cite[\S5]{gro86}, Gross extends the (real-valued) local N\'eron symbol
to divisors with common support. This extension is used in the seminal work
of Gross and Zagier~\cite{GZ86}, see also~\cite{Con04}. We now discuss an
analogous extension of $h_v$ due to Balakrishnan and Besser~\cite{balakrishnanbesser2015}. 

For $v\mid p$ and ramified $\chi_v$, the extension uses Besser's $p$-adic
Arakelov theory~\cite{Bes05} for which the following is required; we will
assume throughout that this is satisfied.
\begin{ass}
\label{A:iso}
Each $W_v$ is isotropic with respect to the cup product pairing.
\end{ass}

Suppose that the divisors $D_1$ and $D_2$ are not relatively prime. The
idea to define the local height pairing of $D_1$ and $D_2$ is to find a
function $f\in F_v(C_v)^\times$ such that $D_1+\div(f)$ and $D_2$ are
relatively prime and to use items~\eqref{hvbiadd} and~\eqref{hvprinc} of
Proposition~\ref{P:hvprops}. However, since $\div(f)$ and $D_2$ have common
support, $f(D_2)$ is either $0$ or not defined, so $\chi_v(f(D_2))$ is not
defined. To remedy this, let $Z^0(C_v)$ be the group of zero cycles on $C_v$,
which is the subgroup of $\Div^0(C_v)$
of divisors that have only $F_v$-rational points in their support.
Suppose that $D_2\in Z^0(C)$. We fix an
$F_v$-rational tangent vector $t_P$ at every $P\in \supp(D_2)$. The vector $t_P$
induces a normalised local parameter $z_P$ at $P$ such that $\partial_{t_P}z_P=1$, where $\partial_{t_P}$ is the derivation corresponding to $t_P$. We then define
\begin{equation}\label{fxt}
  f(P,t_P) \colonequals \frac{f}{z_P^{\ord_P(f)}}(P)\,;
\end{equation}
this only depends on $t_P$, not $z_P$. If $D_2 = \sum_P n_P(P)$, we define
\begin{equation}\label{fDt}
  f(D_2, \underline{t}) \colonequals \prod_{P} f(P,t_P)^{n_P}\,,
\end{equation}
where $\underline{t}=(t_P)_{P\in \supp(D_2)}$.
\begin{defn}\label{D:hvcommon} 
  Let $D_1,D_2\in Z^0(C_v)$ and let $f\in F_v(C_v)^\times$ such that $D_1+\div(f)$ and
  $D_2$ are relatively prime. The local height pairing between $D_1$ and $D_2$ with
  respect to $\underline{t}=(t_P)_{P\in \supp(D_2)}$ is defined as
  \begin{equation}\label{}
    h_{v,\underline{t}}(D_1,D_2)= h^{\CG}_{v,\underline{t}}(D_1,D_2)\colonequals h_v(D_1+\div(f), D_2)
    -\chi_v(f(D_2,\underline{t}))\,.
  \end{equation}
\end{defn} 
We discuss an extension to $\Div^0(C_v)$ in~\S\ref{subsec:common-nonrat} below.

\begin{lem}\label{L:hvcommon}
  Let $D_1,D_2\in Z^0(C_v)$.
  The local height pairing $h_{v,\underline{t}}$ satisfies: 
  \begin{enumerate}
    \item\label{hvcommonindep} $h_{v,\underline{t}}(D_1,D_2)$ is independent of the choice of $f$.
    \item\label{hvcommonprinc}  $h_{v,\underline{t}}(\div(g),
      D_2)=\chi_v(g(D_2, \underline{t}))$ for $g\in F_v(C_v)^\times$.
    \item\label{hvcommonchange}  For each $P\in \supp(D_2)$ choose 
      $\beta_P\in F_v^\times$ and let
      $\underline{\beta}\underline{t}=(\beta_Pt_P)$. Then 
      $$h_{v,\underline{\beta}\underline{t}}(D_1,D_2) =
      h_{v,\underline{t}}(D_1,D_2) + \sum_{P\in \supp(D_2)
      }\ord_P(D_1)\ord_P(D_2)\chi_v(\beta_P).$$ 
  \end{enumerate}
\end{lem}
\begin{proof}
See~\cite[Propositions~2.4 and 3.4]{balakrishnanbesser2015}.
\end{proof}
\begin{rem}\label{R:section_tangent}
  One way to fix a tangent vector at every rational point is to fix a (set-theoretic)
  section $t$ of the tangent bundle defined over $F_v$. In this case, we write $h_{v,t}$ for the corresponding
  local height pairing. 
\end{rem}

\subsubsection{Local Coleman--Gross height pairings  for arbitrary
zero-cycles and $\chi_v$ unramified}
\label{subsec:away-common}

Suppose that $v\nmid p$. We now give an interpretation of the local Coleman--Gross pairing defined above in terms of intersection theory on a regular model $\mathcal{C}_v$ of $C_v$ over $\O_v$. This requires defining the self-intersection $(\bar{P}\cdot \bar{P})_{t_P}$, where $\bar{P}$ is the closure of $P\in C_v(F_v)$ in $\mathcal{C}_v$ and $t_P$ is an $F_v$-rational tangent vector at $P$. The following is taken from the proof of~\cite[Proposition~3.3]{balakrishnan2016quadratic}.

\begin{defn}\label{D:self_int}
  Let $\beta \in F_v^\times$ such
  that $\beta{t_P}$ generates the tangent space of $\mathcal{C}_v$ at $\bar{P}$. Then we
  define
  $$(\bar{P}\cdot \bar{P})_{t_P}\colonequals -\ord_v(\beta)\,.$$
\end{defn}

In particular, if $t_P$ is a generator, then $(\bar{P}\cdot
\bar{P})_{t_P}=0$, in accordance with~\cite[\S5]{gro86}. Moreover, the
dependency on the tangent vector is as in
Lemma~\ref{L:hvcommon}~\eqref{hvcommonchange} (except for the
appearance of $\chi_v$). If $t$ is a section of the
tangent bundle defined over $F_v$  and if $\cD_1,\cD_2\in
\Div(\mathcal{C}_v)$ have the property that their restrictions to the
generic fibre are pointwise $F_v$-rational, then we define $(\cD_1\cdot \cD_2)_{t}$ by linearity.

Definition~\ref{D:self_int} and Lemma~\ref{L:hvcommon} imply:
\begin{cor}\label{C:hvt}
  Let $t$ be an $F_v$-rational section of the tangent bundle of $C_v$ and
  let $D_1,D_2\in Z^0(C_v)$.
  Fix extensions $\cD_i$  of $D_i$ to $\Q$-divisors on $\mathcal{C}_v$ that have trivial
  intersection with all vertical divisors. Then we have
  $$
  h_{v,t}(D_1,D_2) = \chi_v(\pi_v)\cdot(\cD_1\cdot\cD_2)_{t}\,.
  $$
\end{cor}

We will also use the following adjunction formula.
\begin{lem}\label{L:selfint}
  Let $t_P$ be determined by  $\omega(P)$ by duality, where $\omega$ is a differential form on $C_v$ such that $P\notin\div(\omega)$. Then we have
  $$
    (\bar{P}\cdot\bar{P})_{t_P} +(\bar{P}\cdot \div(\overline{\omega})) = 0\,,
  $$
  where $\overline{\omega}$ is the differential form on $\mathcal{C}_v$ induced by $\omega$.
\end{lem}
\begin{proof}
  See the proof of~\cite[Lemma~3.4 (ii)]{balakrishnan2016quadratic}.
\end{proof}

\subsubsection{Local Coleman--Gross height pairings for arbitrary
zero-cycles and $\chi_v$ ramified}\label{subsec:above-common}
Suppose now that $v\mid p$ and $\chi_v$ is ramified. Balakrishnan and
Besser express the local height pairing in terms of Green functions,
introduced by Besser in~\cite{Bes05}. We give a brief summary. For $D_1\in
Z^0(C_v)$, the Green function 
\[G_{D_1}\colon C_v(F_v)\setminus\supp(D_1)\to \Q_p\] 
is a (Vologodsky) Coleman function (see~\cite[Definition~2.2]{Bes05}) that satisfies 
\begin{equation}\label{G-omega}
dG_{D_1} = \omega_{D_1}, 
\end{equation}
by~\cite[Theorem~7.3]{Bes05}, where $\omega_{D_1}$ is the differential of the third kind defined in~\S\ref{subsec:above-disjoint}. In particular, $G_{D_1}$ depends on the subspace $W_v$ and we have
\begin{equation}\label{hvGD}
  h_v(D_1,D_2)=\tr_v(G_{D_1}(D_2))
\end{equation}
for $D_1,D_2\in Z^0(C_v)$ with disjoint support, where $\chi_v =
\tr_v\circ \log_v$, see~\eqref{chitrlog}. By extending scalars,~\eqref{hvGD} extends to $D_1,D_2\in
\Div^0(C_v)$ with disjoint support.

Following~\cite{balakrishnanbesser2015}, we extend~\eqref{hvGD} to divisors with common support by extending $G_{D_1}$ to its poles, equipped with the choice of a tangent vector, as in~\eqref{fxt}. The function $G_{D_1}$ has a logarithmic pole at $P\in \supp(D_1)$, so for any local parameter $z$ at $P$, we can write
$$
G_{D_1} = \ord_P(D_1)\log_v(z)+c_0+O(z)
$$
around $P$. The {\em constant term} of $G_{D_1}$ in $P$ with respect to $z$ is
$c_0$. For each $P\in \supp(D_1)$, we pick a local coordinate $z_{P}$ that is
normalised with respect to a tangent vector $t_P$ as above; let $\underline{t}\colonequals (t_P)_{P\in \supp(D_1)}$. Then we define 
\begin{equation}\label{GDxt}
  G_{D_1}(P, \underline{t})\colonequals G_{D_1}(P, t_P) \colonequals
  c_0(z_P)\,.
\end{equation}
If $t$ is an $F_v$-rational section of the tangent bundle of $C_v$, we also
define $G_{D_1}(P, t)$ (and thus $G_{D_1}(D_2, t)$) in the obvious way. 

\begin{lem}
\emph{(\cite[Proposition~3.4]{balakrishnanbesser2015})}
\label{L:hvtGD}
We have
  \[
    h_{v,t}(D_1,D_2)=\tr_v(G_{D_1}(D_2,t))
\]
  for $D_1,D_2\in Z^0(C_v)$.
\end{lem}

\begin{rem}\label{R:constant_term}
  The constant term was defined by Besser and Furusho for more general Coleman functions in~\cite[Definition~3.1]{BF06}. In this work, we only need it for logarithmic singularities.
\end{rem}

\subsubsection{Local Coleman--Gross height pairings for arbitrary
divisors}\label{subsec:common-nonrat}
So far, we only considered $h_{v,t}(D_1, D_2)$ for $D_1,D_2\in
Z^0(C_v)$ and $t$ an $F_v$-rational section of the tangent bundle. More generally, let $D_1, D_2\in \Div^0(C_v)$. 
Choose a finite field extension $L_w/F_v$ such that 
$D_1,D_2\in Z^0(C_w)$. 
Let $h_{w}$ be the local Coleman--Gross height pairing with respect to the
character $\chi_{L,w}$ defined in~\eqref{chiL} and let $h_{w,t}$ be as
defined in Definition~\ref{D:hvcommon}.
Then we define  
\begin{equation}\label{hvtgeneral}
  h_{v,t}(D_1,D_2) \colonequals \frac{1}{[L_w : F_v]}h_{w,t}(D_1,D_2)\,. 
\end{equation}
This is well-defined by
Proposition~\ref{P:hvprops}~\eqref{hvext} and Definition~\ref{D:hvcommon}. 
By construction, the local height pairing~\eqref{hvtgeneral} is invariant
under finite extensions up to multiplication by the degree of the
extension, as in Proposition~\ref{P:hvprops}~\eqref{hvext}, provided we
make the obvious choice of $t$ when moving to an extension.

For ramified~$\chi_v$, we can use
extension of scalars to define $G_D\colon C_v(\overline{F}_v)\setminus\supp(D)
\to \Q_p$  for any $D\in \Div^0(C_v)$. The extension~\eqref{GDxt} to points
in the support of $D$ remains exactly the same. Then Lemma~\ref{L:hvtGD}
continues to hold.

\begin{rem}\label{CGvsNT}
  Suppose that $\chi_v$ is unramified.
  Recall from Remark~\ref{R:CGNeron} that, in this case, the local
  Coleman--Gross height pairing $h_v$ is the same as the local N\'eron pairing up to a
  constant factor, and that the latter has been extended to divisors with
  common support by Gross in~\cite[\S5]{gro86} in the same way as described
  above. Therefore all results in this section
  for $h_v$ also hold true for the latter.
\end{rem}

\subsubsection{Global Coleman--Gross height pairing in terms of local height pairings for arbitrary
divisors}\label{subsec:global-common-nonrat}

Suppose that $D_1,D_2$ are two global divisors of degree~0 on $C$ (not necessarily relatively
prime or with pointwise rational support) and an $F$-rational section $t$ of the tangent bundle, then the sum over all
$\sum_vh_{v,t}(D_1\otimes F_v, D_2\otimes F_v)$ does not depend on $t$, by
part~\eqref{hvcommonchange} of
Lemma~\ref{L:hvcommon} and the extension discussed
in~\S\ref{subsec:common-nonrat}.
More precisely, we obtain the following generalisation of~\eqref{globloc}:

\begin{prop}\label{P:decomp}
  Let $D_1,D_2\in \Div^0(C)$ and let $t$ be an $F$-rational section of the tangent bundle on $C$. Then
  the global Coleman--Gross height $h$ from Proposition~\ref{P:globalCG} satisfies
  $$
    h(D_1,D_2) = \sum_v h_{v,t}(D_1\otimes F_v,D_2\otimes F_v)\,.
  $$
\end{prop}

\section{$p$-adic N\'eron functions}\label{sec:NF} 

N\'eron constructed the real-valued global N\'eron--Tate pairing on an abelian
variety $A$ over a number field $F$ as a sum of
local pairings over local fields, satisfying certain properties, and showed that these
local pairings exist and are uniquely determined by these
propertie; see \cite{Ner65}. Since the global pairing is symmetric and
bilinear, it gives rise to a quadratic form, the N\'eron--Tate height
function. This function can also be written as a sum of local functions,
which Lang called \textit{N\'eron functions} in~\cite[Chapter~11]{Lan83}.

In this section, we discuss $\Q_p$-valued analogues of N\'eron functions
with respect to a local character $\chi_v\colon F_v^*\to \Q_p$,
generalising work of the first author~\cite{Bia23}. 
When $\chi_v$ is unramified for a non-archimedean place $v$
 of $F$, then the $p$-adic
N\'eron function at $v$ is simply the $\R$-valued N\'eron function, up to a
constant multiplicative factor. For the remaining non-archimedean places, we use work of
Besser, Srinivasan and the third author~\cite{BMS} to construct a $p$-adic N\'eron
function at $v$ associated to a complementary subspace $W_v$ as in
Section~\ref{S:CG}. 
We show in Section~\ref{sec:MT} that sums of $p$-adic N\'eron
functions are global $\Q_p$-valued height pairings as defined by Mazur and
Tate.

\subsection{Real-valued N\'eron functions}\label{subsec:NFreal}
Let $v$ be a place of $F$ and let $A/F_v$ be an abelian variety.
In the following, for some set $S$ and two maps $\mu,\mu'$ with the same
domain and mapping to $S$, we write 
\begin{equation}\label{E:const}
  \mu \equiv \mu' \bmod S
\end{equation}
if  $\mu-\mu'\in S$ is constant.

\begin{prop}
\emph{(N\'eron, see~\cite[Theorems~1.1,~1.5,~5.1,~5.2]{Lan83})}
\label{P:rNeron}
  For each divisor $D$ on ${A}$, there is a function $\mu_D\colon
  A(F_v)\setminus\supp(D)\to \R$, unique
  up to constant functions, such that the following properties are
  satisfied:
  \begin{enumerate}
    \item\label{rNF1} We have 
        $\mu_{D+E} \equiv \mu_{D} + \mu_{E} \bmod \R$.
      \item\label{rNF2} If $D=\div(f)$ is principal, then
    $\mu_{D} \equiv -\log\circ |f|_v \bmod \R$.
  \item\label{rNF3} For $a\in A(F_v)$, we have $\mu_{D_a}(x) \equiv \mu_D(x+a)$ $
    \bmod\ \R$, where $D_a=\tau_a^*D$.
    \item \label{rNF5} We have $\mu_{[2]^*D}\equiv\mu_D\circ
      [2]\bmod \R$.
\item \label{rNF7} 
      Let $F'_w/F_v$ be a finite extension.
      Then we have 
      $$
        \mu_{D\otimes F'_w}(x) \equiv [F'_w:F_v]\mu_D(x) \bmod \R\quad \text{for}\; x\in
        A(F_v)\setminus\supp(D)\,
      $$
      where $\mu_{D\otimes F'_w}$ is a N\'eron function with
      respect to ${D\otimes F'_w}$. 
    \item\label{rNF4} If $v<\infty$, then the function $\mu_D$ is locally constant.
    \item\label{rNF8} If $v<\infty$, then there is a function
      $\mu^{\Q}_D\colon A(F_v)\setminus\supp(D)\to \Q$ such that 
      \[
        \mu_D \equiv 
        \mu^{\Q}_D\cdot c_v\bmod\R\,,
      \]
      where $c_v$ is defined in~\eqref{cv}.
      Moreover, $\mu^{\Q}_D$ has bounded denominator.
  \end{enumerate}
  We call $\mu_D$ a \textit{(real) N\'eron function} with respect to $D$.
\end{prop}
\begin{rem}\label{R:rlambdaDunique}
    Suppose that the class of $D$ is symmetric (respectively
      antisymmetric), and let $d =4$ (respectively $d =2$), so that there
      is a function $f\in F_v(A)$
    such that $[2]^*D -d  D = \div(f)$. By
    Proposition~\ref{P:rNeron}\thinspace{}\eqref{rNF5} and~\eqref{rNF2}, we then have
    $\mu_D\circ [2] - d  \mu_D \equiv
     -\log\circ |f|_v \bmod \R$.
      The choice of $f$ uniquely fixes the N\'eron function. 
\end{rem}

\begin{rem}\label{R:}
  When $F_v$ is archimedean, one can describe real N\'eron functions using theta functions
  (see~\cite[Chapter~13]{Lan83}) or, when $A$ is the Jacobian of a
  hyperelliptic curve, sigma functions (see~\cite{uch11}).
\end{rem}

One way to construct 
N\'eron functions is using (logarithms of)
metrics on line bundles. We now 
review this construction.

\subsection{Real-valued metrics, N\'eron functions and valuations}\label{subsec:metrics}
Let $v$ be any place of $F$ and let $X/F_v$ be a projective variety.
Following~\cite[\S2.7]{BG06}, we define a
\textit{metric} $\|\cdot \|$ on a line bundle $\cL/X$ to be a family of norms on each
fibre $\cL_x$ of $\cL$, where $x$ runs through $X(\bar{F}_v)$. We require that
metrics are continuous and
locally bounded. The latter means that for all open $U\subset X$ and all 
sections $s$ without zeros on $U$, the function $x\mapsto \log\|s(x)\|$ is locally bounded on $U$.
We call $\overline{\cL}\colonequals (\cL,\|\cdot\|)$ a \textit{metrised line bundle}.
One defines the tensor product of two metrised line
bundles by tensoring the bundles and locally multiplying the norms.  
An \textit{isometry} between metrised line bundles is an isomorphism of
line bundles that restricts to an isometry on each fibre. This
induces a group structure on the isometry classes of metrised line bundles. The
pullback of a metrised line bundle under morphisms of varieties is defined
in the obvious way. 

Metrised line bundles were first defined in the archimedean setting. For
instance, when $F_v=\C$, one often requires the metric to be Hermitian.
Such metrics were used to develop
Arakelov geometry and hence have important applications in arithmetic
geometry. Building on work of Chinburg and Rumely~\cite{CR93},  Zhang developed an
analogue of metrised line bundles for non-archimedean $v$ when $X$ is a
smooth projective curve in~\cite{zhang93:admissible_pairing}.

From now on suppose that $A/F_v$ is an abelian variety. It is a natural
question whether there are metrics that are well-behaved with respect to
the group law on $A$. For $F_v=\C$, the existence and uniqueness of such
metrics was shown by Moret-Bailly~\cite{MB85}. We now discuss an extension
of Moret-Bailly's results to arbitrary places $v$ by Zhang~\cite{Zha95}.
We say that $\cL/A$ is \textit{symmetric} (respectively
\textit{antisymmetric}) if there is an isomorphism
\begin{equation}\label{symmanti}
\varphi\colon  [2]^*\cL\simeq \cL^{\otimes d}
\end{equation}
where $d=4$ (respectively $d=2$). For an isomorphism $\varphi$
as in~\eqref{symmanti}, a metric $\|\cdot\|$ on a symmetric or
antisymmetric $\cL$ is \textit{$\varphi$-admissible} if $\varphi$ is an
isometry for the metrics induced by $\|\cdot\|$. Following Zhang, we call
the metric \textit{admissible} if it is $\varphi$-admissible for some
$\varphi$. Zhang shows~\cite[Theorem~2.2]{Zha95} that there is
a unique locally bounded continuous $\varphi$-admissible metric, and that
changing $\varphi$ by a constant also changes this metric by a constant.
Hence, up to a multiplicative constant, there is a unique
\textit{admissible} metric. 

\begin{rem}\label{R:NoRigid}
One can remove the indeterminacy by choosing an $F_v$-rational
  rigidificiation of all line bundles on $A$ as
in~\cite[\S9.5.6]{BG06} or~\cite[\S2]{BMS}, since this fixes an
  isomorphism $\varphi$ as in~\eqref{symmanti}. We do not rigidify our 
  bundles, since in analogy with Proposition~\ref{P:rNeron}, we
  are aiming to build a theory of $p$-adic N\'eron functions that are
  unique up to an additive constant.
\end{rem}

Via the usual decomposition
\begin{equation}\label{EvenOddDecomp}
  \cL^{\otimes 2}\cong \cL^+\otimes \cL^-
 \end{equation}
where 
$\cL^+\colonequals \cL \otimes [-1]^* \cL$ is symmetric  and $\cL^-
\colonequals \cL \otimes
([-1]^* \cL)^{-1}$ is antisymmetric,
one obtains a notion of admissible metrics on all line bundles $\cL/A$.
The set of admissible metrics has good functorial properties,
see~\cite[Lemma~3.5]{hei04} and~\cite[Theorem~3.1]{MB85}.

\begin{rem}\label{R:ZhangDyn}
  Zhang constructs admissible metrics dynamically, using a variant of
  Tate's trick for the construction of global N\'eron--Tate heights. 
  Accordingly, his construction works in more general dynamic situations,
  where one has an analogue of~\eqref{symmanti}.
\end{rem}

We are interested in admissible metrics because of the following result due
to Heinz:
\begin{thm}
\emph{(\cite[Theorem~3.7]{hei04})}
\label{T:NFmet}
  Let $\|\cdot\|$ be an admissible metric on a line bundle $\cL/A$. Let $s$
  be a rational section of $\cL$ and let $D=\div(s)$. Then $x\mapsto
  -\log\|s(x)\|$ defines a real
  N\'eron function with respect to $D$. 
\end{thm}
\begin{rem}\label{R:}
  In fact there is no minus sign in~\cite[Theorem~3.7]{hei04}; this is
  because Heinz uses a different normalisation of real N\'eron functions.
\end{rem}

\subsubsection{Canonical valuations}\label{S:CanVal}
Now suppose that $v$ is non-archimedean. 
Recall that our goal is to construct $\Q_p$-valued analogues of real
N\'eron-functions. To this end, we first note that real N\'eron functions are
given in terms of logarithms of metrics, rather than metrics themselves.
Let $\cL^\times$ denote the $\G_m$-torsor obtained by removing the zero section from
the total space of a line bundle ${\cL}/A$. 
Fix an admissible metric $\|\cdot\|$ on $\cL$.
By Proposition~\ref{P:rNeron}\eqref{rNF8} and Theorem~\ref{T:NFmet}, we
obtain a function $v_\cL\colon \cL^\times(F_v)\to \Q$  by setting
\begin{equation}\label{E:val}
  v_{\cL}(u)\colonequals -\log\|u\|\cdot c_v^{-1}\,.
\end{equation}
Following~\cite[\S2]{BMS}, we call the function $v_\cL$ a
  \textit{canonical valuation}; it is called a {\em N\'eron log-metric} in
  \cite[\S1.1.3]{Bet23}. 

  By construction, the canonical valuation is
  uniquely determined by the line bundle $\cL$ up to an
  additive constant.
  It becomes uniquely determined by fixing a
  rigidification for every line bundle on $A$, see Remark~\ref{R:NoRigid}. Changing the rigidification by a scalar factor $c$ changes
  $v_{\cL}$ by the additive constant $\ord_v(c)\in \Z$ (see~\cite[Remark~2.10]{BMS}).
Using this, the following  result follows straightforwardly from the
  corresponding properties of real N\'eron functions via
  Theorem~\ref{T:NFmet}; see~\cite[Proposition~2.9]{BMS} and~\cite[Lemma~3.0.1]{Bet23} for
  details.
\begin{prop}\label{P:goodval}
  The assignment $\cL\mapsto v_\cL$ satisfies the following properties, and
  the valuations $v_{\cL}$ are uniquely
  determined by them up to an additive constant in $\Z$.
  \begin{enumerate}
    \item $v_\cL \bmod{\Z}$ only depends on $\cL$ up to isomorphism of line bundles.
    \item\label{P:goodadd} $v_{\cL\otimes \mathcal{M}} \equiv v_\cL +
      v_\mathcal{M}\bmod{\Z}$.
    \item\label{P:TrivVal} $v_{\O_A}(x, \cdot) \equiv \ord_v(\cdot)\bmod{\Z}$\, for
      $x\in A$, via the isomorphism $A\times \mathbb{G}_m\simeq
      \O_A^\times$.
    \item\label{P:valfunc} $v_{\varphi^*\cL} \equiv \varphi^*v_\cL\bmod{\Z}$\, for homomorphisms $\varphi\colon A'\to A$ of
      abelian varieties over $F_v$.
    \item $v_\cL$ is locally constant.
    \item $v_\cL$ is $\Q$-valued and has bounded denominator.
  \end{enumerate}
\end{prop}
  
\begin{rem}\label{R:}
  For non-archimedean $v$ of good reduction, the admissible metric, and
  hence the canonical valuation can be
  defined using intersection theory on the N\'eron model,
  see~\cite[Example~2.7.20]{BG06}. 
\end{rem}

\begin{rem}\label{R:GenVal}
  More generally, for any metric $\|\cdot\|$ on a line bundle $\cL$ on a projective variety
  $X/F_v$ such that $\|\cdot\|$ takes values in $\Q\cdot c_v$,
  we can define a corresponding \textit{valuation}
  $v_{\|\cdot\|}\colon \cL^\times(F_v)\to \Q$ via
\begin{equation}\label{E:valgen}
  v_{\|\cdot\|}(u)\colonequals -\log\|u\|\cdot c_v^{-1}\,.
\end{equation}
We will use this construction when $X=C$ or $X=C\times C$ and $C$ is a
  curve of genus~2 in~\S\ref{sec:comp_away}.
  By construction, every such valuation satisfies 
  \begin{equation}\label{ValTrans}
    v_{\|\cdot\|}(bu) = v_{\|\cdot\|}(u) + \ord_v(b)\quad\text{ for all }\;b \in
    F_v^\times\quad\text{and }\;u \in {\cL}^\times(F_v)\,.
  \end{equation}
\end{rem}

\subsection{$p$-adic N\'eron functions: the unramified
case}\label{subsec:NFunram}
In the following, we let $p$ be a prime and fix  a non-archimedean place
$v$ of $F$ such that $v\nmid p$. We also fix a continuous homomorphism
$\chi_v\colon F_v^\times \to \Q_p$.
Let $A/F_v$ be an abelian variety. As in~\S\ref{subsec:hvprops}, consider the $\Q$-algebra $S =
  \Q\cdot\chi_v(\pi_v)\subset \Q_p$. 
  \begin{defn}\label{D:ellNF}
    For each line bundle $\cL/A$, fix 
    a canonical valuation $v_\cL$ on $\cL$.
 Let 
    $D\in \Div(A)$  and let $s$ be a rational section of $\O(D)$
such that $D=\div(s)$, 
    The {\em $p$-adic N\'eron function with respect to $D$ (and 
    $\chi_v$)}
    is the function
    \begin{align*}
\lambda_D\colon
    A({F_v})\setminus\supp(D)\to S \subset \Q_p\\
      x \mapsto v_{\cL}(s(x))\cdot \chi_v(\pi_v)\,.
    \end{align*}
  \end{defn}

Since the canonical valuation $v_{\cL}$ is uniquely
determined by $\cL$ up to an additive constant, the same is true for the
$p$-adic N\'eron function.

  \begin{rem}\label{R:pNFvsRNF}
    By definition and by Theorem~\ref{T:NFmet},  we have
    \begin{equation}\label{}
      \lambda_D \equiv    \mu^{\Q}_D\cdot \chi_v(\pi_v) \bmod{S}\,,
    \end{equation} where $\mu^{\Q}_D$ is the
    $\Q$-valued N\'eron function
    from Proposition~\ref{P:rNeron}\eqref{rNF8}.
  \end{rem}

\begin{prop}\label{P:Neron}
  For each divisor $D$ on ${A}$,
  the $p$-adic N\'eron function $\lambda_D$ satisfies
  the following properties, and is uniquely determined by them modulo $S$.
  \begin{enumerate}
    \item\label{NF1} We have 
        $\lambda_{D+E} \equiv \lambda_{D} + \lambda_{E} \bmod S$.
      \item\label{NF2} If $D=\div(f)$ is principal, then
    $\lambda_{D} \equiv -\chi_v\circ f \bmod S$.
  \item\label{NF3} For $a\in A(F_v)$, we have $\lambda_{D_a}(x) \equiv \lambda_D(x+a)$ $
    \bmod\ S$, where $D_a=\tau_a^*D$.
    \item \label{NF5} We have $\lambda_{[2]^*D}\equiv\lambda_D\circ
      [2]\bmod S$.
    \item\label{NF4} The function $\lambda_D$ is locally constant.
\item \label{NF7} 
      Let $F'_w/F_v$ be a finite extension.
      Then we have 
      $$
        \lambda_{D\otimes F'_w}(x) \equiv [F'_w:F_v]\lambda_D(x) \bmod S\quad \text{for}\; x\in
        A(F_v)\setminus\supp(D)\,
      $$
      where $\lambda_{D\otimes F'_w}$ is a $p$-adic N\'eron function with
      respect to ${D\otimes F'_w}$ and to $\chi_{F'_w,w}$\,.
  \end{enumerate}
\end{prop}
\begin{proof}
The $p$-adic N\'eron satisfies the listed properties
  because of Proposition~\ref{P:goodval}.

  For uniqueness, we follow the proof of~\cite[Lemma~3.0.1]{Bet23}. Suppose we have another family $(\lambda'_D)_D$ of
  functions $A({F_v})\setminus\supp(D)\to S$ satisfying \eqref{NF1}--\eqref{NF5}. For a divisor
  $D$, we can choose the rigidification implicit in constructing  
  $\lambda_D$ so that we have
  \begin{equation}\label{}
  \lambda_{[2]^*D}-\lambda_D\circ [2] = \lambda'_{[2]^*D}-\lambda'_D\circ
    [2]\,.
  \end{equation}
  If $D$ is symmetric (respectively antisymmetric), then the image of the
  locally constant function $$\nu_D\colonequals \lambda_D-\lambda'_D\colon
  A(F_v)\setminus\supp(D)\to S$$ is closed under multiplication by~4
  (respectively~2) and finite, hence trivial. Decomposing divisors up to a
  factor~2 into a symmetric and an antisymmetric part and
  using~\eqref{NF1}, uniqueness $\bmod S$ follows.
\end{proof}
\begin{rem}\label{R:lambdaDunique}
  As in Remark~\ref{R:rlambdaDunique}, we can fix 
the $p$-adic N\'eron function uniquely for symmetric and antisymmetric $D$
by fixing  $f\in F_v(A)$
    such that $[2]^*D -d D = \div(f)$ (where $d\in \{2,4\}$) and requiring 
    $\lambda_D\circ [2] - d \lambda_D =
      -\chi_v\circ f$.
      This corresponds to rigidifying $\O(D)$, see Remark~\ref{R:NoRigid}.
\end{rem}

\begin{rem}\label{R:}
  One can, of course, define $p$-adic N\'eron functions directly by
  requiring the properties in Proposition~\ref{P:Neron} and showing
  existence and uniqueness using Proposition~\ref{P:rNeron}, without using
  canonical valuations. However, when  $\chi_v$ is ramified, this does
  not work anymore. As we will explain, one instead uses certain $p$-adic analytic
  functions that behave somewhat analogously to canonical valuations. In fact,
  canonical valuations can be constructed from the latter;
  see~\cite[\S9]{BMS}.
\end{rem}

\subsection{Canonical log functions}\label{subsec:logfns}
Suppose that $\chi_v$ is ramified; in particular, $v\mid p$.
The construction of $p$-adic N\'eron functions above ultimately relies on
the real-valued analogue, where a dynamical approach can be used. This does
not work in the present setting, so a different idea is needed. 
Here, we give a construction based on $p$-adic log functions developed by Besser
in~\cite{Bes05} and extended in~\cite{BMS}.
These are $p$-adic analytic analogues of the valuations on line bundles discussed
in~\S\ref{subsec:NFreal}. We summarise the theory here, and we refer to~\cite[\S\S3-4]{BMS} for details. 

Let ${\cL}/A$ be a line bundle and let $\O_V({\cL}^\times)$ be
the $F_v$-vector space of Vologodsky functions on ${\cL}^\times$ defined
in Theorem~\ref{T:volsum}.
Recall that its elements are locally analytic functions that locally look like iterated
Vologodsky integrals. 
\begin{defn}\label{D:LogFn}
A  Vologodsky function $\mu\in \O_V({\cL}^\times)$ is a {\em
  log function with respect to $\cL$}
if 
  \begin{enumerate}
    \item $\mu(b u) = \mu(u) + \log_v(b)$ for all $b \in F_v^\times$ and $u
      \in {\cL}^\times(F_v)$; 
 \item $d \mu\in \ocola^1(\cL^\times)$\,,
  \end{enumerate}
where $\log_v$ is the branch of the
logarithm induced by $\chi_v$ via~\eqref{BranchOfP-adicLogComingFromEllv}
  and $\ocola^1(\cL^\times)$ is discussed in~\S\ref{subsec:delbr}.
  A~\textit{metrised line bundle} is a pair $(\cL, \mu)$, where $\mu$ is a log
  function with respect to $\cL$.
\end{defn}
The first condition is clearly the analogue of~\eqref{ValTrans}. 
One can think of log functions as $p$-adic analogues of
logarithms of Hermitian metrics on line bundles over $\C$. Such metrics are
uniquely determined by their curvature form up to a constant. In our case,
there is a higher degree of indeterminacy, but Besser has shown that one
may still attach a  curvature form to a log function, and this is
precisely the reason why one requires the second condition in
Definition~\ref{D:LogFn}, as we now explain.

To define Besser's curvature form, note that by Hodge theory,
the first Chern class $\mathrm{ch}_1({\cL})\in H^2_{\dR}(A/F_v)$ lies in the image of
$\Omega^1(A)\otimes H_{\dR}^1(A/F_v)$ under the cup product pairing. 
The latter is the codomain of the $\delbr$-operator,
see~\eqref{delbar}.
We write
  \begin{equation}\label{CurvSet}
    \cC_\cL 
    = \{\alpha\in \Omega^1(A)\otimes H_{\dR}^1(A/F_v)\,:\,
  \cup \alpha =
  \mathrm{ch}_1({\cL})\}\,. 
  \end{equation}
Let $\pi\colon \cL^\times\to A$ be the projection map.
\begin{prop}
\emph{(\cite[Proposition~4.4]{Bes05})}
\label{P:logfnprops}
  \begin{enumerate}
    \item Let $\alpha\in \cC_\cL$.  Then there is a log function $\mu$ with respect to
$\cL$, unique up to
      the integral of a holomorphic 1-form, such that 
\begin{equation}\label{curvature}
\pi^\ast \alpha=\delbr d \mu\,.
\end{equation}
\item Conversely, let $\mu$ be a log function with respect to $\cL$. Then
  there is a unique element $\alpha\in \cC_\cL$,      called the
      \em{curvature form of $(\cL,\mu)$}, such that~\eqref{curvature} holds.
  \end{enumerate}
\end{prop}

Log functions, metrised line bundles and curvature forms behave functorially with respect to pull-backs and tensor
products, and one  can define isometries of metrised line bundles 
in the natural way; see~\cite[Definition~3.9]{BMS} for details.

We now recall the construction of an analogue of the~\textit{canonical} valuation with
respect to a line bundle from~\cite[\S4]{BMS}. 
In loc. cit. all line bundles are assumed to be rigidified, which makes
various statements rather technical.
For simplicity, we do not make this assumption
in our summary; one then obtains a theory of canonical log functions (and hence
$p$-adic N\'eron functions) that are uniquely determined up to additive constants, similar
to the unramified case, see Remark~\ref{R:NoRigid}. However, in the ramified case there
is another choice to be made, namely the choice of a subspace
of $H_{\dR}^1(A/F_v)$ that is complementary to $H_{\dR}^{1,0}(A/F_v)$. 
When $A$ is the Jacobian of a curve, then this is exactly the same degree
of freedom as for the local Coleman--Gross height introduced
in~\S\ref{subsec:above-disjoint}, identifying the de Rham cohomology of the
curve and its Jacobian via an Abel--Jacobi map.

We first assume that $\cL$ is symmetric, and
we call a log function with respect to $\cL$~\textit{good} if 
some isomorphism $[2]^*\cL\cong \cL^{\otimes 4}$ is an isometry between
the metrised line bundles one obtains by endowing $[2]^*\cL$ and
$\cL^{\otimes 4}$ with the log functions induced by the log
function on $\cL$ via pullback and tensoring, respectively.
Instead of
using a dynamic approach as in the real-valued theory, one now uses the
indeterminacy of the log function associated to a curvature form
$\alpha\in \cC_\cL$: there is a unique choice of holomorphic form  such
that the resulting log function with curvature form $\alpha$ is good, up
to an additive constant.
See~\cite[Theorem~4.14]{BMS} and its proof.

Now we apply this to the Poincar\'e bundle $\cP$ on $A\times A^\vee$,
where $A^\vee$ denotes the dual abelian variety.
It turns
out that we need to restrict the choice of curvature forms $\alpha\in
\cC_\cP$.
Denote by $\pi_A,\pi_{A^\vee}$ the projections from $A\times A^\vee$ to the
respective component, and
consider 
\begin{align*}
  \Omega^1(A\times A^\vee)\otimes \hdr^1(A\times A^\vee/F_v) &\simeq
  \big(\pi_A^\ast \Omega^1(A) \otimes \pi_A^\ast \hdr^1(A/F_v)\big)\\&  \oplus
  \big(\pi_A^\ast \Omega^1(A) \otimes \pi_{A^\vee}^\ast
  \hdr^1(A^\vee/F_v)\big) \\&   \oplus
  \big(\pi_{A^\vee}^\ast \Omega^1(A^\vee) \otimes \pi_A^\ast
  \hdr^1(A/F_v)\big)\\&   \oplus
  \big(\pi_{A^\vee}^\ast \Omega^1(A^\vee) \otimes \pi_{A^\vee}^\ast
  \hdr^1(A^\vee/F_v)\big)\,.
\end{align*}
We call $\alpha \in \cC_\cP$ 
\textit{purely mixed} if the projections onto the first and fourth
components are trivial.

Fix a purely mixed curvature form $\alpha$.
Then this choice lets us define a canonical log function $\log_{\cL}$ on
every line bundle $\cL/A$, unique up to constants, as follows (see~\cite[Definition~4.39]{BMS}).
Let $\mu$ be a good log
function with respect to $\cP$, uniquely determined by the curvature form
$\alpha$ up to a constant.
Every antisymmetric line
bundle $\cL\in \Pic^0(A)=A^\vee$ arises by restriction as
$\cL\simeq \cP_{A\times \{\cL\}}$, and we define $\log_{\cL}$ to be the log
function induced by $\mu$ via restriction. When $\cL/A$ is symmetric, we have
$\cL^{\otimes 2}\simeq (\id\times \phi_\cL)^*\cP$, where $\phi_\cL(x) =
\tau_x\cL\otimes \cL^{-1}\in A^\vee$, and we define
$\log_{\cL}$ as the good log
    function with curvature form $\frac12(\id\times \phi_\cL)^\ast\alpha$.
For a general line bundle $\cL/A$, we define 
\begin{equation}\label{}
  \log_{\cL} \colonequals \frac{1}{2}(\log_{\cL^+}+\log_{\cL^-})\,,
\end{equation}
where $\cL^+$ and $\cL^-$ are defined in~\eqref{EvenOddDecomp}.

By~\cite[Proposition~6.13]{BMS}, there is a bijection from the set of
purely mixed curvature forms to the subpaces  $W_v\subset H_{\dR}^1(A/F_v)$  that
are complementary to $H_{\dR}^{1,0}(A/F_v)$, so we may define, for every
line bundle $\cL/A$, a canonical log function on $\cL$ with respect to
$W_v$, unique up to constant.

\begin{prop}\label{P:LogFnProps}
Set $S\colonequals \im \log_v\subset \Q_p$.
The canonical log function $\log_\cL$ with respect to $W_v$ satisfies the following properties:
  \begin{enumerate}
    \item\label{Logwelldef} $\log_{\cL}$ is uniquely determined by $\cL$ and
      $W_v$ modulo $S$.
    \item\label{Logsum} $\log_{\cL\otimes\cM} \equiv \log_{\cL}+\log_{\cM}
      \bmod{S}$.
    \item\label{Logprinc} $\log_{\O_A}(x,\cdot) \equiv \log_v(\cdot) \bmod{S}$ for
      all $x\in A$, via the isomorphism $A\times \mathbb{G}_m\simeq
      \O_A^\times$.
    \item\label{Logtrans} Suppose $\cL$ is antisymmetric. Then, for $a \in A$, we have $\log_{\tau_a^*\cL}
      \equiv \tau_a^*\log_{\cL}\bmod{S}$, where $\tau_a$ is translation by $a$.
    \item\label{Loggood} $\log_{[2]^*\cL}\equiv \log_{\cL}\circ [2]\bmod{S}$.
  \end{enumerate}
\end{prop}
\begin{proof}
  Recall that $\log_{\cL}$ is determined uniquely by choosing a
  rigidification. If we multiply the chosen rigidification 
by a scalar $c$, then $\log_{\cL}$ changes by
  $\log_v(c)$ (see~\cite[Remark~4.42]{BMS}), proving~\eqref{Logwelldef}.
  For~\eqref{Logsum}, see~\cite[Lemma~4.41]{BMS},~\eqref{Logprinc} follows 
  from~\cite[Proposition~4.30]{BMS} and property~\eqref{Logtrans}
  is~\cite[Proposition~4.38]{BMS}. For symmetric $\cL$,~\eqref{Loggood} is
  immediate, and for antisymmetric $\cL$ it follows
  from~\cite[Proposition~4.32]{BMS}; the general case follows by
  construction. 
\end{proof}

\begin{rem}\label{R:}
  Property~\eqref{Logprinc} requires $\alpha$ to be purely mixed, since
  otherwise~\cite[Proposition~4.30]{BMS} does not hold. 
\end{rem}

\subsection{$p$-adic N\'eron functions: the ramified case}\label{subsec:NFram}

We fix a subspace $W_v$
of $H_{\dR}^1(A/F_v)$ that is complementary to $H_{\dR}^{1,0}(A/F_v)$.
\begin{defn}\label{D:pNF}
For any line bundle ${\cL}$ on $A$, we fix a canonical
 log function $\log_{\cL}$ with respect to $\cL$ and $W_v$ .
For a divisor $D$ on $A$, we define the {\em
   $p$-adic N\'eron function $$\lambda_D\colon  A(F_v)\setminus\supp(D)\to \Q_p$$
on $A$ with respect to $D$, $W_v$ and $\chi_v$}
by 
$$\lambda_D \colonequals -\tr_v\circ\log_{\O(D)}\circ\, s\,, $$
where $\chi_v=\tr_v\circ\log_v$ and $s$ is a section of $\O(D)$ such that
   $D=\div(s)$.
 \end{defn}

Since the canonical log function $\log_{\cL}$ is uniquely determined by $\cL$
and $W_v$ up to an additive constant, and the section $s$ is uniquely
determined by $D$ up to a multiplicative constant, the $p$-adic N\'eron function $\lambda_D$ is
uniquely determined by $D$ up to an additive constant.

\begin{prop}\label{P:pNF}
Set $S\colonequals \im(\chi_v)$.
The $p$-adic N\'eron function satisfies:
 \begin{enumerate}
   \item\label{P:pNF-add} We have 
     $\lambda_{D+E} \equiv \lambda_{D} + \lambda_{E} \bmod S$.
      \item\label{P:pNF-princ} If $D=\div(f)$ is principal, then
        $\lambda_{D} \equiv -\chi_v\circ f \bmod S$.
  \item\label{P:pNF-trans} For $D \in \Div(A)$ algebraically equivalent to~0 and $a\in A(F_v)$, we
    have $\lambda_{D_a}(x) \equiv \lambda_D(x+a) \bmod S$, where
     $D_a=\tau_a^*D$. 
   \item \label{P:pNF-2} We have
     $\lambda_{[2]^*D}\equiv\lambda_D\circ [2]\bmod S$.
    \item\label{P:pNF-ext} 
      Let $L/F$ be a finite extension and let $w$ be an extension of $v$ to
      $L$.
      Then we have 
      $$
        \lambda_{D\otimes L_w}(x) \equiv [L_w:F_v]\lambda_D(x) \bmod
        S \quad \text{for}\; x\in
        A(F_v)\setminus\supp(D)\,,
      $$
      where $\lambda_{D\otimes L_w}$ is a $p$-adic N\'eron function with
      respect to ${D\otimes L_w}$, $W_v\otimes L_w$ and to $\chi_{L,w}$.
  \end{enumerate}
\end{prop}
\begin{proof}
  The first~4 properties follow from the corresponding parts of Proposition~\ref{P:LogFnProps}.
  The final statement follows from the definitions, noting that $\chi_v$
  and $\chi_{L,w}$ have the same image.
\end{proof}
\begin{rem}\label{R:lambdaDpuniq}
    Suppose that the class of $D$ is symmetric (respectively
      antisymmetric), and let $d=4$ (respectively $d=2$), so that
      there is a function $f\in F_v(A)$
    such that $[2]^*D -d D = \div(f)$; by~\eqref{P:pNF-2} and \eqref{P:pNF-princ}, we have 
    \begin{equation}\label{pNFfun}
       \lambda_D\circ [2] - d \lambda_D = -\chi_v\circ f\,. 
     \end{equation}
  Hence, for every $f$ as above, there is a unique $p$-adic N\'eron
      function $\lambda_D$ with respect to $D, W_v$ and $\chi_v$ that
      satisfies~\eqref{pNFfun}.

\end{rem}

\begin{rem}\label{R:}
  We expect property~\eqref{Logtrans} in Proposition~\ref{P:LogFnProps} and
  therefore property~\eqref{P:pNF-trans} in Proposition~\ref{P:pNF} to hold
  more generally for all $D\in \Div(A)$. 
\end{rem}

\begin{rem}\label{R:uniq}
The $p$-adic N\'eron function is not uniquely
determined by~\eqref{P:pNF-add}--\eqref{P:pNF-2} of
Proposition~\ref{P:pNF} up to constant,
since different choices of complementary subspaces give rise to different
canonical log functions. 
  It would be interesting to investigate whether  there is a unique function
satisfying \eqref{P:pNF-add}--\eqref{P:pNF-2} of
Proposition~\ref{P:pNF} if we also require a suitable analytic condition
  involving the chosen subspace.
\end{rem}
\subsubsection{Comparison with Colmez Green functions}\label{subsec:Colmez}
Now suppose that 
$A=J$ is the Jacobian variety of a curve $C$ over $F_v$ and let 
$\Theta\in \Div(J)$ 
be a theta divisor defined over $F_v$. In this setting,
we can express the $p$-adic N\'eron function in terms of the integration
theory developed by Colmez in~\cite{colmez1998integration}, see
also Section~\ref{S:Vol}. 
Colmez associates to every subspace $W_v$ of $H_{\dR}^1(C/F_v)$ that is
complementary to $H_{\dR}^{1,0}(C/F_v)$ and isotropic with
respect to the cup product pairing, a function
\[G_{\Theta, W_v}\colon J(F_v)\setminus\supp(\Theta)\to F_v,\]
unique up to an additive constant, called a {\em symmetric
Green function} (see~\cite[\S II.2]{colmez1998integration} and \cite[\S5.2]{Bia23}).
We refer the
reader to~\cite[\S\S5.1--5.3]{Bia23} for a concise summary of Colmez's
construction, containing everything we need here (and more).
\begin{prop}\label{P:pNFSymGreen}
\emph{(\cite[Theorem~4.44]{BMS})}
Up to an additive constant, the $p$-adic N\'eron function
$\lambda_{\Theta}$ 
with respect to $\Theta$, $W_v$ and $\chi_v$ is the same as $-\tr_v\circ G_{\Theta,
W_v}$, where $\chi_v=\tr_v\circ\log_v$ and $G_{\Theta, W_v}$ is a symmetric Green
  function with respect to $W_v$ and $\chi_v$. 
\end{prop}
This result will be used in the proof of our main comparison result
Theorem~\ref{T:main} in the ramified case, see the proof of
Corollary~\ref{cor:lambdavlambdaX}.

Here and in the following,
  we usually abuse notation by identifying complementary subspaces of
  $H_{\dR}^1(C/F_v)$ and $H_{\dR}^1(J/F_v)$ using an
  Abel--Jacobi map. Our main application will be to genus~2 curves
  with a unique Weierstrass point at infinity, and we will use the
  Abel--Jacobi map with respect to that point.

\section{Mazur--Tate heights}\label{sec:MT}

Let $A/F$ be an abelian variety. 
The Mazur--Tate height pairing, constructed in~\cite{MT83} and recalled in
this section, is a bilinear pairing 
\begin{equation}\label{}
  (\cdot,\cdot)^{\MT}\colon A(F)\times A^\vee(F)\to \Q_p
\end{equation}
which depends on $\chi$ and on a choice of $\chi_v$-splittings, one for each
finite place $v$ of $F$. 
In fact, the construction of Mazur--Tate allows more general codomains and is sufficiently general to 
include the real-valued N\'eron--Tate height pairing as a special case.
The latter also decomposes as a sum of real-valued N\'eron functions (also known
as canonical local heights), as discussed in~\cite[Chapter~11]{Lan83}
or~\cite[\S9.5]{BG06}. 
In this section we show that we can use 
$p$-adic N\'eron functions to construct $\chi_v$-splittings, allowing us to
construct $p$-adic Mazur--Tate heights as sums of $p$-adic N\'eron
functions.

\subsection{$\chi_v$-splittings and $p$-adic N\'eron functions}\label{subsec:splittings}

Fix a finite place $v$ of $F$.
Let
$\Div_a(A_v)$ be the group of divisors on $A_v$ that are algebraically equivalent to~0 and
$Z^0(A_v)$ the group of  zero cycles on $A_v$, i.e. the divisors on
$A_v$ that have degree~0 and only
$F_v$-rational points in their support.

\begin{defn}\label{D:splitting}
  A \textit{$\chi_v$-splitting} on $A_v$ is a biadditive pairing $(\cdot
  ,\cdot )_v$ which assigns to $\frka \in Z^0(A_v)$ and
  $D\in
  \Div_a(A_v)$ with disjoint support an element of $\Q_p$ such that
  \begin{enumerate}
    \item $(\frka, D)_v=\chi_v(f(\frka))$ if $D=\div(f)$ is principal;
    \item $(\tau_x^*\frka,\tau_x^*D)_v = (\frka,D)_v$ for all $x\in A_v({F_v})$,
      where $\tau_x$ is translation by $x$.
  \end{enumerate} 
\end{defn}
In fact, Mazur and Tate define $\chi_v$-splittings using the theory of biextensions
in~\cite[\S1]{MT83}, but they show in~\cite[\S2.2]{MT83} that their definition is
equivalent to the one given above.

We can construct $\chi_v$-splittings using $p$-adic N\'eron functions, by
linearity. 
\begin{lem}\label{L:MTNeron}
Suppose that $\chi_v$ is unramified.
    Choose a $p$-adic N\'eron function $\lambda_D$ on $A_v$ with respect to
  $D$ and $\chi_v$ for every $D\in
  \Div_a(A_v)$. For $\frka = \sum_{x}n_x(x)\in Z^0(A_v)$ relatively prime to $D$, set 
  \begin{equation}\label{}
    (\frka,D)_v\colonequals -\lambda_D(\frka) \colonequals -\sum_x n_x\lambda_D(x)\,.
  \end{equation}
  Then $(\cdot,\cdot)_v$ defines a $\chi_v$-splitting, which we call the
  {\em N\'eron $\chi_v$-splitting}.  
\end{lem}
\begin{proof}
  This follows from Proposition~\ref{P:Neron}.
\end{proof}

\begin{lem}\label{L:MTNeronp}
  Suppose that $\chi_v$ is ramified. Fix a complementary subspace $W_v$.
For every $D\in \Div_a(A_v)$, let $\lambda_D$ be a $p$-adic N\'eron 
  function with respect
  to $D$, $W_v$ and $\chi_v$. For $\frka = \sum_{x}n_x(x)\in Z^0(A_v)$ relatively prime to $D$, set 
  \begin{equation}\label{}
    (\frka,D)_v\colonequals -\lambda_D(\frka) \colonequals -\sum_x n_x\lambda_D(x)\,.
  \end{equation}
  Then $(\cdot,\cdot)_v$ defines a $\chi_v$-splitting, which we call the
  {\em $\chi_v$-splitting corresponding to $W_v$}.  
\end{lem}
\begin{proof}
 The properties of a $\chi_v$-splitting are satisfied by 
  Proposition~\ref{P:pNF}; see also~\cite[Remark~6.4]{BMS}.
\end{proof}

 \begin{defn}\label{D:anasplit}
 If $\chi_v$ is ramified, then we say that a $\chi_v$-splitting is \textit{analytic} if it is the $\chi_v$-splitting corresponding to some $W_v$.
 \end{defn}

 \subsubsection{Canonical splittings for Jacobians}\label{subsec:}
In some cases, it is possible to define a canonical $\chi_v$-splitting.
Recall that an abelian variety over $F_v$ has \textit{semistable reduction} if the connected component of the
special fibre of its N\'eron model is an
extension of an abelian variety $B$ by a torus.
\begin{defn}\label{D:SemOrd}
  We say that an abelian variety $A/F$ has \textit{semistable ordinary} reduction at $v$ if
  $A_v$ has semistable reduction and if the abelian variety $B$ is ordinary.
\end{defn}
Suppose that $J$ has semistable ordinary reduction at $v$. 
In~\cite[\S1.9]{MT83}, Mazur and Tate construct a
canonical $\chi_v$-splitting using formal
completions;
see~\cite[p.~185]{IW03} for a summary. 
We shall not need this construction; since we have specialized to the case
of a Jacobian, we can
use an alternative description due to Papanikolas~\cite{papanikolas} in terms of the
$p$-adic theta function $\theta_X$ of Norman~\cite{norman1985, norman1986}, defined
with respect to the symmetrised theta divisor
$X = \Theta + [-1]^*\Theta$. Here we assume that $p\geq 3$ and that the curve has an
$F_v$-rational point $P_0$; we let $\Theta$ be the theta divisor with respect
to this $P_0$ (see Remark 5.2 of~\cite{papanikolas}) and we let $\mathcal{J}^f$ be the formal group of the N\'eron model $\mathcal{J}$ of $J_v$.
Norman's theta function $\theta_X$ is discussed for genus~$2$
in~\S\ref{S:Norman}. For any genus and any divisor $D\in \Div_a(J_v)$ such that $D$ is linearly equivalent to a divisor of the form $\sum_{\alpha} m_{\alpha}\tau_{\alpha}^*X$ with $\alpha\in \mathcal{J}^f(\overline{\O}_v)$, the function $\theta_X$ gives rise to a function
$\theta_D\colon \mathcal{J}^f(\mathcal{O}_{\overline{F}_v})\to \overline{F}_v$. Here, $\overline{F}_v$ denotes a fixed algebraic closure of $F_v$, and $\overline{\O}_v$ is its valuation ring.

\begin{prop}
\emph{(Papanikolas~\cite[\S5]{papanikolas})}
\label{P:papa}
  Let $D\in \Div_a(J_v)$ be as above and let $\frka=\sum_xn_x(x)\in \Div^0(J({F_v}))$ such that all $x$ are
  in $\mathcal{J}^f(\O_v)\setminus\supp(D)$.
  Then there is a unique $\chi_v$-splitting that extends 
\begin{equation}\label{pap}
  \chi_v(\theta_D(\frka))\colonequals \sum_xn_x\chi_v(\theta_D(x))\,.
\end{equation}
  We call it the \textit{canonical $\chi_v$-splitting}.
\end{prop}

It is a natural question whether the canonical Mazur--Tate splitting is analytic. This would follow from the following conjecture.

\begin{conj}\label{C:TNF1}
Suppose that $J$ has semistable ordinary reduction at $v$ and let
$X=\Theta+[-1]^*\Theta$.
  Then $-\chi_v\circ \theta_X$ 
  is the restriction to
      $\mathcal{J}^f(\O_v)\setminus\supp(X)$ of a $p$-adic N\'eron function with respect to
  $X$, a complementary subspace $W_v$ and $\chi_v$. 
\end{conj}

In Section~\ref{S:CanThetaZeta}, we prove 
Conjecture~\ref{C:TNF1} for Jacobians with ordinary reduction
of curves of genus~$2$
given by a {semistable} odd degree
model (see Definition~\ref{D:CSemOrd}). In this case, we also give an explicit construction of $\theta_X$.
It follows from this construction and from work of the first author~\cite{Bia23}
that in this case the subspace $W_v$ is the unit root subspace of Frobenius if the
reduction is good ordinary. We conjecture that this is always the case. 

\begin{conj}\label{C:TNF2}
Suppose that $J$ has semistable ordinary reduction at $v$, let
$X=\Theta+[-1]^*\Theta$ and assume that Conjecture~\ref{C:TNF1} holds for $J$.  Then the subspace $W_v$ in Conjecture~\ref{C:TNF1} is the unit root subspace of Frobenius.
\end{conj}

\subsection{The global Mazur--Tate height}\label{subsec:MTglob}
Now we let $A/F$ be an abelian variety again.
Choose a $\chi_v$-splitting $(\cdot, \cdot)_v$ for every finite place $v$ of $F$ such that 
for all but finitely many primes, the splitting is the N\'eron
splitting. For $a\in A(F)$ and $b \in A^\vee(F)$, we define the \textit{Mazur--Tate
height pairing} with respect to $\chi_v$ and the chosen splittings by
\begin{equation}\label{MTpair}
  (a, b)^{\MT} \colonequals \sum_v(\frka\otimes F_v, D\otimes F_v)_v
\end{equation}
where $\frka=\sum_xn_x(x)\in Z^0(A)$ satisfies $\sum_xn_xx=a$, and
$D\in \Div_a(A)$ is relatively prime to $\frka$ and its class is $b$. 
\begin{prop}

\emph{(Mazur--Tate~\cite[\S 3]{MT83})}
\label{P:MT}
  The pairing $(\cdot,\cdot)^{\MT}$ is well-defined and bilinear.
\end{prop}
It also satisfies various functoriality properties, see~\cite[\S 3.4]{MT83}.

\subsubsection{Analytic Mazur--Tate heights}\label{subsec:MTan}
Suppose that for all $v$ such that $\chi_v$ is unramified,
the $\chi_v$-splitting is the N\'eron splitting. For all $v$ such that
$\chi_v$ is ramified, we fix a complementary subspace $W_v$, and we choose
the $\chi_v$-splitting corresponding to $W_v$ (in the sense of Lemma~\ref{L:MTNeronp}). We denote the resulting height
pairing by $(\cdot,\cdot)^{\MT}_{\underline{W}}$, where
${\underline{W}}=(W_v)$, and $v$ runs through the ramified primes for
$\chi$.  We call a Mazur--Tate
height pairing \textit{analytic} if all $\chi_v$-splittings for primes $v$
of ramification of $\chi$ are analytic.
We also assume that $A$ is principally polarised,
and we fix 
$\Theta\in\Div(A)$ such that $\phi_\Theta(a) \colonequals
[\tau_a^*\Theta-\Theta]$ defines a principal polarisation
$\phi_\Theta\colon A\to A^\vee$.
For $a\in A(F)$, we define the corresponding Mazur--Tate height function by
\begin{equation}\label{MThtfun}
  h^{\MT}_{\underline{W}}(a) \colonequals - (a,\phi_\Theta(a))^{\MT}_{\underline{W}}\,.
\end{equation}
To ease notation, we write $h^{\MT}\colonequals
h^{\MT}_{\underline{W}}$.

We now show that $h^{\MT}$ can be decomposed into a sum
of $p$-adic N\'eron functions. 
Let $X \colonequals \Theta+[-1]^*\Theta$. 
For every finite place $v$ of $F$, we want to choose a $p$-adic N\'eron function with respect to $X_v$ (and $W_v$, if $\chi_v$ is ramified). These are all only defined up to a constant.
In order to make a consistent choice, we  fix
a function $f$ such that $[2]^*X -4X = \div(f)$ and use
Remarks~\ref{R:lambdaDunique} and~\ref{R:lambdaDpuniq}.
Equivalently, we choose a global rigidification of the line bundle
$\O(X)$.

The following generalises a result for elliptic curves due to Mazur-Stein-Tate~\cite[\S2.6]{MST06}.
\begin{prop}\label{P:MTdec}
For $a\in A(F)\setminus\supp(\Theta)$ non-torsion, we have
\begin{equation*}
h^{\MT}(a) = \sum_{v} \lambda_{X, v}(a) \,.
\end{equation*}
\end{prop}

\begin{proof}
  For the proof, we compare the two sides with the canonical $p$-adic
  heights introduced in~\cite[\S5]{BMS}, whose definition we first
  recall. For a line bundle $\cL$ on $A$, we choose a global rigidification
  and we let $\log_{\cL,\p}$ denote the corresponding
  canonical log function at a place $\p$ such that $\chi_\p$ is ramified,
  and we let $v_{\cL, \q}$ denote the  corresponding canonical valuation at
  a place $\q$
  such that $\chi_\q$ is unramified. The canonical height function $\hat{h}_{\cL} = 
\hat{h}_{\cL,\underline{\alpha}, \chi}$ is then defined by
    \begin{equation*}
      \hat{h}_{\cL}(a) = \sum_{\p}t_\p(\log_{\cL,\p}(u)) + \sum_{\q}
      v_{\cL,\q}(u)\cdot\chi_\q(\pi_\q)\in \Q_p\,,
    \end{equation*}
    where $u\in \cL^\times(F)$ is in the fibre above $a$, the first sum is over
    $\p$ such that $\chi_\p$ is ramified and the second sum runs through the
    remaining places. It is easy to see that $\hat{h}_{\cL}$ is independent of the
    choice of $u$ and of the choice of rigidification. 
    We immediately obtain 
    $ \hat{h}_{\O(X)}= -\sum_v\lambda_{X,v}$.

It is shown in~\cite[\S6.1.4]{BMS} that when $A$ is a Jacobian with theta
  divisor $\Theta$, functoriality of the underlying canonical valuations
  and log functions implies
  that the canonical height with respect to the Poincar\'e bundle $\cP$ on
  $A\times A^\vee$ satisfies
  $\hat{h}_{\cP}(a,\phi_\Theta(a))= \hat{h}_{\O(X)}(a)$ for all $a\in A(F)$.  In fact, the proof given there
  works in exactly the same way in the present more general setting, since
  we have $(\id\times \phi_{\Theta})^*\cP =\O(X)$.
  By~\cite[Corollary~6.7]{BMS}, we have
  \begin{equation}\label{BMShcomp}
    h^{\MT}(a)=-\hat{h}_{\cP}(a,\phi_\Theta(a))\,,
  \end{equation}
  which finishes the proof.
\end{proof}

\begin{rem}\label{R:globcomp}
  As discussed in Section~\ref{S:intro}, the global
height functions $h^{\MT}$ and $h^{\CG}$ with respect to the same choices
$(W_v)_v$ and $\chi$ are equal. Namely, let $a\in J(F)$; then
we have $h^{\CG}(a) = -\hat{h}(a,\phi_\Theta(a))$
  by~\cite[Theorem~6.10]{BMS}, where we write $h^{\CG}(a)$ for
  $h^{\CG}(a,a)$.  Hence~\eqref{BMShcomp} implies that $h^{\MT}(a)
= h^{\CG}(a)$.
\end{rem}

\subsubsection{Canonical Mazur--Tate heights}\label{subsec:MTcan}
We continue to assume that for all $v$ such that $\chi_v$ is unramified,
the $\chi_v$-splitting is the N\'eron splitting. We also assume, in
addition, that $A$ has  semistable ordinary reduction at all $v\mid p$ such that $\chi_v$ is ramified and that all $\chi_v$-splittings are canonical when $\chi_v$ is
ramified. 
We call the resulting pairing~\eqref{MTpair}  the \textit{canonical Mazur--Tate height
pairing with respect to $\chi$}.
If Conjecture~\ref{C:TNF1} holds true, then we get a decomposition of
this pairing as in Proposition~\ref{P:MTdec}.
For Jacobians of genus~$2$ curves, our explicit description of the
Norman theta function in Section~\ref{S:CanThetaZeta} gives such a
decomposition.

\section{Comparison in genus~$2$}\label{S:com}
For the remainder of this paper, we fix the following notation. Let $v$ be
a finite place of $F$. We denote 
by $C/F_v$ a smooth projective curve of genus~$2$ given by an affine equation
\begin{equation}\label{Ceqn}
  C\colon y^2 = b(x) =  x^5 + b_1x^4+b_2x^3+b_3x^2+b_4x+b_5\;,\qquad b_1,\ldots,b_5\in
  \O_F\,.  
\end{equation} 
Denote by $J$ the Jacobian of $C$ and by
$\Theta$ the theta divisor on $J$ with respect to the base point
$\infty$. Then $\Theta$ is symmetric and we set $X=2\Theta$.
We let $\psi\colon C\times C\to J$ denote the~\textit{difference morphism}
given by 
\begin{equation}\label{E:diff}
  \psi(Q_1,Q_2) = [Q_1-Q_2]\,.
\end{equation}
By Riemann--Roch, $\psi$ is surjective. 
In fact, any $a\in J\setminus\{0\}$ has at most~2 preimages, namely if 
$\psi(Q_1,Q_2) = a$, then the only other preimage of  $a$ under $\psi$ is
$(Q_2^-,Q_1^-)$, where $Q^-$ is the image of $Q$ under the hyperelliptic
involution. Hence $\psi$ ramifies exactly along the points of order~2. The
 preimage of the origin consists of all pairs $(Q_1,Q_2)$ such that
 $Q_2=Q_1^-$. The support of the theta divisor $\Theta$ consists of all
 points whose preimage contains a point of the form $(Q,\infty)$.
(or, equivalently  $(\infty,Q)$). Finally, note that if $a\in J(L)$ for some extension $L/F_v$, then
 $a=\psi(Q_1,Q_2)$, where either $Q_1,Q_2\in C(L)$ or $Q_1,Q_2$ are
 conjugate over a quadratic extension of $L$.

In this section, we show that we can express the $p$-adic N\'eron function 
with respect to $\Theta$, $\chi_v$ (and a complementary subspace $W_v$, if
$\chi_v$ is ramified) in terms of 
a certain local height function on $J(F_v)\setminus{\supp(\Theta)}$, which we now
define.
The construction is based on the extension of the local Coleman--Gross height pairing discussed
in~\S\ref{subsec:local-arb}
and requires fixing a section of the tangent bundle on $C$, see~\S\ref{subsec:away-common} and
~\S\ref{subsec:above-common}.
As in~\cite{balakrishnan2016quadratic}, we choose the section $t$ determined by duality by $\omega_1=\frac{dx}{2y}$ at all affine points and by $\omega_2 =\frac{xdx}{2y}$ at the point at infinity. 
\begin{defn}\label{D:CGfun}
  Let $a=\psi(Q_1,Q_2)\in J(F_v)\setminus\supp(\Theta)$. If $\chi_v$ is ramified, choose a complementary subspace $W_v$.
  The {\em local Coleman--Gross height function} with respect to $\chi_v$ (and $W_v$, if $\chi_v$ is ramified) is defined by
$$
  \lambda^{\CG}_v(a)\colonequals h^{\CG}_{v,t}(Q_1-Q_2, Q_1-Q_2)\,.
$$
\end{defn}

\begin{rem}\label{R:}
It follows from the definition of $h^{\CG}_{v,t}$ that
$h^{\CG}_{v,t}(Q_1-Q_2, Q_1-Q_2) = h^{\CG}_{v,t}(Q_2^--Q_1^-,
Q_2^--Q_1^-)$. Hence $\lambda_v^{\CG}(a)$ is
well-defined. 
\end{rem}

The main results of the present section are Theorem~\ref{thm:comp_away} and
Theorem~\ref{thm:comp_above}; they compare $\lambda_v^{\CG}$ with the $p$-adic N\'eron
function~$\lambda_{X,v}$ from Definitions~\ref{D:ellNF} and~\ref{D:pNF}.
Since $p$-adic N\'eron functions are only defined up to constant, we first
discuss our normalisation.
We let $\phi_m$ be the $m$-th division polynomial constructed by Kanayama~\cite{kan05} and Uchida~\cite{uch11}. Then we have
$$[2]^*\Theta-4\Theta=\mathrm{div}(\phi_2)\,.$$
By Proposition~\ref{P:Neron} and Proposition~\ref{P:pNF}, there is a unique $p$-adic N\'eron function
$$\lambda_{X,v}\colon J(F_v)\setminus\supp(\Theta)\to \Q_p$$ with respect to
$X=2\Theta$ and $\chi_v$ (and $W_v$, if $\chi_v$ is ramified) that satisfies 
\begin{equation}\label{eq:KU_divpol}
\lambda_{X,v}\circ [2] - 4\lambda_{X,v} = -\chi_v\circ \phi_2^2\,,
\end{equation}
and we will henceforth consider only this normalisation of $\lambda_{X,v}$.
As discussed in~\S\ref{subsec:NFreal} and the proof of
Proposition~\ref{P:pNF}, this corresponds to
a choice of rigidification on $X$, though we will not need this.

\subsection{Unramified primes}\label{sec:comp_away}

The goal of this subsection is to prove the following result.

\begin{thm}\label{thm:comp_away}
Suppose that $\chi_v$ is unramified.
Then we have
    \begin{equation*}
        \lambda_v^{\CG} = \lambda_{X,v}\,.
    \end{equation*}
        \end{thm}

The main tool of the proof is Zhang's admissible pairing on divisors on $C$
with disjoint support
(see~\cite{zhang93:admissible_pairing}). Using work of Heinz and Zhang, we
can use it to give an explicit expression of the $\Q$-valued N\'eron
function $\mu_{X,v}^{\Q}$, and hence the $p$-adic 
N\'eron function $\lambda_{X,v}$, see Proposition~\ref{P:NFad}.
But we can also express $\lambda_v^{\CG}$ in terms of
the admissible pairing via work of Zhang and Cinkir.

  We first recall a construction of Zhang's admissible pairing due to
  Heinz~\cite{hei04}
  using admissible metrics. One first needs to define a notion of
  admissibility on metrics on line bundles on $C$ and on $C\times C$. Heinz
  follows an earlier construction due to Moret-Bailly in the archimedean
  case, see~\cite{MB85}.
  Let $\iota\colon C\hookrightarrow J$ denote the Abel--Jacobi
  embedding with respect to the base point $\infty$. 
  It is easy to see that for any line bundle $\cL/C$ there is an integer $n$
  such that $\cL^{\otimes n}$
  can be written as $i^*\cM$, where $\cM/J$ is a line bundle whose class in the
  N\'eron--Severi group of $J$ is a multiple of $\Theta$
  (see~\cite[\S4.2]{HeinzThesis}). A metric
  $\|\cdot\|$ on $\cL$ is~\textit{admissible} if it is the pullback of an
  admissible metric on $\cM$. By~\cite[Lemma~4.1]{hei04}, this notion is
  independent of the choice of $n$ or $\cM$.
  
  Let $\pi_1,\pi_2\colon C\times C\to C$ denote the projections. 
  Heinz calls a metric $\|\cdot\|$ on a line
  bundle $\cL$ on $C$~\textit{bi-admissible} if for any embedding $i\colon
  C\hookrightarrow C\times C$ such that $\pi_1\circ i$ is the identity and
  $\pi_2\circ i$ is constant (or the other way around), the metric
  $i^*\|\cdot\|$ is an admissible metric. If a bi-admissible metric exists
  on $\cL$, then it is unique up to a constant, and we denote it by
  $\|\cdot\|^{\bi}_{\cL}$. Moreover, the tensor product of
  bi-admissible metrics is bi-admissible, and  the pullback of a bi-admissible
  by $\pi_1$ or $\pi_2$ is admissible. See~\cite[Definition~4.2]{hei04} and
  the discussion following it for these properties. 
If $\|\cdot\|^{\bi}_{\cL}$ takes values in $\Q\cdot c_v$,
we let $v_{\cL}$ denote the corresponding valuation as
  in Remark~\ref{R:GenVal}. 
  By~\cite[Proposition~4.3]{hei04},
  there exists a bi-admissible metric $\|\cdot\|_{\O(\Delta)}^{\bi}$ on the diagonal
  bundle $\O(\Delta)$. Suitably normalized (which we shall assume), it takes values in $\Q\cdot
  c_v$ by construction. 

  The admissible pairing of Zhang and Heinz is real-valued, but for our
  purposes it is more convenient to define the following $\Q$-valued
  version:
  \begin{defn}\label{D:}
    The~\textit{admissible pairing} between two divisors $D=\sum_ia_iP_i$
    and $E=\sum_{j}b_jQ_j$ on $C$ with
    disjoint support is
    \begin{equation}\label{}
      (D,E)_a \colonequals
      v_{\O(\Delta)}(1(D,E))\colonequals-\sum_{i,j}a_ib_jv_{\O(\Delta)}(1(P_i,Q_j))\in
      \Q\,,
    \end{equation}
    where $1$ is the canonical section on $\O(\Delta)$.
  \end{defn}
  The admissible pairing is symmetric and bilinear,
  see~\cite[Theorem~4.4]{hei04}. In fact, it is not hard
  to see using~\cite[Theorem~4.4]{hei04} that $(\cdot,\cdot)_a\cdot \chi_v(\pi_v)$ extends the local height
  pairing $h_v$ defined in~\eqref{IntersThryForm} to divisors of arbitrary
  degree with disjoint support, see~\cite[Remark~4.6]{hei04}. We will present Zhang's original
  construction of the admissible pairing from~\cite[\S4.1]{zhang93:admissible_pairing} in
  Lemma~\ref{L:ZhangAd} below.

  We can now use Theorem~\ref{T:NFmet} and functoriality properties of
  (bi-)admissible metrics to express the $\Q$-valued N\'eron
  function~$\mu^{\Q}_{X,v}$ from Proposition~\ref{P:rNeron}\eqref{rNF8} 
  in terms of the admissible pairing. Recall from Remark~\ref{R:pNFvsRNF}
  that this gives an expression for the $p$-adic N\'eron function.
  \begin{prop}\label{P:NFad}
    Let $a=\psi(Q_1,Q_2)\in J(F_v)\setminus\supp(\Theta)$. Then 
    \begin{equation}\label{}
      \mu^{\Q}_{X,v}(a) \equiv -2(Q_1,Q_2)_a -2(Q_1+Q_2,\infty)_a
      \bmod{\Q}\,.
    \end{equation}
  \end{prop}
  \begin{proof}
    Theorem~\ref{T:NFmet} and Proposition~\ref{P:rNeron} imply that 
    \begin{equation}\label{}
    \mu^{\Q}_{X,v}(a) =2v_{\O(\Theta)}(s(a))\,, 
    \end{equation}
    where $v_{\O(\Theta)}$ is the canonical valuation 
    and $s$ is a section of $\O(\Theta)$ whose
    divisor is $\Theta$. 
    By~\cite[Lemma~4.2.10]{MB85} or~\cite[Lemma~2.1]{hei04}, we have 
    $$
      \psi^*\O(\Theta) \simeq
      \O(\Delta)\otimes\pi_1^*\O(\infty)^{-1}\otimes\pi_2^*\O(\infty)\,,
    $$
from which we   can easily deduce by comparing divisors that
    \begin{equation}\label{}
      \psi^*\Theta = \Delta + \pi_1^*\infty + \pi_2^*\infty\,.
    \end{equation}
    For $i\in \{1,2\}$, let $s_i$ be a section of $\cL_i = \pi_i^*\O(\infty)$
    such that $\div(s_i) = \pi_i^*\infty$.
    By functoriality of (bi)-admissible metrics, we find that
    \begin{equation}\label{NFbi}
      \mu^{\Q}_{\Theta,v}(a)  \equiv v_{\O(\Delta)}(1(Q_1,Q_2))+
      v_{\cL_i}(      s_1(Q_1,Q_2)  + 
      v_{\cL_i}(s_2(Q_1,Q_2))\bmod{\Q}\,,
    \end{equation}
    where the valuations on the bundles $\cL_i$ are bi-admissible by the 
two final bullet points after~\cite[Definition~4.2]{hei04}. The first
    summand in~\eqref{NFbi} is $-(Q_1,Q_2)_a$ by definition. 
    Since $\O(\infty) = i_1^*\O(\Delta)$, where $i_1(Q) = (Q,\infty)$, we
    find
    \[
      v_{\cL_i}(s_i(Q_1,Q_2))  \equiv v_{\O(\infty)}(1(Q_i)) =
      v_{\O(\Delta)}(1(Q_1,Q_2)) \equiv -(Q_i,\infty)_a\bmod{\Q}\,,
    \]
    where the valuation on $\O(\infty)$ is associated to the admissible
    metric as in Remark~\ref{R:GenVal}.
  \end{proof}
 A more general version of Proposition~\ref{P:NFad} was stated and
 used in the proof of \cite[Proposition~8.5]{muller2016canonical}, but no
 proof was given.

We now discuss Zhang's original construction of the admissible pairing from~\cite[\S4.1]{zhang93:admissible_pairing}, or rather its restriction
  to a pair of distinct points in $C(F_v)$.
Let $a=\psi(Q_1,Q_2)\in J(F_v)\setminus\supp(\Theta)$ and let 
  $L_w/F_v$ be a finite extension such that $Q_1,Q_2\in C(L_w)$ and such that $C\times L_w$ has a
  semistable regular model $\mathcal{C}$ with the following properties: The
  model $\mathcal{C}$ is a strong desingularisation of the 
  subscheme of $\P^{1,3,1}_{\O_w}$ defined by  the degree-6
  homogenisation of~\eqref{Ceqn}; 
no two irreducible components of the special fibre $\mathcal{C}_s$ intersect in more than one point and no irreducible component intersects itself.
The need for working with such an extension will become apparent below;
the key reason for choosing a strong desingularisation so that the Zariski
closures on $\mathcal{C}$ of two divisors
on $C$ do not intersect if their naive reductions modulo $w$ are disjoint.
  By the discussion in~\S\ref{subsec:common-nonrat} and by
  Proposition~\ref{P:pNF}~\eqref{P:pNF-ext}, 
  we may assume that $L_w=F_v$. 
The special fibre
  $\mathcal{C}_s$ gives rise, via its dual graph, to a metrised graph $G$
  whose set of vertices $V(G)$ consists of the irreducible components of
  $\mathcal{C}_s$, and whose set of
  edges $E(G)$ corresponds to the intersection points between the components. The
  metric on $G$ is determined by attaching length~1 to all edges.
  Following~\cite{CR93} and~\cite{zhang93:admissible_pairing}, we call $G$
  the \textit{reduction graph} of $C$.
  See~\cite{baker_faber06:metrized_graphs_laplacian_operator_electric, baker_rumely07:harmonic_analysis_metrized_graphs} for
  the facts about metrised graphs that we will need. Since the reduction
  graph is the skeleton of the Berkovich space $C^{\an}/\C_p$, the arguments below can alternatively be phrased in terms of
  Berkovich theory.

  For $P\in C(F_v)$, denote by $\bar{P}\in \mathcal{C}(\O_v)$ the
  corresponding section and by $\Gamma_P\in V(G)$ the component of the special
  fibre $\mathcal{C}_s$ that $\bar{P}$ intersects. 

  For an irreducible component $\Gamma\in V(G)$, we define 
  \begin{equation}\label{aGamma}
    a_\Gamma \colonequals -\Gamma^2+2p_a(\Gamma)-2\,.
  \end{equation}
We call a $\Q$-divisor $\mathcal{K}$ on $\mathcal{C}$ \textit{canonical} if 
$O(\mathcal{K})$ is isomorphic to the relative dualising sheaf
$\omega_{\mathcal{C}/\O_v}$. 
The adjunction
formula~\cite[Theorem~9.1.37]{liu02:algebraic_geometry_arithmet_curves} for
vertical divisors then 
implies that we have  
\begin{equation}\label{adjvert}
  (\mathcal{K}\cdot \Gamma) = a_\Gamma\quad\text{for all}\; \Gamma\in
  V(G)\,.
\end{equation}

Following~\cite[\S2.1]{zhang93:admissible_pairing}, we define the
\textit{canonical divisor on $G$} by  
$$
K_G \colonequals \sum_\Gamma
a_\Gamma \Gamma\in \Q^{V(G)}\,.
$$ 
Zhang assigns in~\cite[\S4.1]{zhang93:admissible_pairing} an Arakelov--Green function $g=g_D$ on $G\times G$ to any
divisor $D\in \Q^{V(G)}$ of
degree not equal to $-2$. We shall only require the \textit{admissible Arakelov--Green
function} 
$$g\colonequals g_{K_G}\colon V(G)\times V(G)\to \Q\,,$$ restricted to
$V(G)$ and normalised to have values in $\Q$ (by multiplying the function
$g_{K_g}$ constructed in~\cite[\S4.1]{zhang93:admissible_pairing} by $c_v^{-1}$).
The reason we care about this function is that it is the main
ingredient of Zhang's construction of the admissible pairing.

  \begin{lem}
  \emph{(\cite[\S4.1]{zhang93:admissible_pairing},~\cite[Theorem~4.4]{hei04})}
  \label{L:ZhangAd} 
  Let $P_1,P_2 \in C(F_v)$ be distinct and
  let $(-\cdot -)$ be the intersection multiplicity on
    $\mathcal{C}$ as in~\S\ref{subsec:away-disjoint}.
  Then we have
    \begin{equation}\label{AdmPairing}
    (P_1,P_2)_a = (\bar{P_1} \cdot \bar{P_2})
      +g(\Gamma_{P},\Gamma_{P'})\,.
    \end{equation}
  \end{lem}

Rather than going through Zhang's construction of $g$, which is quite technical,
we will be content with a fairly explicit formula that relates the admissible
Arakelov--Green functions to the discrete Laplacian operator on $G$.
See~\cite{baker_rumely07:harmonic_analysis_metrized_graphs} for
more properties of Arakelov--Green functions and~\cite{cinkir2014explicit} and~\cite{vDijkKaya} for further explicit
expressions for Arakelov--Green functions.
Let $L$ denote minus the intersection matrix on $\mathcal{C}_s$, then $L$
is the \textit{discrete Laplacian matrix} of $G$. We let $L^+\in
\Q^{\#V(G)\times \#V(G)}$ denote the
Moore--Penrose pseudoinverse of $L$. It satisfies $L^+LL^+=L^+$ and
$LL^+L=L$. By abuse of notation, we also write
$L^+$ for the bilinear pairing on $\Q^{V(G)}$ 
  associated to $L^+$.
  If $E=\sum_{\Gamma}
  c_\Gamma\Gamma \in \Q^{V(G)}$ and $d_\Gamma = (E\cdot \Gamma)$, then  we have
  \begin{equation}\label{L+}
    c_\Gamma=-L^+(D,\Gamma)\,,\quad \text{where}\; D = \sum d_\Gamma\Gamma\,.
  \end{equation}
  See~\cite{cinkir2014explicit} and~\cite[\S7]{HdJ15} for more properties of $L^+$.
  \begin{lem}\label{L:gL+}
  There is a constant $\kappa \in \Q$ such that 
   \begin{equation}\label{E:gformula}
    g(\Gamma_1, \Gamma_2)= \frac{1}{4}\left(-L^+(K_G, \Gamma_1+\Gamma_2)-
    L^+(\Gamma_1, \Gamma_1) - L^+(\Gamma_2, \Gamma_2) +
    4L^+(\Gamma_1,\Gamma_2)\right)+\kappa
   \end{equation}
    for all $\Gamma_1,\Gamma_2\in V(G)$. 
  \end{lem}
  \begin{proof}
    The metrised graph $G$ has an interpretation as a resistive electrical
    network, and we denote by $r(\Gamma_1,\Gamma_2)$ the 
effective resistance between two components and by 
 $j_{\Gamma}(\Gamma_1,\Gamma_2)$ the voltage function
 (see~\cite{baker_faber06:metrized_graphs_laplacian_operator_electric}).
It follows from \cite[Equation~(7)]{cinkir2014explicit}
that there is a constant $\kappa'$ such that
\begin{equation}\label{grj}
  g(\Gamma_1, \Gamma_2) = \frac{1}{4}\left(\sum_{\Gamma}a_\Gamma
  j_{\Gamma}(\Gamma_1,\Gamma_2)-r(\Gamma_1,\Gamma_2) \right) +\kappa'\,.
\end{equation}
By \cite[Lemma~6.1]{cinkir2014explicit}, we have
    \begin{equation}\label{r}
r(\Gamma_1,\Gamma_2)=L^+(\Gamma_1, \Gamma_1)-2L^+(\Gamma_1,
\Gamma_2)+L^+(\Gamma_2,
\Gamma_2)
    \end{equation}
and 
    \begin{equation}\label{j}
j_{\Gamma}(\Gamma_1,\Gamma_2)=L^+(\Gamma, \Gamma)-L^+(\Gamma,
\Gamma_1)-L^+(\Gamma,\Gamma_2)+L^+(\Gamma_1,\Gamma_2)\,.
    \end{equation}
Since we have
    \[
      \sum_\Gamma a_\Gamma L^+(\Gamma, \Gamma') = L^+(K_G, \Gamma')
    \]
    for all $\Gamma'\in V(G)$,~\eqref{j} implies
      \begin{equation}\label{aGammajGamma}
\sum_{\Gamma}a_\Gamma j_{\Gamma}(\Gamma_1,\Gamma_2) 
      = -L^+(K_G, \Gamma_1+\Gamma_2) + L^+(\Gamma_1,\Gamma_2)\cdot\sum_\Gamma
      a_\Gamma + \sum_\Gamma a_\Gamma L^+(\Gamma,\Gamma)\,.
      \end{equation}
    The final term is constant, and we have 
    $ \sum_{\Gamma \in V(G)} a_\Gamma = 2g(C)-2 = 2$.
    Rewriting~\eqref{grj} using~\eqref{aGammajGamma} and~\eqref{r} yields
    the desired result.
  \end{proof}
  For our purposes,~\eqref{E:gformula} can be taken as a definition of
  $g$.
  We end our preliminary considerations with the following simple identity:
  \begin{cor}\label{C:3termg}
    Let $\Gamma_1,\Gamma_2,\Gamma_3 \in V(G)$. Then there exists a constant
    $\kappa(\Gamma_3)\in \Q$, independent of $\Gamma_1$ and $\Gamma_2$, such that
    $$
      2g(\Gamma_1,\Gamma_2) +
      2g(\Gamma_1,\Gamma_3)+2g(\Gamma_2,\Gamma_3)
      =  
      L^+(\Gamma_1+\Gamma_2,2\Gamma_3-K_G)-L^+(\Gamma_1-\Gamma_2,\Gamma_1-\Gamma_2)
      + \kappa(\Gamma_3)\,.$$ 
  \end{cor}
  \begin{proof}
    This follows from Lemma~\eqref{L:gL+} using a straightforward
    computation.
  \end{proof}

\begin{proof}[Proof of Theorem~\ref{thm:comp_away}]
  We first show that $\lambda_v^{\CG}-\lambda_{X,v}$ is constant. 
   Proposition~\ref{P:NFad} and Lemma~\ref{L:ZhangAd} imply
   \[
     -\mu^{\Q}_{X,v}(a)\equiv
        2(\overline{Q}_1\cdot\overline{Q}_2) +
        2(\overline{Q}_1+\overline{Q}_2,\overline{\infty})+
        2g(\Gamma_1,\Gamma_2) + 2g(\Gamma_1+\Gamma_2,
        \Gamma_\infty)\bmod{\Q}\,,
      \]
    where we write $\Gamma_i$ for $\Gamma_{Q_i}$.
    From Corollary~\ref{C:3termg} we get
      \begin{equation}\label{lambdaXform}
        -\mu^{\Q}_{X,v}(a)- 
 2(\overline{Q}_1\cdot\overline{Q}_2) -
        2(\overline{Q}_1+\overline{Q}_2,\overline{\infty})        \equiv L^+(\Gamma_1+\Gamma_2,2\Gamma_\infty-K_G)-
        L^+(\Gamma_1-\Gamma_2,\Gamma_1-\Gamma_2)\bmod{\Q}\,.
      \end{equation}

    In order to compare this to $\lambda_{v}^{\CG}(a)$,
    consider the differential $\omega_\mathcal{C}= \frac{dx}{2y}$ on
    $\mathcal{C}$; its divisor is $2\overline{\infty} + V$ for some
    vertical divisor $V$. We denote by $\Phi\in \Q^{V(G)}$ a vertical $\Q$-divisor on $\mathcal{C}$ such that $(\overline{Q}_1 - \overline{Q}_2) + \Phi$ has intersection multiplicity~0 with all vertical components.
    Let $t_i$ be our chosen tangent vector at $Q_i$.
    Then we have 
    \begin{align*}
\frac{\lambda_{v}^{\CG}(a)}{\chi_v(\pi_v)}&= 
      \left(((\overline{Q}_1-\overline{Q}_2+\Phi)\cdot
      (\overline{Q}_1-\overline{Q}_2+\Phi)\right)\\
      &=     \left((\overline{Q}_1\cdot
        \overline{Q}_1)_{t_1}+ (\overline{Q}_2\cdot  \overline{Q}_2)_{t_2}
        -2(\overline{Q}_1\cdot
        \overline{Q}_2)+2(\overline{Q}_1-\overline{Q}_2\cdot \Phi)
        +\Phi^2\right)\\
      &=     \left((\overline{Q}_1\cdot
        \overline{Q}_1)_{t_1}+ (\overline{Q}_2\cdot  \overline{Q}_2)_{t_2}
        -2(\overline{Q}_1\cdot
        \overline{Q}_2)-\Phi^2\right)
    \end{align*}
    by Corollary~\ref{C:hvt}. Using Lemma~\ref{L:selfint}, we obtain
    \begin{align}
      -\frac{\lambda_{v}^{\CG}(a)}{\chi_v(\pi_v)}&= 
                (\overline{Q}_1\cdot  \mathrm{div}(\omega_\mathcal{C})) +
                (\overline{Q}_2\cdot  \mathrm{div}(\omega_\mathcal{C}))
                +2(\overline{Q}_1 \cdot  \overline{Q}_2) +\Phi^2
                    \nonumber\\&=
                2(\overline{Q}_1\cdot  \overline{Q}_2)
                +2(\overline{Q}_1+\overline{Q}_2\cdot  \overline{\infty})
                +(\overline{Q}_1+\overline{Q}_2\cdot  V) +
                \Phi^2\,.\label{lcg}
                \end{align}

  By~\eqref{lambdaXform} we have shown
  $\lambda_{X,v}\equiv\lambda_v^{\CG}\bmod{\Q\chi_v{\pi_v}}$
  if we can prove 
that 
                \begin{equation}\label{E:zz}
     (\overline{Q}_1+\overline{Q}_2\cdot  V) +
                \Phi^2\equiv L^+(\Gamma_1+\Gamma_2,2\Gamma_\infty - K_G)
                  -L^+(\Gamma_1-\Gamma_2,\Gamma_1-\Gamma_2)\bmod{\Q}\,,
                \end{equation}
                where the constant is independent of $Q_1$ and $Q_2$.
                A simple computation reveals
  \begin{equation}\label{Phiform}
    \Phi^2 =-L^+(\Gamma_1-\Gamma_2,\Gamma_1-\Gamma_2)\,.
  \end{equation}
  See~\cite[Proposition~7.4]{HdJ15} for details, noting that their
  pairing $g_\Gamma$ is our $L^+$.

  It remains to rewrite the intersection multiplicity $
     (\overline{Q}_1+\overline{Q}_2\cdot  V)$ in terms of $L^+$. By adding
     a suitable multiple of the entire special fibre $\mathcal{C}_s$, if
     necessary, we may assume without loss of generality that
     $(\overline{\infty}\cdot V)=0$, so that $\Gamma_\infty$ is not one of
     the components in the support of $V$. 
     The divisor $\mathrm{div}(dx/2y)=2\infty$ is a canonical divisor on $C$; hence
     it extends to a canonical $\Q$-divisor $\mathcal{K}$ on $\mathcal{C}$
     by~\cite[Proposition~2.5]{CK09}.
     This extension is not unique, but we may
     require that $\Gamma_\infty$ is not one of the irreducible
     components in the support of the vertical divisor
     $\mathcal{K}-2\overline{\infty}$, and this uniquely fixes
     $\mathcal{K}$.
     It follows
     from~\cite[Corollary~6.4.13]{liu02:algebraic_geometry_arithmet_curves}
     (see~\cite[\S 5.1]{balakrishnan2016quadratic} for details)
     that $\div(\omega_\mathcal{C})-\mathcal{K}$ is supported only in
     components of multiplicity $>1$, so since $\mathcal{C}$ is semistable,
     we have that 
     \begin{equation}\label{KV}
       \mathcal{K} = 2\overline{\infty}+V\,.
     \end{equation}
     Hence the adjunction
     formula~\eqref{adjvert} implies that
     \[
       (V\cdot \Gamma) = (\mathcal{K} - 2\overline{\infty}\cdot \Gamma) =
       a_\Gamma - \delta_{\Gamma,\Gamma_\infty}\,.
     \]
Writing $V=\sum c_\Gamma\Gamma$ and letting
$D\colonequals K_G-2\Gamma_\infty$, we obtain
     from \eqref{L+} that
     $$
       (\bar{P}\cdot V) = c_\Gamma = -L^+(\Gamma, D)= -L^+(\Gamma_P, K_G) + 2L^+(\Gamma_P,\Gamma_\infty)\,.$$
In particular, we find that 
$$(\bar{Q_1}+\bar{Q_2} \cdot V) = -L^+(\Gamma_1+\Gamma_2, K_G) +
2L^+(\Gamma_1+\Gamma_2,\Gamma_\infty)\,.$$
Together with~\eqref{Phiform} this proves~\eqref{E:zz},
and therefore $\lambda_{X,v}$ and $\lambda_v^{\CG}$ differ only by a
constant.

To show exact equality, we choose an affine point $Q_1=(x,y)\in C(F_v)$
such that $\ord_v(x)<0$, and we let $Q_2$ be the image of $Q_1$ under the
hyperelliptic involution. Then $$(\overline{Q}_1\cdot  \overline{Q}_2) =
2\ord_v(2y)-6\ord_v(x) = \ord_v(4)-\ord_v(x)$$ and $(\overline{Q}_1\cdot
\overline{\infty})=-\frac{1}{2}\ord_v(x)$.
It follows from~\eqref{lcg} that
$$\lambda_v^{\CG}(a) = \chi_v(\pi_v)(3\ord_v(x)-\ord_v(4))\,.$$
By~\cite[Proposition~6.1]{dJM14},  we have $\mu^{\Q}_{X,v}(a) = -\ord_v(k_1)$, where $(k_1,k_2,k_3,k_4)$ are integral coordinates for the image of $a$ on the Kummer surface of $J$ such that one of them is a unit.  In this case, \cite[\S 2]{FlynnSmart} implies that 
$$(k_1:k_2:k_3:k_4) = \left(1:2x:x^2:\frac{T}{4y^2}\right)$$
where $T\in \Z[b_1,\dots,b_5][x]$ is monic of degree $8$. Thus, we conclude that
\begin{equation*}
 \lambda_{X,v}(a) = - \chi_v(\pi_v)v\left(\frac{T}{4y^2}\right) =
  \chi_v(\pi_v)(3\ord_v(x)-\ord_v(4)) = \lambda_v^{\CG}(a). \qedhere
\end{equation*}
\end{proof}
\begin{rem}
A generalisation of Theorem~\ref{thm:comp_away} to hyperelliptic
curves of arbitrary genus is work in
progress due to Tianci Kang.
\end{rem}

\subsection{Ramified primes}\label{S:comram}
We now show that the comparison result Theorem~\ref{thm:comp_away} also
holds in the ramified case. We keep the same notation as
in~\S\ref{sec:comp_away}, except that we now
assume that $\chi_v$ is ramified. 

\subsubsection{$v$-adic sigma functions and $p$-adic N\'eron functions}\label{subsec:vadicsigma}
In the setting that we restricted to in this section, namely that of the Jacobian of a genus $2$ hyperelliptic curve given by an odd degree model, we can give a concrete description of $\lambda_{X,v}$ by combining Proposition~\ref{P:pNFSymGreen} with \cite[Theorem~5.30]{Bia23}.
This requires working with Grant's explicit embedding of $J$ into $\P^8$ and the resulting formal group parameters $(T_1, T_2) =:T$; see \cite{Grant1990} and \cite[\S2]{Bia23} for details. 
Extending work of Blakestad \cite{blakestadsthesis}, \S3.2 (and in particular Proposition~3.5) in \cite{Bia23} provides a bijection between the set of isotropic, complementary subspaces $W_v$, and the set of certain power series $\sigma_v(T) = T_1(1+O(T_1,T_2))\in F_v[[T_1,T_2]]$ that satisfy analogous properties to the complex hyperelliptic sigma function. 

To be more precise, recall that $\iota \colon C \xhookrightarrow{} J$
denotes the
embedding of $C$ into $J$ with respect to the point at infinity $\infty\in
C(F)$. For $i\in \{1,2\}$, consider the invariant differential $\Omega_i$
on $J$ satisfying $\iota^{*}\Omega_i =  \omega_i =\frac{x^{i-1}dx}{2y}$ and let $D_i$ be its dual invariant derivation. Further, for $1\leq i,j\leq 2$, let $X_{ij} = X_{ji}$ be the rational function on $J$ described explicitly on the symmetric square of $C$ by the formulas of \cite[(4.1), (1.4)]{Grant1990}.

If $W_v$ is an isotropic complementary subspace, then by
\cite[Proposition~3.5]{Bia23}, there exists a unique $2\times 2$ symmetric
matrix $c = (c_{ij})$ with entries in $F_v$ such that $\iota^{*}W_v$ is
spanned by the classes of the differentials
\begin{equation}\label{eq:etais}
\begin{aligned}
\eta_1^{(c)} = (-3x^3 - 2b_1x^2 -b_2x + c_{12} x + c_{11})\frac{dx}{2y},\\
 \eta_2^{(c)}= (-x^2 + c_{22} x + c_{12})\frac{dx}{2y}.
 \end{aligned}
\end{equation}
Conversely, given a symmetric matrix $c = (c_{ij})\in F_v^{2\times 2}$, the classes of the differentials \eqref{eq:etais} span a subspace of $H^1_{\dR}(C/F_v)$ that is isotropic with respect to the cup product and complementary to the space of holomorphic forms. For this reason, we say that $W_v$ \emph{corresponds to} the symmetric matrix $c = (c_{ij})$.

\begin{defn}[{\cite[\S 3.2]{Bia23}}]\label{def:sigma}
Let $W_v$ be the subspace corresponding to a symmetric matrix $(c_{ij})\in F_v^{2\times 2}$.
The \emph{$v$-adic sigma function} $\sigma_v(T) = \sigma_v^{(c)}(T)\in F_v[[T_1,T_2]]$ 
attached to $W_v$ is the unique odd solution of the form $T_1 (1+ O(T_1,T_2))$ to the system \begin{equation*}
D_iD_j(\log(\sigma_v^{(c)}(T))) =-X_{ij}(T) + c_{ij}, \qquad \text{for all } 1\leq i,j\leq 2.
\end{equation*}
\end{defn}
 
By \cite[Proposition~3.4]{Bia23}, a $v$-adic sigma function induces a
function on a finite index subgroup $H_v$ of the model-dependent formal
group of $J$. Furthermore, it satisfies
\begin{equation*}
    \frac{\sigma_v([m](T))}{\sigma_v(T)^{m^2}} = \phi_m(T),\qquad \text{for all } m\in \Z_{>0},
\end{equation*}
and we can use this formula to extend the domain of $\sigma_v$ to $J(F_v)$.

This was used in \cite{Bia23} to construct a $p$-adic height on $J$, whose
local term $[F_v:\Q_p]\lambda_v^{B}$ at the ramified prime $v$ (and not
only by \cite[Remark~4.6\thinspace{}(iii)]{Bia23}) is given in terms of the
sigma function $\sigma_v$ determined by our choice of $W_v$.
In particular,
for a non-torsion point $a\in J(F_v)\setminus \supp(\Theta)$, we have
\begin{equation*}
\lambda_v^B(a) = -\frac{2}{ m^2[F_v : \Q_p]}\cdot \chi_{v}\left(\frac{\sigma_v(T(ma))}{\phi_m(a)}\right),
\end{equation*}
for any positive integer $m$ such that $ma\in H_v\setminus\supp(\Theta)$ (see \cite[Definition 4.2]{Bia23}).

\begin{prop}[{\cite[Theorem~5.30]{Bia23}}]\label{P:eqGsigma}
There exists a Colmez symmetric $v$-adic Green function $G_{\Theta, W_v}$ such that
 \begin{equation*}
 \lambda^B_{v} = -\frac{2}{[F_v:\Q_p]}\tr_{v}\circ G_{\Theta, W_v}\,.
 \end{equation*}
\end{prop}

\begin{rem}
Unlike in \cite{colmez1998integration} and \cite{Bia23}, our definition of
  $G_{\Theta,W_v}$, introduced in \S\ref{subsec:Colmez}, depends on the choice of branch of the $p$-adic logarithm induced by $\chi_v$. This explains the different trace map in Proposition~\ref{P:eqGsigma} compared to \cite[Theorem~5.30]{Bia23}.
\end{rem}
The following result
justifies why in \cite{Bia23} the function $\lambda_v^{B}$ is called ``a $p$-adic N\'eron function with respect to $X$''.
 \begin{cor} \label{cor:lambdavlambdaX}
 We have the equality
 \begin{equation*}
[F_v:\Q_p]\lambda^B_v = \lambda_{X,v}.
 \end{equation*}
 That is, the function $[F_v:\Q_p]\lambda_v^{B}$ is the $p$-adic N\'eron
   function with respect to $X$, $\chi_v$ and $W_v$ that satisfies~\eqref{eq:KU_divpol}.
 \end{cor}

  \begin{proof}
 By Proposition~\ref{P:pNF} \eqref{P:pNF-add} (recall that $X = 2\Theta$), Proposition~\ref{P:pNFSymGreen} and Proposition~\ref{P:eqGsigma}, we have that $[F_v:\Q_p]\lambda^B_v - \lambda_{X,v}$ is constant. By our choice of $\lambda_{X,v}$ and \cite[Proposition~4.7\thinspace{}(ii)]{Bia23}, both $[F_v:\Q_p]\lambda^B_v$ and $\lambda_{X,v}$ satisfy~\eqref{eq:KU_divpol}, so the constant is zero. 
\end{proof}
By \cite{Bia23}, we can express the local Coleman--Gross height pairing at $v$, for divisors with disjoint support, in terms of the local height $\lambda_v^B$, and hence, in terms of $\lambda_{X,v}$ in view of Corollary~\ref{cor:lambdavlambdaX}. In particular, we have:
\begin{cor}\label{cor:comp_disjoint}
Let $P_1,P_2,Q_1,Q_2\in C(F_v)\setminus \{\infty\}$ such that $Q_i\neq P_j$ for all $i,j\in\{1,2\}$. 
Then we have 
\begin{equation*}
  h_v^{\CG}(P_1 - P_2,Q_1 - Q_2) = -\frac{1}{2}\sum_{1\leq i,j\leq 2} (-1)^{i+j}\lambda_{X,v}([Q_i - P_j]).
\end{equation*}
\end{cor}
\begin{proof}
    This follows immediately from Corollary~\ref{cor:lambdavlambdaX} and \cite[Corollary~5.32, Example 5.33]{Bia23}.
\end{proof}

\subsubsection{Comparison theorem at ramified primes}
Recall that $\chi_v$ is assumed to be ramified. We shall now use
Corollary~\ref{cor:comp_disjoint} to compare $\lambda_{X,v}$ with
$\lambda_v^{\CG}$.

\begin{thm}\label{thm:comp_above}
We have
    \begin{equation*}
        \lambda_v^{\CG} = \lambda_{X,v}\,.  
    \end{equation*}
        \end{thm}

\begin{proof}
  Let $a = [Q_1-Q_2]\in J(F_v)\setminus\supp(\Theta) $ for some $Q_1,Q_2 \in C(\overline{F}_v)\setminus\{\infty\}$. By \S\ref{subsec:common-nonrat} and Proposition~\ref{P:pNF}~\eqref{P:pNF-ext}, we may assume, possibly after moving to an extension, that $Q_1$ and $Q_2$ are $F_v$-rational. 

We can extend the equality of Corollary~\ref{cor:comp_disjoint} to the case $P_i = Q_i$ using the section $t$ of the tangent bundle of $C$ that we fixed at the beginning of \S\ref{S:com}. 
Since $\lambda_{X,v}$ is even, this yields that
\begin{equation*}
  \lambda_v^{\CG}(a) - \lambda_{X,v}(a)
\end{equation*}
is given by the sum of the values of
\begin{equation}\label{eq:rem_terms}
-\frac{1}{2}\lambda_{X,v}([P_1-Q_1]) , \qquad -\frac{1}{2}\lambda_{X,v}([P_2 - Q_2])
\end{equation}
at $Q_1= P_1$ and $Q_2= P_2$, where these are defined using our choice of $t$, similarly to \S\ref{subsec:above-common}. Let $z$ be the normalised local parameter at $Q_i$ with respect to $t$ of Lemma~\ref{lemma:T1T2}. Then
\begin{align*}
    T_1([ Q_i(z)-Q_i]) = z + O(z^2),\qquad  
    T_2([Q_i(z)-Q_i]) = x(Q_i) z + O(z^2)\,,
\end{align*}
and, for $z\ne 0$ sufficiently small so that $[Q_i(z) - Q_i]\in H_v$,
\begin{equation*}
\lambda_{X,v}([Q_i(z) - Q_i]) = -2\chi_v(\sigma_v([Q_i(z) - Q_i])) = -2\chi_v(z(1+ O(z))) = -2\tr_v(\log_v(z) + O(z)).
\end{equation*}
Since the constant term of $\log_v(z) + O(z)$ is zero, the terms in~\eqref{eq:rem_terms} are zero.
\end{proof}

\begin{lem}\label{lemma:T1T2}
Let $P\in C(F_v)\setminus\{\infty\}$. Then we can make the following choice of a local parameter with respect to $\omega_1$ (and hence $t$):
\begin{enumerate}[label=(\roman*)]
    \item If $P$ is a Weierstrass point (i.e.\ $y(P) = 0$), then $z =
      \frac{y}{b^{\prime}(x(P))}$, where $C$ is defined by $y^2=b(x)$.
    \item \label{it:gen}If $P$ is a non-Weierstrass point, then $z = \frac{x-x(P)}{2y(P)}$.
\end{enumerate}
In both cases, if $P(z)$ denotes the expansion around $P$ in $z$, we have
\begin{align*}
  T_1([P(z)-P]) = z + O(z^2),\qquad  T_2([P(z) -P])= x(P) z + O(z^2)\,.
\end{align*}
Moreover, there exists a neighbourhood $V$ of $0$ for which $\sigma_v(T_1(z),T_2(z))$ converges for all $\alpha\in V$ and its value at such $\alpha$ is $\sigma_v(T_1(\alpha),T_2(\alpha))$.
\end{lem}

\begin{proof}
In each case, it is clear that the given $z$ is a uniformiser at $P$.

Let us consider~\ref{it:gen} first. Then $P(z) = (x(z),y(z))$ is given by
  $x(z) = x(P) + 2y(P)z$ and $y(z)\in F_v[[z]]$ is the unique power series
  of the form $y(z) = y(P) + O(z)$ such that $y(z)^2 = b(x(z))$ (recall
  that $C$ is given by the equation $y^2=b(x)$). We have
\begin{equation*}
\omega_1(z) = \frac{1}{2y(z)}\frac{d x(z)}{dz} dz = \frac{2y(P)}{2y(P)+O(z)}dz = (1+O(z))dz,
\end{equation*}
so $z$ is normalised with respect to $\omega_1$. Using the explicit
  formulas for $T_i$ in terms of the $x$ and $y$ coordinates of the image
  of $P$ under the hyperelliptic involution and of
  $P(z)$ that one can deduce from \cite[Theorem~4.2, (4.1),
  (1.4)]{Grant1990}, a bit of algebra shows the claim on the leading terms
  of $T_1$ and $T_2$; see \cite[Lemma~B.1]{Bia23} for a similar
  computation. To see that the radius of convergence of the power series
  for $T_1$ and $T_2$ is non-zero, it suffices to observe that from $y(z)^2
  = b(x(z))$ we may deduce that the valuation of the $n$-th coefficient of
  $y(z)$ is bounded from below by a linear polynomial in $n$. Since the
  sum, product and inverse (if defined at all) of power series whose
  coefficients have valuation bounded from below by a linear polynomial in the degree have the same property -- with a possibly different polynomial -- we see that $T_1$ and $T_2$ do too. Then for $\alpha$ with large enough valuation, $T_1(\alpha), T_2(\alpha)$ converge and are such that $\sigma_v(T_1(\alpha),T_2(\alpha))$ converges (see \cite[Proposition~3.4\thinspace{}(i)]{Bia23} and its proof). Moreover, in this case, this value is equal to the value at $\alpha$ of $\sigma_v(T_1(z),T_2(z))$ (see for instance  \cite[\S 2 Substitution Theorem]{mattuck}).

The Weierstrass case is analogous, except that we substitute $y(z) = b^{\prime}(x(P))z$ into $y^2 =b(x)$ and solve for $x$, which gives $x(z) = x(P) + b^{\prime}(x(P))z^2 + O(z^4)\in  F_v[[z]]$. 
\end{proof}
\begin{rem}
Let $v$ be a prime at which the character $\chi$ is unramified. For
  divisors with disjoint support, Colmez \cite{colmez1998integration} gives
  an alternative analytic definition of the local height pairing of
  Coleman--Gross at $v$ (see also \cite[Theorem~5.28]{Bia23}). Since Corollary~\ref{cor:comp_disjoint} still holds \cite[Corollary~5.32]{Bia23}, with a suitable extension of the notion of constant term introduced in \S\ref{subsec:above-common}, it should be possible to use Colmez's definition to give a unified proof of Theorem~\ref{thm:comp_away} and Theorem~\ref{thm:comp_above}.
\end{rem}

\section{Canonical theta and zeta functions in genus~$2$}\label{S:CanThetaZeta}

We continue to denote by $C/F$ a smooth projective genus~$2$ curve given by an affine equation 
\eqref{Ceqn}, with Jacobian $J/F$, and we let $v$ be a finite place of $F$. 

\begin{defn}\label{D:CSemOrd}
We say that $C$ has \textit{semistable reduction} at $v$ if the subscheme
  of $\P^{1,3,1}_{\O_v}$ defined by the degree-$6$ homogenisation
  of~\eqref{Ceqn} is semistable. We say that $C$ has \textit{semistable
  ordinary reduction} at $v$ if $C$ has semistable reduction at $v$ and if
  $J$ has ordinary reduction at $v$. 
\end{defn}

In this section, we show that  Conjecture~\ref{C:TNF1} holds in the situation that $C$ has semistable ordinary reduction at $v$; see Corollary~\ref{C:g2TNF1}. This implies that the canonical Mazur--Tate $\chi_v$-splitting is analytic. Moreover, if $C$ has semistable ordinary reduction at all primes $v$ such that $\chi_v$ is ramified, then the canonical Mazur--Tate height is analytic.

The main ingredient of the proof is to show that the Norman $v$-adic theta function \cite{norman1985, norman1986} with respect to $X= 2\Theta$
is, up to sign, the square of a $v$-adic sigma function (Definition~\ref{def:sigma}). 
The former induces the canonical 
Mazur--Tate splitting by Proposition~\ref{P:papa}, whereas the latter gives a $p$-adic N\'eron function by Corollary~\ref{cor:lambdavlambdaX}.

Having shown that the canonical Mazur--Tate $\chi_v$-splitting is analytic,
in the second part of the section we give an explicit description of the
subspace $W_v$ that it corresponds to; see Theorem~\ref{thm:constants}. In
future work, we will describe an algorithm to compute this subspace. If $C$
has semistable ordinary reduction at all primes $v$ of ramifications for
$\chi$,
by the above together with \cite[Corollary~5.35]{Bia23}, the canonical Mazur--Tate height is then equal to the Coleman--Gross height with respect to the subspaces $W_v$ described in Theorem~\ref{thm:constants}. Combining the algorithm for the computation of $W_v$ with algorithms of \cite{Bia23} and algorithms to compute Vologodsky integrals of differentials of the third kind building on \cite{KatzKaya, Kaya} respectively, we obtain two different methods for computing the canonical Mazur--Tate height.

Some of the techniques used in this section are very similar (and sometimes identical) to those of \cite[\S3, \S4]{blakestadsthesis}. In fact, we essentially extend to semistable reduction some of Blakestad's results for the good reduction case. 
\begin{rem}\label{R:}
If $C$ has semistable ordinary reduction at $v$, then so does $J$, but the
converse is not true in general. In fact, it is possible that $J$ has semistable
  reduction at $v$, so that there is a semistable model of $C$ over $\O_v$, but there is no semistable model of the
  form~\eqref{Ceqn} over $\O_v$, nor over any finite extension. 
  This can happen when $C$ has reduction
  type $[I_n-I_m-\ell]$ with $\ell>0$ in the notation of
  Namikawa--Ueno~\cite{NU73}. See~\cite[\S5, \S10]{muller2016canonical} for
  a related discussion.
\end{rem}

\subsection{Analyticity of the canonical Mazur--Tate splitting}\label{S:Norman}

Suppose that $v\nmid p \geq 3$ and that $C$ has semistable ordinary
reduction at $v$.

Let $\mathcal{J}^f$ be the formal group of the N\'eron model $\mathcal{J}$
of $J$. It is isomorphic to the formal group of Grant's model of $J$
\cite{Grant1990} induced by~\eqref{Ceqn}, that is
\[\mathcal{J}^f \simeq \mathrm{Specf}(\O_v[[T_1,T_2]])\]
where $T_1,T_2$ are as in \S\ref{subsec:vadicsigma}. This follows from \cite[Chapter~9.7, Corollary~2]{BLR90} and the paragraph after it. 

The divisor $X = 2\Theta$ is totally symmetric in the sense of~\cite[Assumption 5.1]{papanikolas} (see
also \cite[Remark 5.2]{papanikolas}), so the construction of Norman applies
and produces a function on $\mathcal{J}^f(\overline{\O}_v)$, where
$\overline{\O}_v$ is the valuation ring of a fixed algebraic closure of
$F_v$. We only give a brief description and refer the reader to
\cite[\S5]{papanikolas} for further details. 

The choices of the parameters $T_1,T_2$ yield a point $\gamma$ in
$\mathcal{J}^f$. Now let $\varepsilon = [2]^{-1}\gamma$, which makes sense
as $p\neq 2$. Since $X$ is a totally symmetric (Cartier) divisor, it can be
written as $\{(U_i,f_i)\}$ where $X\vert_{U_i} = \div(f_i)\vert_{U_i}$,
$[-1]^*U_i = U_i$ and $[-1]^*f_i = f_i$. For $\ell,m > 0$, one constructs a
certain function $\hat{F}_{\ell,m}$ on $\mathcal{J}^f$ by prescribing
congruence conditions on its divisor, and then sets
\[\tilde{\theta}_{X,\ell,m}(z) = \big(\hat{F}_{\ell,m}(\delta + \varepsilon)^{-1}\hat{F}_{\ell,m}(\delta)^{-1}f_i(\delta)\big)\big\vert_{\delta = 0}.\]
The limit of these functions exist, and defines Norman's theta function:
\[\tilde{\theta}_{X}(z) = \lim_{\ell,m\to\infty}\tilde{\theta}_{X,\ell,m}(z).\]
In terms of the parameters $T = (T_1,T_2)$, this function is expressible as a power series $\tilde{\theta}_X(T)\in F_v^{\times}\cdot\O_v[[T]]$. In fact, $\tilde{\theta}_X$ is defined only up to multiplication by a constant in $F_v^{\times}$. We will fix a suitable normalisation later (see Definition~\ref{defn:normalising_theta}).

Given a point $u\in \mathcal{J}^f(\overline{\O}_v)$, let $X_u$ be the image
of $X$ under translation by $u$. Let $D\in \Div_a(J)$ be a divisor
algebraically equivalent to $0$. Then there are $m_u\in \Z$ so that we can
write
$D = \sum_{u\in \mathcal{J}^f(\overline{\O}_v)} m_u X_u$. 
For $w\in \mathcal{J}^f(\overline{\O}_v)$ not in the support of $D$, we set
\begin{equation}
\tilde{\theta}_D(w) = \prod_u \tilde{\theta}_X (w - u)^{m_{u}}
\end{equation}
(this is well-defined).

\begin{prop}
\label{prop:theta_of_lin_eq_0}
If $D$ is linearly equivalent to $0$, there is a choice of rational function $f$ such that $D = \mathrm{div}(f)$ and $\tilde{\theta}_D(T) = f(T)$.
\end{prop}
\begin{proof}
This follows from the proof of \cite[Theorem~5.4\thinspace{}(b)]{papanikolas}. 
\end{proof}

\begin{prop}
\label{prop:even}
$\tilde{\theta}_X$ is an even function.
\end{prop}

\begin{proof}
This follows from \cite[(13), (14)]{papanikolas}. Indeed, the function $\hat{F}_{\ell,m}$ depends on the
  point $\gamma$ discussed above, and we indicate this dependency by adding a
  superscript. Now, since $X$ is symmetric, we may take
  $\hat{F}_{\ell,m}^{-\gamma} = [-1]^*\hat{F}_{\ell,m}^{\gamma}$. Note that the fact
  that $F_{\ell, m, n}^{\gamma}$, with notation as in \emph{loc.\ cit.}, satisfies \cite[(13)]{papanikolas} gives  the
  analogous result for $F_{\ell,m,n}^{-\gamma}\colonequals [-1]^*F_{\ell,m,n}^\gamma$. Then
\begin{align*}
\frac{\tilde{\theta}_X(z)}{\tilde{\theta}_X(-z)} &= \lim_{\ell, m\to \infty}
  \frac{\hat{F}_{\ell,m}^{-\gamma}(\delta -
  \varepsilon)\hat{F}_{\ell,m}^{-\gamma}(\delta)}{\hat{F}_{\ell,m}^{\gamma}(\delta +
  \varepsilon)\hat{F}_{\ell,m}^{\gamma}(\delta)}\bigg\vert_{\delta=0}\\
&= \lim_{\ell, m\to \infty} \frac{\hat{F}_{\ell,m}^{\gamma}(-\delta +
  \varepsilon)\hat{F}_{\ell,m}^{\gamma}(-\delta)}{\hat{F}_{\ell,m}^{\gamma}(\delta +
  \varepsilon)\hat{F}_{\ell,m}^{\gamma}(\delta)}\bigg\vert_{\delta = 0 } = 1. \qedhere
\end{align*}
\end{proof}

Let $m,s,p_1,p_2\colon J\times J \to J$ be defined by $m(u,w) = u+w$, $s(u,w) = u -w$, $p_1(u,w) = u$, $p_2(u,w) = w$. 
\begin{prop}
\label{prop:quasi_par_formula}
The function 
\begin{equation}
\label{eq:prod_sums}
  (u,w)\mapsto 
\frac{\tilde{\theta}_X(u+w)\tilde{\theta}_X(u-w)}{\tilde{\theta}_X(u)^2\tilde{\theta}_X(w)^2}
\end{equation}
is the restriction to $\mathcal{J}^f\times \mathcal{J}^f$ of a rational function with divisor
\begin{equation}
\label{eq:div}
m^{*} X + s^{*} X - 2p_1^{*}X - 2p_2^{*}X.
\end{equation}
\end{prop}

\begin{proof}
This is the analogue of \cite[Proposition~17]{blakestadsthesis}, which we may follow very closely; see also \cite{arledge_grant} for similar arguments. First of all, there exists a function $\psi$ on $J\times J$ with divisor \eqref{eq:div} by the theorem of the square and the seesaw principle (see, for instance, the introduction to \cite{arledge_grant}).

Fix a point $u\in \mathcal{J}^f(\overline{\O}_v)\setminus \supp(\Theta)$. By Proposition~\ref{prop:theta_of_lin_eq_0} with $D = X_u + X_{-u} - 2X$, we have
\begin{equation}
\label{eq:on_u}
\frac{\tilde{\theta}_X(u+w)\tilde{\theta}_X(w-u)}{\tilde{\theta}_X(u)^2\tilde{\theta}_X(w)^2}
 = c_u\psi(u,w),
\end{equation} 
for some non-zero constant $c_u$ depending on $u$ only. Similarly, for a fixed point $w\in \mathcal{J}^f(\overline{\O}_v) \setminus \supp(\Theta)$, applying Proposition~\ref{prop:theta_of_lin_eq_0} to $D =  X_w + X_{-w} - 2X$  gives that
\begin{equation}
\label{eq:on_v}
\frac{\tilde{\theta}_X(u+w)\tilde{\theta}_X(u-w)}{\tilde{\theta}_X(u)^2\tilde{\theta}_X(w)^2}
 = d_w\psi(u,w),
\end{equation} 
for some non-zero $d_w$ depending on $w$ only. Since $\tilde{\theta}_X$ is even by Proposition~\ref{prop:even}, the expressions of \eqref{eq:on_u} and \eqref{eq:on_v} are equal for all $u,w$. Thus $c_u = d_w$ for every $u,w$.  
\end{proof}

The proposition holds more generally for an abelian variety with semistable ordinary reduction at $v$ and $X$ a divisor satisfying \cite[Assumption 5.1]{papanikolas}. We next make the rational function explicit for $J$, in terms of the rational functions $X_{ij}$  of \cite{Grant1990} that appeared in \S\ref{subsec:vadicsigma}.

\begin{prop}
\label{prop:fct_explicit}
The function of Proposition~\ref{prop:quasi_par_formula} is
\begin{equation}
\label{eq:explicit_fct}
c(X_{11}(u)  - X_{11}(w) + X_{12}(u)X_{22}(w) - X_{12}(w)X_{22}(u))^2
\end{equation}
for some $c\in F_v^{\times}$.
\end{prop}

\begin{proof}
Since $X = 2\Theta$, by \cite[Examples 2]{arledge_grant}, the function \eqref{eq:explicit_fct} has the right divisor.
\end{proof}

Recall that $\tilde{\theta}_X$ is well-defined only up to multiplication by an element of $F_v^{\times}$. We finally fix a choice of normalisation: namely, we want our theta function to satisfy Proposition~\ref{prop:fct_explicit} with $c=1$. Since rescaling $\tilde{\theta}_X$ by $d$ rescales \eqref{eq:prod_sums} by $d^{-2}$, in principle such a normalisation can only be achieved over a quadratic extension of $F_v$. (Corollary~\ref{cor:theta_sigma2} below implies that the normalisation is in fact defined over $F_v$.)
\begin{defn}
\label{defn:normalising_theta}
Let $\theta_X\in \overline{F}_v^{\times}\cdot \O_v[[T_1,T_2]]$ be a scalar multiple of Norman's theta function $\tilde{\theta}_X$ is the sense above such that $\theta_X$ satisfies Proposition~\ref{prop:fct_explicit} with $c=1$.
\end{defn}

\begin{rem}
$\theta_X$ is unique up to multiplication by $\pm 1$.
\end{rem}
Recall the invariant derivation $D_1, D_2$ on $J$ that we fixed in \S\ref{subsec:vadicsigma}.
\begin{prop}
\label{prop:diff_eq_theta_X}
For $i,j\in \{1,2\}$, we have
\begin{equation*}
D_iD_j(\log(\theta_X)) = -2X_{ij} + 2f_{ij},\qquad \text{for some } f_{ij}\in F_v,
\end{equation*}
with $f_{12} = f_{21}$.
\end{prop}

\begin{proof}
This is the analogue of \cite[Proposition~34]{blakestadsthesis} for Blakestad's $v$-adic sigma function and the proof is identical. As for the sigma function, an easy computation shows
\begin{equation*}
(D_i^u - D_i^w)(D_j^u+D_j^w)\log\left(\frac{\theta_X(u+w)\theta_X(u-w)}{\theta_X(u)^2\theta_X(w)^2}\right) = 2 D_iD_j(\log(\theta_X(w))) - 2 D_iD_j(\log(\theta_X(u))).
\end{equation*} 
Now, as in the proof of \cite[Proposition~34]{blakestadsthesis}, we have that\footnote{There are two sign typos in \cite[Proposition~33]{blakestadsthesis} which also carry over to the proof of \cite[Proposition~34]{blakestadsthesis}.}
\begin{equation*}
2(D_i^u - D_i^w)(D_j^u+D_j^w)\log\left(X_{11}(w)  - X_{11}(u) - X_{12}(u)X_{22}(w) + X_{12}(w)X_{22}(u)\right)
\end{equation*}
is equal to  $4X_{ij}(u) - 4X_{ij}(w)$.
By Proposition~\ref{prop:fct_explicit}, we deduce that $D_iD_j(\log(\theta_X)) + 2X_{ij}$ is constant. The constant is easily seen to be in $F_v$, since there is a scalar multiple of $\theta_X$ which belongs to $ \O_v[[T_1,T_2]]$ and the expansion of $X_{ij}$ in $T_1,T_2$ has $F_v$-rational coefficients.
\end{proof}

We now prove that the power series $\theta_X(T)$ is uniquely determined, up to sign, by some of its properties.

\begin{prop}
\label{prop:uniqueness}
Let $f_{ij}$ be as in Proposition~\ref{prop:diff_eq_theta_X}. Up to multiplication by $\pm 1$, there exists a unique power series $G\in \overline{F}_v[[T_1,T_2]]$ satisfying 
\begin{enumerate}[label = (\roman*)]
    \item\label{it:even} $G$ is an even function on  $\mathcal{J}^f$,
    \item\label{it:system} $D_iD_j(\log(G)) = - 2X_{ij} + 2f_{ij}$, and
    \item\label{prope:prod_sum} $\frac{G(u+w)G(u-w)}{G(u)^2G(w)^2} = (X_{11}(u)  - X_{11}(w) + X_{12}(u)X_{22}(w) - X_{12}(w)X_{22}(u))^2$.
\end{enumerate}
\end{prop}

\begin{proof}
Propositions~\ref{prop:even},~\ref{prop:fct_explicit}
  and~\ref{prop:diff_eq_theta_X} show existence. For uniqueness, given
  $G_1,G_2$, the system
\begin{align*}
    D_1(F) =G_1, \qquad D_2(F) = G_2
\end{align*}
has at most one solution $F$, up to addition by a constant. Evenness of $G$
  implies oddness of $D_j(\log(G))$, so the constant term in $D_j(\log(G))$ is zero. Similarly, property~\ref{prope:prod_sum} determines the final constant of integration.
\end{proof}

\begin{cor}\label{cor:theta_sigma2}
We have that $\theta_X = \pm \sigma_v^2$ where $\sigma_v$ is the $v$-adic
  sigma function with respect to the subspace $W_v$ of $H^1_{\dR}(J/F_v)$
  corresponding to the symmetric matrix $f=(f_{ij})_{ij}$ from
  Proposition~\ref{prop:diff_eq_theta_X}.
\end{cor}

\begin{proof}
By Definition~\ref{def:sigma}, the function $\sigma_v^2$ satisfies properties~\ref{it:even} and~\ref{it:system} of Proposition~\ref{prop:uniqueness}. By \cite[Proposition~3.4\thinspace{}(iii)]{Bia23}, it also satisfies~\ref{prope:prod_sum}.
\end{proof}

\begin{thm}\label{T:thetaneron}
Suppose that $C$ has semistable ordinary reduction at $v$. Let $W_v$ be the subspace of $H^1_{\dR}(J/F_v)$ corresponding to the
  symmetric matrix $f =(f_{ij})_{ij}$.
The
$p$-adic N\'eron function with respect to the divisor $X$, the local character $\chi_v$ and the subspace 
$W_v$, normalised so that \eqref{eq:KU_divpol} holds, satisfies  
\begin{equation}\label{}
  \lambda_{X,v}(u) = -\chi_v(\theta_X(u))
\end{equation}
  for all $u\in \mathcal{J}^f(\O_v)\setminus\supp(\Theta)$.
\end{thm}
\begin{proof}
 The result follows from Corollary~\ref{cor:theta_sigma2} and Corollary~\ref{cor:lambdavlambdaX}.
\end{proof}
\begin{cor}\label{C:g2TNF1}
If $C$ has semistable ordinary reduction at $v$, then Conjecture~\ref{C:TNF1} holds for $J$.
\end{cor}
Corollary~\ref{cor:theta_sigma2} states that, up to multiplication by $\pm
1$, the normalised theta function $\theta_X$ is the square of a $v$-adic
sigma function. When~\eqref{Ceqn} has good reduction at $v$, $J$ has good
ordinary reduction and we have $p\geq 5$, Blakestad \cite{blakestadsthesis} constructed a canonical $v$-adic sigma function, which by \cite[Proposition~3.8]{Bia23} is the $v$-adic sigma function that corresponds to the unit root eigenspace of Frobenius. We next show that, in the good ordinary case, the sigma function related to $\theta_X$ is in fact Blakestad's.

\begin{lem}
\label{lem:over_Ov}
$\theta_X = \pm T_1^2(1+O(T_1,T_2)^2)\in\O_v[[T_1,T_2]]$.
\end{lem} 

\begin{proof}
Since $\theta_X\in F_v^{\times}\cdot \O_v[[T]]$, it converges for all $T\in \overline{F}_v^2$ with $\min\{\ord_v(T_i)\}>0$.
Since $\theta_X$ is, up to multiplication by $\pm 1$, the square of a $v$-adic sigma function, it vanishes to order $2$ at $T_1=0$, and is non-zero everywhere else in its domain of convergence (see \cite[Proposition~3.4\thinspace{}(ii)]{Bia23}).
Recall that 
  $$\theta_X = \pm T_1^2 u(T)\in F_v^{\times} \cdot
  \O_v[[T]]\,,\quad\text{with}\; u(T) = 1+O(T_1,T_2)^2\,.$$ So $u(T)$ is non-zero on $\{(T_1,T_2): \ord_v(T_i)>0\}$. 

Assume that $u(T)\not\in \O_v[[T]]$. Consider $u_1(T_1) = u(T_1,0) = 1 + O(T_1)^2$. If $u_1(T_1)\not \in \O_v[[T_1]]$, but $cu_1(T_1)\in \O_v[[T_1]]$ for some $c\neq 0$, then $u_1(T_1)$ has some zeros of positive valuation by the $p$-adic Weierstrass preparation theorem, a contradiction.

Write $u(T) = \sum_{k,\ell} a_{k,\ell}
  T_1^kT_2^{\ell}$. In view of the previous step, we deduce that there
    exists $\ell \in \Z_{> 0}$ such that $h(T_1) = \sum_{k} a_{k,\ell}
    T_1^k\in F_v \cdot \O_v[[T_1]]\setminus \O_v[[T_1]]$. By the $p$-adic
    Weierstrass preparation theorem, we see that we can choose $T_1^{(0)}$
    of positive valuation such that $\ord_v(h(T_1^{(0)}))< 0$. Applying
    once again the $p$-adic Weierstrass preparation theorem, 
    $u(T_1^{(0)},T_2)$ has zeros of positive valuation as a series in
    $T_2$, another contradiction. Therefore $u(T)\in \O_v[[T]]$.
\end{proof}

\begin{cor}\label{C:g2TNF2}
  If the model~\eqref{Ceqn} has good reduction at $v$, $J$ has good
  ordinary reduction at $v$ and $p\geq 5$, then $\theta_X$ is, up to multiplication by $\pm 1$, the square of Blakestad's $v$-adic sigma function. If in addition $F_v/\Q_p$ is unramified, then Conjecture~\ref{C:TNF2} holds for $J$.
\end{cor}
\begin{proof}
    The first claim follows from Lemma~\ref{lem:over_Ov} in view of \cite[Propositions~34~and~27, Corollary~37]{blakestadsthesis}; see also \cite[Theorem~3.6]{Bia23}. 
    By \cite[Proposition~3.8]{Bia23}, the corresponding subspace is the unit root eigenspace of Frobenius.
\end{proof}
In general, by Corollary~\ref{cor:theta_sigma2}, the canonical Mazur--Tate splitting corresponds to choosing $W_v$ so that $\iota^{*}W_v\subset H^1_{\dR}(C/F_v)$ is spanned by the classes of the differentials
\begin{equation*}
\begin{aligned}
\eta_1^{(f)} = \left(-3x^3 - 2b_1x^2 -b_2x + f_{12} x + f_{11}\right)\frac{dx}{2y},\\
 \eta_2^{(f)}= \left(-x^2 + f_{22} x + f_{12}\right)\frac{dx}{2y}.
 \end{aligned}
\end{equation*}
\begin{rem}\label{R:}
Under the assumptions of Corollary~\ref{C:g2TNF2}, the classes of $\eta_1^{(f)}$ and $\eta_2^{(f)}$ span the unit root eigenspace of Frobenius. It might be possible to extend this result to semistable
ordinary reduction using
\cite{katz1981crystalline, iovita2000formal}.  
We instead extend the explicit description of the constants $f_{ij}$ provided by Blakestad \cite{blakestadsthesis} in the good ordinary case to the semistable ordinary setting. 
\end{rem}

\subsection{Blakestad's Zeta Functions}\label{subsec:zeta_fcts}

We now extend the construction of Blakestad's zeta functions
\cite[Chapter~3]{blakestadsthesis} to our setting. It is easily checked
that all of his results, including proofs, remain valid except the proof of
\cite[Proposition~23]{blakestadsthesis} which we adapt slightly, see Lemma~\ref{lem:CongForABCDIJRS}. Assume that $p\geq 5$ and $C$ has semistable reduction at $v$.

Set $t = -\frac{x^2}{y}$; it is a local parameter at the point $\infty\in
C(F)$. In~\cite[\S3.2.1]{blakestadsthesis}, Blakestad constructs, for every positive integer $n$, functions $\phi_n$ and $\psi_n$ on $C$, regular away from $\infty$, with $t$-expansions
\begin{align}\label{BlakestadPhinPsin}
\begin{split}
\phi_n &= \frac{1}{t^{3p^n}} + \frac{A_n}{t^3} + \frac{B_n}{t} + I_nt + R_nt^3 + O(t^5), \\
\psi_n &= \frac{1}{t^{p^n}} + \frac{C_n}{t^3} + \frac{D_n}{t} + J_nt + S_nt^3 + O(t^5),
\end{split}
\end{align}
for some $A_n,B_n,C_n,D_n,I_n,J_n,R_n,S_n$ in $\O_F$. 

\begin{lem}\label{lem:CongForABCDIJRS}
For every positive integer $n$, there is a congruence
\[
\begin{pmatrix}
A_{n+1} & B_{n+1} & I_{n+1} & R_{n+1} \\
C_{n+1} & D_{n+1} & J_{n+1} & S_{n+1}
\end{pmatrix}
\equiv 
-\begin{pmatrix}
A_{n}^p & B_{n}^p \\
C_{n}^p & D_{n}^p
\end{pmatrix}
\begin{pmatrix}
A_{n} & B_{n} & I_{n} & R_{n} \\
C_{n} & D_{n} & J_{n} & S_{n}
\end{pmatrix}
(\mod\ \pi_v),
\]
and, therefore we have
\[
\begin{vmatrix}
A_{n} & B_{n} \\ 
C_{n} & D_{n}
\end{vmatrix}
\equiv
\begin{vmatrix}
A_{1} & B_{1} \\ 
C_{1} & D_{1}
\end{vmatrix}^{p^{n-1} + \cdots + p + 1}
(\mod\ \pi_v).
\]
\end{lem}

\begin{proof} 
Write $\tilde{C}$ for the reduction of $C$ modulo $\pi_v$. The first claim
  is due to Blakestad when $C$ has good reduction at $v$; see
  \cite[Proposition~23]{blakestadsthesis}. The key features of his proof
  are the Riemann--Roch space dimension counts 
\begin{equation}\label{eq:DimCounts}
  \dim\mathcal{L}(3\tilde{\infty}) = \dim\mathcal{L}(2\tilde{\infty})\ \ \
  \text{and} \ \ \ \dim\mathcal{L}(\tilde{\infty}) = \dim\mathcal{L}(0),     
\end{equation}
where $\tilde{\infty}$ is the point at infinity on $\tilde{C}$, which
  follow from the Riemann--Roch theorem on the good reduction curve
  $\tilde{C}$. On the other hand, even if $\tilde{C}$ is singular, the
  point $\tilde{\infty}$ is  smooth because of the
  shape of the defining equation of $C$, and hence the equalities in
  \eqref{eq:DimCounts} are still true. In other words, Blakestad's proof
  works almost verbatim in our situation.

The second claim follows from the first one as in the proof of
  \cite[Corollary~24]{blakestadsthesis}.
\end{proof}

For each positive integer $n$, set 
\begin{equation}
\label{BlakestadHn}
H_n \coloneqq A_nD_n - B_nC_n\in \O_F.  
\end{equation}
By Lemma~\ref{lem:CongForABCDIJRS}, we have 
\[H_n\equiv H_1^{\frac{p^n-1}{p-1}}(\mod\ \pi_v).\]
From now on, we
make the assumption that $H_1$ is invertible in $\O_v$. By
Proposition~\ref{H1ord}, this is equivalent to saying that $C$ has semistable
ordinary reduction at $v$. Our assumption clearly implies that $H_n$ is invertible in $\O_v$ for each $n$.

Define functions $\zeta_{1,n},\zeta_{2,n}$ and constants $\alpha_n,\beta_n,\gamma_n,\delta_n$ by
\[\begin{pmatrix}
\zeta_{1,n} \\
\zeta_{2,n}
\end{pmatrix} \colonequals \begin{pmatrix}
A_n & B_n \\
C_n & D_n
\end{pmatrix}^{-1} \begin{pmatrix}
\phi_n \\
\psi_n
\end{pmatrix},\ \quad \ \begin{pmatrix}
\alpha_n & \delta_n \\
\beta_n & \gamma_n
\end{pmatrix} \colonequals \begin{pmatrix}
A_n & B_n \\
C_n & D_n
\end{pmatrix}^{-1} \begin{pmatrix}
I_n & R_n \\
J_n & S_n
\end{pmatrix}.\]
Looking at $t$-expansions modulo $\pi_v^n$, we get
\begin{align*}
-d\zeta_{1,n} \equiv \left(3x^3 + 3b_1x^2 - \alpha_nx - (3b_1b_2 - b_1\alpha_n + 3b_3 + 3\delta_n)\right)\frac{dx}{2y}, \\
-d\zeta_{2,n} \equiv \left(x^2 - \beta_nx - (b_2 - b_1\beta_n + 3\gamma_n)\right)\frac{dx}{2y}.
\end{align*}
We have
\[\alpha_n\to \alpha,\ \ \ \beta_n\to \beta,\ \ \ \gamma_n\to \gamma,\ \ \ \delta_n\to \delta\]
for some constants $\alpha,\beta,\gamma,\delta \in \O_v$, and
\[\zeta_{1,n}\to \zeta_1,\ \ \ \zeta_{2,n}\to \zeta_2\]
for some Laurent series $\zeta_1\in \frac{1}{t^3} + \O_v[[t]]$ and
$\zeta_2\in \frac{1}{t} + \O_v[[t]]$.
This implies that
\begin{align*}
-d\zeta_{1} = \left(3x^3 + 3b_1x^2 - \alpha x - (3b_1b_2 - b_1\alpha + 3b_3 + 3\delta)\right)\frac{dx}{2y}, \\
-d\zeta_{2} = \left(x^2 - \beta x - (b_2 - b_1\beta + 3\gamma)\right)\frac{dx}{2y}.
\end{align*}
We define two differentials $\eta_1\colonequals d\zeta_1 - b_1 d\zeta_2$ and $\eta_2 \colonequals d\zeta_2$ by the following formulas
\begin{equation}
\begin{aligned}
\eta_1 = (-3x^3 - 2b_1x^2 -b_2x + f_{21}^{C} x + f_{11}^{C})\frac{dx}{2y},\\
 \eta_2= (-x^2 + f_{22}^{C} x + f_{12}^{C})\frac{dx}{2y},
 \end{aligned}
\end{equation}
where 
\begin{equation}\label{eq:fijC}
\begin{aligned}
f_{11}^C &\colonequals 2b_1b_2 - b_1\alpha + b_1^2\beta + 3\delta - 3b_1\gamma + 3b_3, \qquad
&& f_{12}^C \colonequals b_2 - b_1\beta + 3\gamma, \\
 f_{21}^C &\colonequals b_2 + \alpha - b_1\beta
, \qquad && f_{22}^C \colonequals \beta.
\end{aligned}
\end{equation}
In particular, they are almost in the same form as $\eta_1^{(f)}$ and
$\eta_2^{(f)}$, except that we do not yet know if the constants $f_{21}^C$
and $f_{12}^C$ are equal. We will show in Theorem~\ref{thm:constants} that
$(f_{ij}^C)$ is indeed a symmetric matrix and that it is equal to the
matrix $f=(f_{ij})$ of Theorem~\ref{T:thetaneron}.

\begin{rem}
The subspace of $H_{\dR}^1(C/F_v)$ generated by $\eta_1$ and $\eta_2$ is comprised of precisely those differential forms whose formal integrals around $\infty$ have coefficients with bounded powers of $p$. 
\end{rem}

\subsection{Cartier--Manin and ordinarity}
\label{S:CartierManinOrdinarity}

We continue to assume that $p\geq 5$ and $C$ has semistable reduction at $v$, as in the previous subsection. Recall the quantity $H_1\in \O_F$ from \eqref{BlakestadHn}.
The goal of this subsection is to show:
\begin{prop}\label{H1ord}
The reduction of  $C$ at $v$ is semistable ordinary if and only if $H_1$ is invertible in $\O_v$.
\end{prop}

In order to prove Proposition~\ref{H1ord}, 
we write 
\[b(x)^{(p-1)/2} \bmod
\pi_v \equalscolon \sum_{k} c_k x^k\] 
and we define a matrix
\begin{equation}\label{HDef}
H = \left(\begin{smallmatrix} c_{p-1} & c_{2p-1}\\
c_{p-2} & c_{2p-2}\end{smallmatrix} \right)\in M_2(k_v)\,,
\end{equation}
where $k_v$ is the residue field of $\O_v$.

\begin{lem}\label{Hvsord}
The curve $C$ has semistable ordinary reduction at $v$ if and only if $\det(H)\neq 0$.
\end{lem}

\begin{proof}
If $C$ has good reduction, then the statement is due to Manin~\cite{Man62},
  see also~\cite[Theorem~3.1]{yui1978jacobian}. From now on we suppose that
  the reduction $\tilde{C}$ of $C$ modulo $\pi_v$ is singular. 
The equation of $C$ implies that $\tilde{C}$ is irreducible.
The matrix $H$ then represents the Cartier--Manin operator on $\tilde{C}$ with respect to the generating set $x^i\frac{dx}{2y}$ of the regular
  differentials. This follows from the description of the Cartier--Manin
  operator on  hyperelliptic curves due to Manin~\cite{Man62}, which extends to our
  setting via the work of St\"ohr on the Cartier--Manin operator of irreducible, but not necessarily smooth curves in~\cite[\S4]{stohr1993poles}.

 Define the Hasse--Witt invariant $\sigma(\tilde{C})$  as the rank of the $r^{\text{th}}$ power
  of this operator for any $r\ge2$ (this is independent of $r$. Hence the matrix $H$ is invertible if and only if $\sigma(\tilde{C})=2$.
 Let $\tilde{C}'$ denote the normalisation of $\tilde{C}$.
  By~\cite[Corollary~4.2]{stohr1993poles}, 
  $\sigma(\tilde{C})-\sigma(\tilde{C}')$ is the number of singular points
  in $\tilde{C'}(\overline{\F}_p)$, which we denote by $s$, and which
  is
equal to the difference between the geometric genera of $\tilde{C}$ and $\tilde{C}'$. 

Under our assumptions, $C$ has semistable ordinary reduction at $v$ if and only if $J$ has ordinary reduction at $v$, which means that the abelian variety $B$ is ordinary, where the connected component of the  N\'eron model of $J$ over $\O_v$ is an extension of $B$ by a torus. By the above, $\sigma(\tilde{C})=2$ if and only if one of the following holds:
\begin{itemize}
    \item $s=2$ and $\sigma(\tilde{C}')=0$, in which case $B=0$;
    \item $s=1$ and $\sigma(\tilde{C}')=1$, so that $B$ is an ordinary elliptic curve;
    \item $s=0$ and $\sigma(\tilde{C}')=2$, hence $B = \mathrm{Jac}(\tilde{C}')$ is an ordinary abelian surface. 
\end{itemize}
  These are precisely the cases that $C$ has semistable ordinary reduction at $v$.
\end{proof}
\begin{rem}
The result remains valid if $C$ is any hyperelliptic curve over $F_v$  with semistable reduction. 
In this case, $\tilde{C}$ may be reducible. Then $J$ is ordinary and, after applying a transformation, $\tilde{C}$ is of the form $y^2=h(x)^2$. One can show that this implies $\det(H)=1$.
\end{rem}

\begin{proof}[Proof of Proposition~\ref{H1ord}]
  In the appendix, we relate $H_1$ to the matrix $H$;
  Corollary~\ref{cor:Blakestad_vs_CartierManin} shows that $H_1$ is
  invertible if and only if $\det(H)$ is a unit. Hence the result
follows immediately from Lemma~\ref{Hvsord}.
\end{proof}

\subsection{The canonical subspace}
Assume that $p\geq 5$ and that $C$ has semistable ordinary reduction at
$v$. The goal of this subsection is to show that the differentials
$\eta_i^{(f)}$ arising from the normalised theta function $\theta_X$ are
equal to the differentials $\eta_i$ arising from Blakestad's zeta functions
discussed in \S\ref{subsec:zeta_fcts}. This yields an explicit formula for the symmetric matrix $(f_{ij})$, which, by Corollary~\ref{cor:theta_sigma2}, allows us to compute $\theta_X$ using the algorithms of \cite{Bia23}.

The main result is Theorem~\ref{thm:constants} below. It is an extension to semistable reduction of \cite[Theorem~36, Corollary~37]{blakestadsthesis}. We provide here a slightly different proof, but Blakestad's proof would also be applicable.

We first prove a stronger version of \cite[Lemma~26]{blakestadsthesis}.

\begin{lem}
\label{lemma:frac_implies_tens}
Let $t$ be a uniformiser at $\infty\in C$. Suppose there exists $\xi\in
  \mathrm{Frac}(\O_v[[t]])$ such that $\frac{d \xi}{dt} dt=
  \left(s\frac{dx}{2y} + r\frac{xdx}{2y}\right)(t)$ for some $r,s\in F_v$. Then $\xi$ is constant.
\end{lem}

\begin{proof}
Write $\xi = \frac{a(t)}{b(t)}$ with $a(t), b(t)\in\O_v[[t]]$, $b(t)\neq 0$. Assume that $a(t)$ is also non-zero. Then by the Weierstrass Preparation Theorem, we can write
\begin{equation*}
a(t) = c_a F_a(t) u_a(t),\qquad b(t) = c_b F_b(t) u_b(t),
\end{equation*}
for some non-zero $c_a,c_b\in \O_v$, $u_a(t), u_b(t)\in \O_v[[t]]$ unit power series and $F_a(t),F_b(t)$ distinguished polynomials. 

We first want to show that we must have $\xi\in F_v\otimes \O_v[[t]]$. If not, then $F_b(t)$ is non-trivial and $F_a(t)$ is not divisible by it. Thus $\xi$ has poles at some $t\in \overline{F}_v$ with $\ord_v(t)>0$; hence so does $\frac{d \xi}{dt}dt$,  contradicting the fact that $s\frac{dx}{2y} + r\frac{xdx}{2y}$ is holomorphic. 

After possibly multiplying by some $c\in F_v^{\times}$ to clear denominators, we may then apply \cite[Lemma~26]{blakestadsthesis}.
\end{proof}

\begin{thm}
\label{thm:constants}
Assume that $p\geq 5$, that and that $C$ has semistable ordinary reduction
  at $v$. Then the canonical Mazur--Tate splitting corresponds to the subspace $W_v$ spanned by the classes of 
\begin{equation*}
\begin{aligned}
\eta_1^{(f)} = \left(-3x^3 - 2b_1x^2 -b_2x + f_{12} x + f_{11}\right)\frac{dx}{2y},\\
 \eta_2^{(f)}= \left(-x^2 + f_{22} x + f_{12}\right)\frac{dx}{2y},
 \end{aligned}
\end{equation*}
where  $f_{ij} = f_{ij}^C = f_{ji}^C$ for all $i,j$ are the constants of~\eqref{eq:fijC}.

\end{thm}

\begin{proof}
Let $\zeta_1$ and $\zeta_2$ be as in \S\ref{subsec:zeta_fcts} and consider the function $\xi_i = (\zeta_i)_1+(\zeta_i)_2$ on $\{D,D\}\subset C^{(2)}$, where $D$ is the residue disc of the point at infinity.

By \cite[Lemma~35]{blakestadsthesis}\footnote{There is a typo in the formula for $D_2(\xi_1)$ there: it should be $D_2(\xi_1) = -\wp_{12}-\frac{1}{2}\wp_{2222}-b_1\wp_{22}+\alpha+b_2$.} 
or by the proof of \cite[Lemma~5.40]{Bia23}, we have, as functions on $\{D,D\}$,
\begin{align*}
D_j(2\xi_2) = -2X_{j2} + 2f_{j2}^{C} \ \ \ \text{and} \ \ \ D_j(2\xi_1-2b_1\xi_2+2X_{222}) = -2X_{1j} + 2f_{j1}^{C},
\end{align*}
where $X_{222}$ is the rational function on $J$ defined in \cite[(4.1), (1.4)]{Grant1990}.
Thus, restricting $\theta_X$ to $\{D,D\}\subset C^{(2)}$, we see by Proposition~\ref{prop:diff_eq_theta_X} that
\begin{align*}
D_j(D_2(\log(\theta_X)) - 2\xi_2) &= 2(f_{j2} - f_{j2}^C),\\
D_j(D_1(\log(\theta_X)) - 2(\xi_1-b_1\xi_2+X_{222})) &= 2(f_{j1} -
  f_{j1}^C)\,.
\end{align*}

  It suffices to show that $D_j(G)=0$ for $j\in \{1,2\}$ and $$G\in \{D_2(\log(\theta_X)) - 2\xi_2, D_1(\log(\theta_X)) -
  2(\xi_1-b_1\xi_2+X_{222})\}\,.$$ Since $T_1$ and $T_2$ have power series
  expansions in $t_1 = -\frac{x_1^2}{y_1},t_2  = -\frac{x_2^2}{y_2}$ with
  integral coefficients \cite[(4.6), (4.7)]{blakestadsthesis}, we have that
  $G\in \mathrm{Frac}(\O_v[[t_1,t_2]])$. By definition there exist $C_1,C_2\in F_v$ such that
\begin{align*}
dG &= C_1\left(\frac{dx_1}{2y_1}+\frac{dx_2}{2y_2}\right) + C_2\left(\frac{x_1dx_1}{2y_1}+\frac{x_2dx_2}{2y_2}\right) \\
&= (C_1 + x_1C_2)\frac{dx_1}{2y_1}+ (C_1 +x_2C_2)\frac{dx_2}{2y_2},
\end{align*}
so that $G = H(t_1)+ H(t_2)$ for some $H(t)\in F_v[[t]]$ satisfying
  $\frac{dH}{dt} dt = \left(C_1\frac{dx}{2y} +
  C_2\frac{xdx}{2y}\right)(t)$. Choosing $t_2^0 \in \pi_v\O_v\setminus
  \{0\}$, we have $H(t_1) = G(t_1, t_2^0) - H(t_2^0)\in
  \mathrm{Frac}(\O_v[[t_1]])$. Thus, $H$ is constant by Lemma
 ~\ref{lemma:frac_implies_tens} and hence $D_j(G) = 0$.
\end{proof}
\begin{cor}
$\alpha = 3\gamma$.
\end{cor}
\begin{proof}
    This follows from $f_{12}^C = f_{21}^C$.
\end{proof}

\appendix
\section{$H_1$ vs $H$}

In this appendix, we relate $H_1$ defined in \eqref{BlakestadHn} to the matrix $H$ defined in \eqref{HDef}; the main result is Corollary~\ref{cor:Blakestad_vs_CartierManin}. Fix a smooth projective genus~$2$ curve $C$ given by an equation of the form 
\[y^2 = b(x)= x^5 + b_1x^4 + b_2x^3 + b_3x^2 + b_4 x + b_5\]
over a field $K$ of characteristic $\ne 2$. Write $\omega_1 = \frac{dx}{2y}$, $\omega_2 = x\frac{dx}{2y}$ and $t = -\frac{x^2}{y}$.

\begin{prop}
\label{prop:t_expansions_diff}
The expansions of the
differentials $\omega_1$ and $\omega_2$
in terms of the parameter $t$ at $\infty\in C(K)$ are given by
\begin{align*}
\frac{\omega_1(t)}{dt} = \sum_{n\geq 1}\left(\sum_{\substack{k_1+2k_2+3k_3+4k_4+5k_5 + 1= n\\
k_0 + \cdots+ k_5 = n}}\binom{n}{k_0,\dots,k_5}\prod_{i=1}^5 b_i^{k_i}\right)t^{2n},\\
\frac{\omega_2(t)}{dt} = 1 + \sum_{n\geq 1}\left(\sum_{\substack{k_1+2k_2+3k_3+4k_4+5k_5 = n\\
k_0 + \cdots +k_5 = n}}  \binom{n}{k_0,\dots,k_5}\prod_{i=1}^5 b_i^{k_i}\right)t^{2n}.
\end{align*}
\end{prop}

\begin{proof}
It suffices to adapt the proof of \cite[Theorem~1]{Yasuda} to our setting.
  As in loc.\ cit., the result is proved for a curve over
  $\Z[\frac{1}{2},b_1,\ldots,b_5,\Delta(b_1,\ldots,b_5)^{-1}]$, where
  $$\Delta(b_1,\ldots,b_5)=\mathrm{disc}(x^5+b_1x^4+b_2x^3 + b_3x^2 +b_4 x
  +b_5)\,.$$ We may work over $\R$ by using an injective ring homomorphism 
\[\Z\left[\frac{1}{2},b_1,\ldots,b_5,\Delta(b_1,\ldots,b_5)^{-1}\right]\hookrightarrow \C\]
such that all $b_i$ are sent to real numbers.
  In order to make the comparison with Yasuda's proof as smooth as possible, we adopt similar notation. It is enough to show equality for $t\in\R$ with $|t|$ sufficiently small; we start with proving the statement on the expansion of $\omega_2$.

Let $u = \frac{1}{x}$. Then $(u,t)$ satisfies $\tilde{f}(u,t) = 0$, where
\begin{equation*}
\tilde{f}(u,t) =   1-\frac{t^2}{u} - b_1t^2 -b_2t^2u  -b_3t^2u^2 - b_4t^2u^3 -b_5t^2u^4.
\end{equation*}
Now fix $t_0\in \R^{\times}$ and consider the following holomorphic function in the variable $u\in \C$:
\begin{equation*}
f(u) = u \cdot \tilde{f}(u,t_0);
\end{equation*}
by Rouch\'e's theorem, $f(u)$ and $u$ have the same number of zeros on $|u|<1$, provided that $|t_0|$ is sufficiently small. Moreover, applying the intermediate value theorem to the restriction of $f$ to the real interval $[0,1]$, we see that the unique zero $u_0$ of $f(u)$ in the open disc $|u|<1$ is a positive real number. 
This gives a point $(x_0,y_0) = (u_0^{-1}, -t_0^{-1}u_0^{-2})\in C(\R)$. Note that in verifying that the assumptions of Rouch\'e's theorem apply, we also see that $f(u)$ has no roots on $|u| = 1$.

Jensen's formula yields
\begin{equation}
\label{eq:logu0}
  \log u_0 = 2\log t_0 - \frac{1}{2\pi} \mathrm{Re}\int_{0}^{2\pi}\log(1-g(e^{i\theta}))d\theta, 
\end{equation}
where $g(u) = 1-\tilde{f}(u,t_0) $. 

Now,
\begin{align*}
    \frac{1}{2\pi} \int_{0}^{2\pi}\log(1-g(e^{i\theta}))d\theta = -\sum_{n\geq 1} \frac{1}{2\pi n}\int_0^{2\pi}g(e^{i\theta})^n d\theta \\
    = -\sum_{n\geq 1}\frac{t_0^{2n}}{n} \sum_{\substack{k_1+2k_2+3k_3+4k_4+5k_5 = n\\
    k_0+\cdots +k_5 = n}}\binom{n}{k_0,\dots,k_5}\prod_{i=1}^5 b_i^{k_i}\,.
\end{align*}
Substituting into \eqref{eq:logu0} and using that $\omega_2 = \frac{t}{2}\frac{du}{u}$, we obtain the series expansion of $\omega_2$ stated in the theorem.  

Consider now the function $u\frac{f^{\prime}(u)}{f(u)}$. By the above, on $|u|\leq 1$, this has a pole only at $u=u_0$, with residue $u_0 = x_0^{-1}$. Thus, by the residue theorem, we have
\begin{equation}
\label{eq:res_thm}
    x_0^{-1} = \frac{1}{2\pi i} \int_{|u| = 1}  \frac{uf^{\prime}(u)}{f(u)}du.
\end{equation}
On the other hand, 
\begin{equation*}
\frac{u f^{\prime}(u)}{f(u)} = \biggl(1 - t_0^2\sum_{i=1}^5 ib_iu^{i-1}\biggr) \sum_{n\geq 1}g(u)^n.    
\end{equation*}
Thus, 
\begin{equation}
\label{eq:integral}
    \frac{1}{2\pi i} \int_{|u| = 1}  \frac{uf^{\prime}(u)}{f(u)}du =  A_{-1}-t_0^2\sum_{i=1}^5 i b_i A_{-i},
\end{equation}
where $A_{-\ell}$ is the coefficient of $u^{-\ell}$ in $\sum_{n\geq 1}g(u)^n$. In particular, we have
\begin{align*}
    A_{-\ell} = \sum_{k_1,\dots, k_5\geq 0 }\binom{ \sum_{j=1}^5 jk_j + \ell}{ \sum_{j=2}^5(j-1)k_j+ \ell, k_1,\dots, k_5}t_0^{2(\sum_{j=1}^5 jk_j + \ell)}\prod_{j=1}^5 b_j^{k_j}
    \end{align*}
and hence
\begin{align*}
    t_0^2ib_iA_{-i}= \sum_{k_{1}, \dots, k_5\geq 0}\frac{i
    k_{i}}{\sum_{j=1}^5 jk_j +1} \binom{\sum_{j=1}^5 jk_j +
    1}{\sum_{j=2}^5(j-1)k_j +1,k_1,\dots, k_5}t_0^{2(\sum_{j=1}^5 jk_j
    +1)}\prod_{j=1}^5 b_j^{k_j}\,.
\end{align*}
Substituting back into \eqref{eq:integral} and \eqref{eq:res_thm}, we find
\begin{equation*}
x_0^{-1} = \sum_{k_1,\dots, k_5 \geq 0}\frac{1}{\sum_{j=1}^5 jk_j +1} \binom{\sum_{j=1}^5 jk_j + 1}{\sum_{j=2}^5(j-1)k_j +1,k_1,\dots, k_5}t_0^{2(\sum_{j=1}^5 jk_j +1)}\prod_{j=1}^5 b_j^{k_j}.
\end{equation*}
Finally, observing that $\frac{\omega_1}{dt} = \frac{t}{2}\frac{d}{dt}(x^{-1})$, we obtain the claimed expansion of $\omega_1$.
\end{proof}

From now on, we specialise to the case in Section~\ref{S:com}: $K = F$ is a number field, $b_1,\dots,b_5\in \O_F$, $v$ is a finite prime of $F$ above $p$. Moreover, we assume that $p\geq 5$ and $C/F$ has semistable reduction at $v$. Write, as in \S\ref{S:CartierManinOrdinarity},
\[b(x)^{(p-1)/2} \bmod \pi_v = \sum_{k} c_k x^k.\] 

\begin{lem}
\label{lem:aidi_in_terms_ci}
Let $a_i$ and $d_i$ be the coefficient of $t^i$ in the expansion of $\frac{\omega_1(t)}{dt}$ and $\frac{\omega_2(t)}{dt}$, respectively. We then have 
\begin{align*}
c_{2p-1} = a_{p-1}, \qquad &c_{p-1} \equiv a_{3p-1} - b_1
  a_{p-1}\bmod{\pi_v}, \\
c_{2p-2} = d_{p-1}, \qquad &c_{p-2}\equiv d_{3p-1}- b_1 d_{p-1}\bmod{\pi_v}.
\end{align*}

\end{lem}

\begin{proof}
  The equalities follow immediately from
  Proposition~\ref{prop:t_expansions_diff}.
  The other two expressions are almost identical and we will focus only on
  the former. 
  By Proposition~\ref{prop:t_expansions_diff}, we have
\[a_{3p-1} = \sum_{\substack{k_1+2k_2+3k_3+4k_4+5k_5+1 = (3p-1)/2\\
k_0 + \cdots+ k_5 = (3p-1)/2}}\binom{(3p-1)/2}{k_0,\dots,k_5}\prod_{i=1}^5 b_i^{k_i}.\]
Since $p < \frac{3p-1}{2} < 2p$, there is precisely one factor in $\frac{3p-1}{2}!$ that is divisible by $p$. 

Now let $k_0,\dots,k_5$ be nonnegative integers such that
  $k_1+2k_2+3k_3+4k_4+5k_5+1 = (3p-1)/2$ and $k_0 + \cdots + k_5 =
  (3p-1)/2$. If each $k_i$ is less than $p$, then the corresponding term is
  $0$ modulo $\pi_v$. Otherwise, there are two possibilities: 
\begin{enumerate}
    \item\label{c_p-1} $k_0\geq p$ and $k_1,k_2,k_3,k_4,k_5<p$, or 
    \item\label{b_1a_p-1} $k_1\geq p$ and $k_0,k_2,k_3,k_4,k_5<p$.
\end{enumerate}
One can easily check that the sum over tuples as in \eqref{c_p-1} (resp.
  \eqref{b_1a_p-1}) is nothing but $c_{p-1}$ (resp. $b_1a_{p-1}$) modulo $\pi_v$. The result follows.
\end{proof}

\begin{cor}
\label{cor:Blakestad_vs_CartierManin}
  The matrix $H$ is invertible if and only if
  $H_1\not\equiv 0\bmod{\pi_v}$.
\end{cor}

\begin{proof}
Recall that $H_1 = A_1D_1 - B_1C_1$ for some $A_1,B_1,C_1,D_1\in \O_F$ that appear in the $t$-expansions of the functions $\phi_1$ and $\psi_1$; see \eqref{BlakestadPhinPsin}. The meromorphic differentials $\phi_1\omega_1, \phi_1\omega_2, \psi_1\omega_1$ and $\psi_1\omega_2$ are regular away from $\infty$, and it is easy to check that
\begin{align*}
& \Res_\infty(\phi_1\omega_1) = a_{3p-1} + A_1, & &\Res_\infty(\phi_1\omega_2) = d_{3p-1} + b_1A_1 + B_1, \\
& \Res_\infty(\psi_1\omega_1) = a_{p-1} + C_1, & &\Res_\infty(\psi_1\omega_2) = d_{p-1} + b_1C_1 + D_1,
\end{align*}
where $a_i$ and $d_i$ are as in Lemma~\ref{lem:aidi_in_terms_ci}. Therefore,
\begin{align*}
& A_1 = -a_{3p-1}, & & B_1 = -b_1A_1-d_{3p-1}, \\
& C_1  = -a_{p-1}, & & D_1 = -d_{p-1}-b_1C_1.
\end{align*}
By Lemma~\ref{lem:aidi_in_terms_ci}, we then conclude that $H_1 \equiv
  \det(H) \bmod{\pi_v}$. 
\end{proof}

\bibliographystyle{amsalpha}
\bibliography{master}

\end{document}